\documentclass[11pt]{article}
\usepackage{amsmath,amsfonts,amssymb,graphicx,amsthm,color,natbib,dsfont,array,booktabs,multirow,authblk}
\usepackage[colorlinks=true, linkcolor=black, citecolor=black, urlcolor=black]{hyperref}
\usepackage[normalem]{ulem}
\usepackage{geometry}

\allowdisplaybreaks

\newcommand{\E}{\operatorname{E}}
\newcommand{\cov}{\operatorname{Cov}}
\newcommand{\var}{\operatorname{Var}}

\newtheorem{theorem}{Theorem}[]
\newtheorem{lemma}[]{Lemma}
\newtheorem{proposition}[]{Proposition}
\newtheorem{corollary}[]{Corollary}
\newtheorem{definition}[]{Definition}

\theoremstyle{definition}
\newtheorem{remark}[]{Remark}

\newtheorem{example}[]{Example}

\begin{document}

\author[1]{{Fadoua} {Balabdaoui}}
\author[1,2]{{Antonio} {Di Noia}}
\affil[1]{Seminar for Statistics, Department of Mathematics, ETH Zurich}
\affil[2]{Faculty of Economics, Euler Institute, Università della Svizzera italiana}
\renewcommand\Affilfont{\itshape\small}

\title{Asymptotic theory for nonparametric testing of $k$-monotonicity in discrete distributions}
\date{\today}
\maketitle
\begin{abstract}
In shape-constrained nonparametric inference, it is often necessary to perform preliminary tests to verify whether a probability mass function (p.m.f.)\ satisfies 
qualitative constraints such as monotonicity, convexity, or in general $k$-monotonicity. 
In this paper, we are interested in nonparametric testing of $k$-monotonicity of a finitely supported discrete distribution.
We consider a unified testing framework based on a natural statistic which is directly derived from the very definition of $k$-monotonicity.
The introduced framework allows us to design a new consistent method to select the unknown knot points that are required to consistently approximate the limit distribution of several test statistics based either on the empirical measure or the shape-constrained estimators of the p.m.f.
We show that the resulting tests are asymptotically valid and consistent for any fixed alternative. 
Additionally, for the test based solely on the empirical measure, we study the asymptotic power under contiguous alternatives and derive a quantitative separation result that provides sufficient conditions to achieve a given power.  
We employ this test to design an estimator for the largest parameter $k  \in \mathbb N_0$ such that the p.m.f.\ is $j$-monotone for all $j = 0, \ldots, k$, and show that the estimator is different from the true parameter with probability which is asymptotically smaller than the nominal level of the test. Finally, we conduct an extensive simulation study to validate the theory, assess the finite-sample performance of the proposed methods, and illustrate them on several real datasets.
\end{abstract}


\section{Introduction}
\label{sec:intro}
\subsection{Background}
Modelling count data is an important task in many statistical problems. Several parametric families of discrete distributions have been proposed in the literature with their associated inference and goodness-of-fit procedures. For real-world applications, there is often a great interest in modelling the distribution of count data by assuming that the underlying probability mass function (p.m.f.)\ satisfies some shape constraint such as monotonicity, convexity, log-concavity, unimodality, $k$-monotonicity including complete monotonicity ($k = \infty$); see e.g.\  \cite{balabdaoui2013}, \cite{DUROT2013282}, \cite{chee2016}, \cite{balabdaoui2016}, \cite{balabdaoui2020completely}.
Among those constraints, $k$-monotonicity has received attention because of its generality, as it comprises monotonicity ($k=1$) and convexity ($k=2$), and its wide applicability in many real-world problems. In ecology, for instance, $k$-monotonicity naturally arises in the problem of species richness, where the goal is to estimate the total number of species in a population (based on the observed ones). This is done by assuming that the distribution of species abundances is $k$-monotone, as it allows for identifying and estimating the total number of species. We refer to the work of   \cite{durot2015nonparametric} where the convex model was used, \cite{giguelay2017estimation}, \cite{chee2016} and \cite{balabdaoui2020completely} for the $k$-monotone and $\infty$-monotone models.

Several papers have developed pointwise and global asymptotic theory for nonparametric estimators of discrete distributions under such shape constraints, see e.g.\  \cite{balabdaoui2013}, \cite{balabdaoui2016}, \cite{balabdaoui2017}, \cite{jankowski2009}, \cite{giguelay2017} and \cite{balabdaoui2020completely}. 
However, it appears that less attention has been paid to proposing testing procedures for these qualitative constraints. This oversight often stems from the complexity involved in deriving distributional results for the test statistics.

The current work is motivated by the attempts to construct tests for convexity and more generally $k$-monotonicity already made in   \cite{BALABDAOUI20188} and \cite{GIGUELAY201896}. Our main goal is to propose a testing procedure that is less conservative than the one proposed in \cite{GIGUELAY201896}, straightforward to implement, easy to explain to a practitioner, and for which theoretical guarantees are established.  

\subsection{Contributions}
In this paper, we propose a unified framework for testing the $k$-monotonicity of a discrete distribution based on a random sample. As will become clear in the next sections, our testing procedure is based on the derivation of the asymptotic distribution of a very simple statistic. This statistic, whose sign determines whether $k$-monotonicity of the underlying distribution holds or not, is similar to the one used in  \cite{GIGUELAY201896}. Despite this similarity,  we take a very different route, which yields a less conservative and more powerful test.  

It is worth noting that we shall restrict attention to discrete distributions that are supported on a finite subset of integer numbers, however, we do not assume that the support is known.

Such a requirement is rather standard in the related literature as it is not very restrictive. In fact, in this paper and all the references mentioned above, no upper bound is imposed on the cardinality of the support. Additionally, the same requirement allows us to derive asymptotic theoretical guarantees which do not depend on the underlying distribution; i.e., which are distribution-free.

The contributions of this paper are as follows:

\begin{enumerate}

\item  We characterise the limit distribution of the test statistic, after centring and rescaling, in all settings and not only under $k$-monotonicity. Thus, our approach to the testing problem is much more general than the one adopted in \cite{GIGUELAY201896}. In that paper, the authors use stochastic ordering under $k$-monotonicity to determine the rejection region for their test. Resorting to this stochastic ordering is the very reason for which the test proposed in \cite{GIGUELAY201896} is conservative. For more details, we refer the reader to Section 3.5.   
 
\item  We design a novel data-driven approach to select the knot points of the underlying p.m.f.\ (see Definition \ref{def:k-mon}). Since the limit distribution of the statistic depends on these special points, our selection procedure is crucial for constructing asymptotically valid tests. One of our main results is to show that, under $k$-monotonicity, the probability that the set of selected points and the set of the true knots do not agree is of order $1/\sqrt n$, where $n$ is the sample size.

\item We show that our test is asymptotically valid and it is consistent under any fixed alternative, and, by construction, it is less conservative than the one proposed by \cite{GIGUELAY201896}. Achieving the latter is possible since the selection procedure allows us to consistently estimate the distribution of the test statistic under the null hypothesis.

\item We derive quantitative separation results which give a lower bound on $n$ as well as the extent of non-$k$-monotonicity of the true p.m.f.\ so that our test has power at least $1-\beta$ for any given $\beta \in (0,1)$. Sufficient conditions can be derived from the first ones in such a way that they only involve $k$, the cardinality of the support, $\beta$, and the targeted level $\alpha$.  Additionally, we study the behaviour of our test under contiguity and show that it is asymptotically unbiased.

\item We construct alternative tests for monotonicity ($k=1$) and convexity ($k=2$) constraints, which are based on known shape-constrained estimators (see \citep{jankowski2009}, and \citep{DUROT2013282}), and also on our knot selection procedure. We show that these further tests are asymptotically valid and consistent.

\item We propose a method to estimate the \lq\lq strongest\rq\rq \  $k$-monotonicity parameter $k_0$ defined as the largest $k\in \mathbb{N}_0$ such that the true p.m.f.\ is $j$-monotone for all $j = 0, \ldots, k$.  We show that the probability that the obtained estimator agrees with $k_0$ is asymptotically at least $1-\alpha$.  

\end{enumerate}

We would like to emphasise here that one difference between our paper and that of \cite{GIGUELAY201896} is the definition we adopt for $k$-monotonicity. In fact, while the inequality constraints in \cite{giguelay2017} and \cite{GIGUELAY201896} are imposed on all non-negative integers (even in the case where the true p.m.f.\ is finitely supported), we only assume here that the constraints hold for 
$j \in \{m, \ldots, M - k \}$, where $m$ and $M$ are respectively the minimum and maximum of the support of the true p.m.f.  For this reason, $k$-monotonicity in the sense used in \cite{giguelay2017} and \cite{GIGUELAY201896} implies the one we consider here.   Note that they are equivalent in the case of monotonicity; i.e., $k=1$, because any monotone p.m.f.\ takes value $0$ on the right of its support and hence continues to be non-increasing. An example that we construct below in Section \ref{sec:asymp} illustrates the existing difference between our definition of convexity and the one considered in \cite{giguelay2017}.  We will return to this aspect after Definition \ref{def:k-mon} below.

\subsection{Organisation of the paper}
The present paper is organised as follows: In Section \ref{sec:asymp} we introduce $k$-monotonicity sample statistics and investigate their asymptotic behaviour. In Section \ref{sec:tests} we construct a test for $k$-monotonicity and discuss its calibration, which is ensured by employing a general data-driven procedure which allows us to consistently select the unknown knots of the true p.m.f.\ In the same section, we present an estimation method to select the largest $k$-monotonicity parameter and show that the probability that the resulting estimator and the true parameter are different is smaller than the nominal level of the test.  Section \ref{sec:mono-conv} is devoted to the important constraints of monotonicity and convexity. Besides the tests constructed in Section \ref{sec:asymp}, we present a second testing approach based on the $\ell_2$-distance between the raw empirical estimator and a shape-constrained version thereof. In Section \ref{sec:num} we show the results of an extensive Monte Carlo simulation study that we conducted to validate the theoretical results and assess the finite sample performance of the proposed tests by comparing them with other testing procedures. To highlight the practical use of the proposed methods, we implemented them on several real datasets. Section \ref{sec:proofs} presents the proofs of the theoretical results. In the Appendix, we provide some additional experiments, derivations, and details regarding why bootstrap resampling cannot be used to approximate the limiting distribution of the test statistic based on the empirical measure.

\section{Asymptotics for \texorpdfstring{$k$}{k}-monotonicity sample statistics}
\label{sec:asymp}
In the sequel, we consider a probability space $(\Omega, \mathcal{F},P)$ on which a random variable (r.v.) $X:\Omega \to \mathcal{S}\subset \mathbb{Z}$ is defined such that $|\mathcal{S} | < \infty$; i.e., $X$ is supported on a finite subset of the integers. Let $p$ denote the probability mass function (p.m.f.) associated with the distribution of $X$.  We start this section by recalling the definition of $k$-monotonicity. 

\begin{definition}
\label{def:k-mon}
        Let $k\in \mathbb{N}_0 $ and $\mathcal{S}=\{m,\dots,M\}$ for some (not necessarily known) integers  $m$ and $M$ such that $m \le M - k$. The p.m.f.\ $p$ associated with the distribution of a r.v.\ $X$ supported on $\mathcal{S}$ is said to be $k$-monotone if
        \begin{align}\label{eq:kmonIneq}
            \nabla^k p(j)=(-1)^k\Delta^kp(j)\geq 0
        \end{align}
        for all $j \in \mathcal{S}_k := \{m,\dots,M-k\} $, where $\Delta^k$ is the $k$-th degree discrete Laplacian defined recursively as
\begin{align*}
\Delta^0 p(j) &= p(j), \\
\Delta^1 p(j) &= p(j+1) - p(j), \\
\Delta^r p(j) &= \Delta^1[\Delta^{r-1} p(j)] , \quad \text{for } r \ge 2.
\end{align*}
Moreover, if $\nabla^kp(j)>0$, then $j$ is called a $k$-knot point or simply a knot of $p$.
\end{definition}
Note that $\nabla^0$ is the identity operator, thus, $0$-monotonicity simply means non-negativity of $p$, a property which is always satisfied. For $k \in \mathbb{N}$,  $\nabla^k$ can be seen as a discrete forward difference operator that captures certain geometric properties of a given sequence of numbers. For instance, $\nabla^1$ returns the right-hand slopes and $\nabla^2$ the right-hand curvatures.
As mentioned in Section \ref{sec:intro}, our definition of $k$-monotonicity is different from the one used in \cite{giguelay2017} and \cite{GIGUELAY201896}. The main difference is that we restrict attention to the support of the underlying distribution and hence require the constraint in \eqref{eq:kmonIneq} to be satisfied only for the integers in $\mathcal{S}_k$. \cite{giguelay2017} uses the same definition as in \cite{lefevre2013multiply} where the inequalities in \eqref{eq:kmonIneq} have to be satisfied for all $j \in \mathbb N_0$. This means that these inequalities have to hold beyond the set $\mathcal{S}_k$.  In this sense, the definition used in \cite{giguelay2017} and \cite{GIGUELAY201896} is more restrictive. On the other hand, some interesting properties can be shown in that case, e.g.\ $k$-monotonicity implies strict $l$-monotonicity for $1\leq l\leq k-1$, where the strictness refers to the inequalities in \eqref{eq:kmonIneq};  see \cite{giguelay2017}. In the following example, we again stress the fact that our focus in this work is on finitely supported distributions. For this reason, we are mainly interested in the \lq\lq local\rq\rq \  shape of the corresponding p.m.f.\ on its support and not in the \lq\lq global\rq\rq \ one.
\begin{example}
    To give a concrete example, consider the p.m.f.\ $p$  defined on the set $\{0,1,2,3\}$ as follows: 
    $$p(0)= 1/3,\quad  p(1) = p(2)=1/6, \quad p(3)=1/3,  \quad p(j) =0 \quad \forall \,\, j > 3.$$  Then, with $k=2$ we have that $\mathcal{S}_2 = \{0,1 \}$, $$\nabla^2 p(0) = p(0) + p(2) -2 p(1) =  1/6 > 0,$$ and $$\nabla^2 p(1) = p(1) + p(3) -2 p(2)  = 1/6 > 0.$$  This means that $p$ is convex according to Definition \ref{def:k-mon}.  However, the same p.m.f.\ is not convex in the sense of \cite{giguelay2017}.  This can be seen in a simple sketch of the p.m.f.\, or using a formal calculation. For instance, 
    $$\nabla^2 p(2) =p(2)+ p(4)  -  2p(3) =  1/6 +0- 2/3 = -1/2 < 0.$$  Note that a convex p.m.f.\ in the sense of \cite{giguelay2017} has to be strictly decreasing on $\mathbb N_0$, a property which is clearly violated in this example.
\end{example}
From Definition \ref{def:k-mon} it follows that $k$-monotonicity of a p.m.f. $p$ is equivalent to the condition
\begin{align*}
    \rho_k:=\min_{j \in \mathcal{S}_k}\nabla^k p(j) \geq 0.
\end{align*}
To describe our statistical procedure to test whether a sample is drawn from a $k$-monotone $p$, let us consider $X_1,\dots,X_n$ independent copies of $X$, for which we adopt the compact notation $X_{1:n}$ and $x_{1:n}$ for its realisation.  Define the empirical p.m.f.\ as 
$$\widehat{p}_n(j) = \frac{1}{n}\sum_{i=1}^n \mathbf{1}_{\{X_i=j\}}.$$
In the sequel, we shall focus on studying the properties of the estimator 
\begin{align*}
    \widehat {\rho}_{k,n}= \min_{j \in \widehat{\mathcal  {S}}_{k,n}}\nabla^k \widehat p_n(j)
\end{align*}
where $\widehat{\mathcal  {S}}_{k,n} = \{X_{(1)},\dots,X_{(n)}-k\}$ is obtained through the extreme order statistics $X_{(1)}= \min \{X_{1:n}\}$, and $X_{(n)}= \max \{X_{1:n}\}$. Let us consider the following centred and re-scaled version of $\widehat {\rho}_{k,n}$:
$$
T_{k,n}:= \sqrt{n} (\widehat {\rho}_{k,n}- \rho_k).
$$
For our testing purposes, we will later propose to use a version of the statistic $T_{k,n}$ 
as it is straightforward to compute and completely depends on the empirical measure. Our proposal comes with the advantage that the asymptotic behaviour of $T_{k,n}$ can be established. For the next results, let us introduce the set
\begin{align}\label{eq:Ih}
I_{h}^{(k)} = \{ j \in \mathcal S_k:  \nabla^kp(j) = h \}
\end{align}
for a given $h \in \mathbb R$. In the following, we adopt a simplified notation by omitting the superscript, thus, we use $I_{h}$ to implicitly indicate $ I_{h}^{(k)}$.

\begin{theorem}
\label{thm:main}
As $n\to \infty$, 
$$T_{k,n}\overset{d}{\to} W_{I_{\rho_k}}=\min_{j\in I_{\rho_k}} Z_j,$$
where $(Z_j)_{j \in \mathcal S_k}$ is a random vector with law $\mathcal N(0,\Sigma_k)$ such that
\begin{align}
\label{eq:Sigma_k}
(\Sigma_k)_{r,s}= \cov[\nabla^k\mathbf{1}_{\{X_1=m+r-1\}},\nabla^k\mathbf{1}_{\{X_1=m+s-1\}}]
\end{align}
for $r, s \in \{1,\ldots, M-k-m +1 \}$. In particular, if $\rho_k=0$ then $p$ is $k$-monotone and $I_0=\mathcal{S}_k \setminus J_0$ where $J_0$ is the set of knots of $p$; i.e.\ $J_0 =  \{j \in \mathcal{S}_k: \nabla^k p(j) > 0 \}$.
\end{theorem}

Note that Theorem \ref{thm:main} can be proved either using the asymptotic properties of the empirical p.m.f.\ or the generalised Delta-method for quasi-differentiable functions (\citealp{marcheselli2000generalized}). In the Appendix, we provide an explicit expression of the asymptotic covariance $\Sigma_k$ for $k=1,2$.

\section{Asymptotic tests for \texorpdfstring{$k$}{k}-monotonicity}
\label{sec:tests}
\subsection{Test statistic and calibration}
In this section, we present a class of asymptotic tests which asymptotically achieve a given nominal level of significance $\alpha\in(0,1)$ and are consistent for testing $k$-monotonicity for any given integer $k \ge 1$. Now, testing whether a finitely supported p.m.f.\ is $k$-monotone amounts to considering the hypotheses
\begin{align}
\label{eq:hp}
H_0^{(k)}: \rho_k  \ge 0 \ \ \mathrm{against} \ \  H_1^{(k)}: \rho_k < 0.
\end{align}
Let us introduce the test statistic
\begin{align*}
\widehat{T}_{k,n} := \sqrt{n} \widehat \rho_{k,n}.
\end{align*}
Moreover, let $(\widehat Z_{j,n})_{j\in \widehat{  \mathcal{S}}_{k,n}}\sim \mathcal{N}(0,\widehat \Sigma_{k,n})$ where $\widehat \Sigma_{k,n}$ is the empirical estimator of $\Sigma_k$ obtained via the empirical counterpart of \eqref{eq:Sigma_k}, and let $\widehat W_{I_h,n}=\min_{j\in I_h} \widehat Z_{j,n}$ where, without loss of generality,   $I_h\subseteq \widehat{ \mathcal{S}}_{k,n}$ because $\widehat S_{k,n}=\mathcal{S}_k$ almost surely for $n$ large enough.

For a given significance level $\alpha \in (0,1)$, we define the following testing critical thresholds
$$
\widehat t_{\alpha,n}(I_h) = \mathbf{1}_{\{I_h\neq\emptyset\}}  \ F^{-1}_{\widehat W_{I_h,n}}(\alpha),
$$
where $F_{Y}$ denotes the distribution function of a r.v.\ $Y$.
In the next preparatory proposition, we show that the test which rejects $H_0^{(k)}$ if and only if $\widehat T_{n, k} < \widehat t_{\alpha,n}(I_0)$ is asymptotically valid and consistent for any fixed alternative. 

\begin{proposition}
\label{prop:test-properties}
    Let  $p$ be the true p.m.f.\ and define  
    \begin{align*}
     \mathcal{H}_k : =  \{h \in \mathbb R:  \exists  \ j \in \mathcal S_k \ \text{such that} \ \nabla^k p(j) = h \}.
    \end{align*}
     For $h \in \mathcal H_k$, let $I_h$ be the set defined in \eqref{eq:Ih}. 
    \begin{enumerate}
    \item  If $p$ satisfies $H_0^{(k)}$ and $\rho_k=0$ then
    $ \lim_{n \to \infty} P( \widehat T _{k,n} < \widehat t_{\alpha,n}(I_0))= \alpha $.

    \item  If $p$ satisfies $H_0^{(k)}$ and $\rho_k>0$  then
    $ \lim_{n \to \infty} P( \widehat T _{k,n} < \widehat t_{\alpha,n}(I_0))=0 $.
    
    \item  If $p$ satisfies $H_1^{(k)}$  then for any integer $r \in \{1, \ldots,\vert \mathcal{H}_k \vert \} $ and distinct values $h_1, \ldots, h_r \in \mathcal H_k$ 
    $$ \lim_{n \to \infty} P( \widehat T _{k,n} < \widehat t_{\alpha,n}(\cup_{j=1}^r I_{h_j} )) =1. $$
    \end{enumerate}
\end{proposition}

Proposition \ref{prop:test-properties} is a consequence of the asymptotic result of Theorem \ref{thm:main}, the definition of $I_0$ and the consistency of $\widehat \Sigma_{k,n}$.  Note that the boundary case $\rho_k=0$ is the least favourable one under $H_0^{(k)}$ when using the test statistic $\widehat{T}_{k,n}$ for the testing problem \ref{eq:hp}. In this regard, it is well justified that the calibration is done for this case. Note that geometrically the condition $\rho_k =0$ means that the true p.m.f.\ $p$ has regions where $k$-monotonicity is not strict. In the monotone case, for example, $\rho_1 = 0$ if and only if $p$ has flat regions, that is, regions where $p$ is constant. We also remark that the slightly more general formulation given for point 3 of Proposition \ref{prop:test-properties} is used later when we replace the unknown $I_0$ with a suitable estimator. In the next subsection, we show that it is possible to design a set estimator which is consistent under $H_0^{(k)}$, but larger under $H_1^{(k)}$. Thus, point 3 of Proposition \ref{prop:test-properties} ensures that consistency under any fixed alternative is preserved.

\subsection{Data-driven selection of (non)-knots}
Proposition \ref{prop:test-properties} is a useful theoretical result, but cannot be implemented in practice since the quantile $\widehat t_{\alpha,n}(I_0)$ depends on the unknown set $I_0$. In the following, we propose an estimator of this set which is consistent in the sense that the probability that it is different from $I_0$ tends to $0$ under $H_0^{(k)}$. We will show a much stronger result by proving that the convergence happens at the $n^{-1/2}$-rate.  Later, we shall compare this estimation approach with two other alternative methods that we introduce in Section \ref{sec:num}.  Our estimation approach for the set $I_0$, which we will now describe, leads to the rejection region 
\begin{align}
\label{eq:generic-region}
    C_{k, n}(\alpha) = \{ x_{1:n}\in \mathcal{S}^n: \widehat{T}_{k,n}(x_{1:n}) <  \widehat {t}_{\alpha,n}(\widehat I_{0,n} ) \},
\end{align}
and is based on a data-driven selection of the elements of $I_0$. The selection procedure tests whether for a given integer $j\in \mathcal{S}_k$ the equality $\nabla^k p(j) = 0$ holds.  Under this equality and using the Central Limit Theorem, we have that
\begin{align}
\label{eq:convDist}
 \frac{\sqrt{n} \nabla^k \widehat p_n(j)}{\sqrt{(\Sigma_k)_{j-m+1,j-m+1}}} \overset{d}{\to}  \mathcal{N}(0,1).
\end{align}
Thus, this selection approach consists of conducting simultaneous tests which aim at picking all of the indices $j \in \widehat{\mathcal  {S}}_{k,n}$ that are \lq\lq compatible\rq\rq \ with the weak convergence stated in \eqref{eq:convDist}. Formally, we set
\begin{align}
\label{eq:datadriven}
    \widehat{I}_{0, n}:= \Big\{j \in \widehat{\mathcal{S}}_{k,n} :  \frac{\sqrt{n} \nabla^k \widehat p_n(j)} { \sqrt {(\widehat{\Sigma}_{k,n})_{j-m+1,j-m+1}} }  \leq z_ {1-1/n}\Big\}
\end{align}
 where $z_{\beta}$ is the $\beta$-order quantile of the standard Gaussian.
The following theorem gives the limit set of $\widehat I_{0,n}$ in all relevant cases. In particular, it implies that, under $H_0^{(k)}$, $\widehat{I}_{0,n}$ converges to the true set of the non-knot points $I_0$  at the $n^{-1/2}$-rate.

\begin{theorem}
\label{thm:selection}
    Let $\widehat I_{0,n}$ be the set defined in \ref{eq:datadriven}.  It follows
     \begin{align*}
      P(\widehat I_{0,n}  \ne {I}^* ) = O(1/\sqrt n)  
     \end{align*}
   where:  
    \begin{enumerate}
        \item If $\rho_k=0$ then ${I}^*  = I_0 $ .
        \item If  $\rho_k > 0$ then ${I}^*  = \emptyset $.
        \item If $\rho_k < 0$ then ${I}^*  = I_- \cup I_0$  with  $I_- = \cup_{h \in \mathcal H^-_k}  I_h$  and 
  \begin{align*}
  \mathcal H^-_k  =   \{h\in [\rho_k, 0): \exists \ j \in \mathcal{S}_k \ \text{such that} \ \nabla^k p(j)=h \}. 
  \end{align*}
   \end{enumerate}
\end{theorem}

\begin{remark}
A close inspection of the proof of Theorem  \ref{thm:selection} reveals that the quantile $z_{1-1/n}$ in (\ref{eq:datadriven}) can be replaced by $z_{1-c/n^a}$ for any $c > 0$ and $a > 1/2$. Taking $c=1$ and $a=1$ is a simple and effective choice ensuring a small value of the quantile and such that $1-c/n^a>0$ for any $n>1$. Moreover, since $z_{1- c/n^a} \sim\sqrt{ 2a  \log n} $, taking $a\in (1/2,1]$ is practically irrelevant, and $a=1$ is the most obvious choice. Numerical experiments show that the choice $z_{1-1/n}$ achieves very satisfactory results; see Section \ref{sec:num}.
\end{remark}

The results of Proposition \ref{prop:test-properties} and Theorem \ref{thm:selection} lead to the following corollary, which is a straightforward consequence and hence the proof is omitted.

\begin{corollary}
\label{cor:rejection-region}
Consider the test based on the rejection region $
C_{k,n}(\alpha)$ given in (\ref{eq:generic-region}) with $\widehat I_0$ as in (\ref{eq:datadriven}).
\begin{enumerate}
\item If $p$ satisfies $H_0^{(k)}$ and $\rho_k =0$ then $\lim_{n \to \infty} P(X_{1:n} \in C_{k,n}(\alpha)) = \alpha $. 

\item  If $p$ satisfies $H_0^{(k)}$ and $\rho_k  > 0$ then $\lim_{n \to \infty} P(X_{1:n}\in C_{k,n}(\alpha)) =0 $.

\item If $p$ satisfies $H_1^{(k)}$ then  $\lim_{n \to \infty} P(X_{1:n} \in C_{k,n}(\alpha)) =1 $.

\end{enumerate}

\end{corollary}

\subsection{Data-driven selection of \texorpdfstring{$k$}{k}}
\label{sec:khat}
Here we propose a data-driven approach to select the largest value of $k\in \mathbb{N}_0$ such that $p$ is $j$-monotone for all $j =0, \ldots, k $.  More precisely, let us define
\begin{align}
\label{eq:k0}
k_0 = 
\max \big \{ k \in \{0, \ldots, M-m \}: \text{$p$ is $j$-monotone for all $j \in \{0, \ldots, k \}$}\big\}.
\end{align}
Fix $\alpha \in (0,1)$ and consider the estimator $\widehat k_n$ defined as 
\begin{align}\label{eq:khat}
\widehat k_n  = 
\begin{cases}
0, \ \  \  \text{if $X_{1:n} \in C_{1, n}(\alpha)$}, \\
\max \big \{ k \in \{1, \ldots, X_{(n)}- X_{(1)} \}: X_{1:n} \notin C_{j,n}(\alpha) \ \text{for all $j \in \{1, \ldots, k \}$}\big\},  \  \text{otherwise}.
\end{cases}
\end{align}

In the following Proposition, we show that the probability that $\widehat k_n$ agrees with $k_0$ is asymptotically at least $1-\alpha$. 
\begin{proposition}\label{prop:selectionk}
Let $\widehat k_n$ and $k_0$ be defined as in \eqref{eq:k0} and \eqref{eq:khat}.   It holds that
\begin{align*}
\limsup_{n \to \infty} P(\widehat k_n  \ne k_0 )   \le \alpha.
\end{align*}
Moreover, if $k_0 = 0$, then $\lim_{n \to \infty} P(\widehat k_n  \ne 0 ) = 0$.
\end{proposition}

\subsection{Contiguity and separation}
In Proposition \ref{prop:test-properties}, we have shown that our testing procedure is consistent against any fixed alternative. In the next Proposition, we perform a local power analysis and study the behaviour of the test against local alternatives; see \cite{van2000asymptotic} for a general reference.
In particular, we study the asymptotic behaviour of the test under contiguous alternatives that shrink to the null hypothesis $H_0^{(k)}$ at the $n^{-1/2}$-rate.

 \begin{definition}
 \label{def:contig}
  Let $p$ be a p.m.f.\ satisfying $\rho_k=0$, and denote again by $I_0$ the set of its non-knots. We say that $(p_n)_n$ is a sequence of alternatives  contiguous to $p$ if
  \begin{align*}
  p_n(j) = p(j) + \frac{h(j)}{\sqrt n} , \quad j \in \mathcal S,
  \end{align*}
for some function $h: \mathcal S \to \mathbb R$ such that $\sum_{j \in \mathcal S} h(j) = 0$, $\nabla^k h(j) \le 0$ for all $j \in I_0$ and $\min_{j \in I_0 }\nabla^k h(j) = - \delta$ for some $\delta > 0$. 
  \end{definition}

 In the next theorem, under the contiguous alternative defined in Definition \ref{def:contig}, we derive the asymptotic behaviour of the statistic $\sqrt{n}(\widehat{\rho}_{k,n}-\min_{j\in \mathcal{S}_k}\nabla^k p_n (j))$, of the set $\widehat I_{n,0}$ and, finally, we show that our test is asymptotically unbiased.
\begin{theorem}
\label{thm:contig}
  Let $X_{n,1}\dots,X_{n,n}$ be a triangular array of i.i.d.\ samples drawn from $p_n$. Then, it holds:
  \begin{enumerate}
    \item As $n \to \infty$,
    \begin{align*}
    \sqrt{n}(\widehat{\rho}_{k,n}-\min_{j\in \mathcal{S}_k}\nabla^k p_n (j))\overset{d}{\to}   \min_{j \in I_0}  (Z_j  +  \nabla^k h(j) )    + \delta
    \end{align*}
    with  $(Z_j)_{j\in \mathcal{S}_k}\sim \mathcal{N}(0,\Sigma_k)$ such that
    $$
    (\Sigma_k)_{r,s} =  \cov[\nabla^k\mathbf{1}_{\{X_{1}= m+r-1 \}},\nabla^k\mathbf{1}_{\{X_{1}= m+s-1 \}}]
    $$ 
    for $r,s \in \{1,\dots, M-k-m-1\}$ and $X_1\sim p$.
    \item $P(\widehat{I}_{0,n}\neq I_0) = O(1/\sqrt{n})$.
    \item For all $\alpha\in (0,1)$, 
    $$\liminf_{n\to\infty}P(\widehat{T}_{k,n}< \widehat t_{\alpha,n}(\widehat{I}_{0,n})) \ge \alpha.$$
    If $\max_{j \in I_0} \nabla^k h(j) = - \kappa $ for some $\kappa \in (0, \delta]$, then 
    \begin{align*}
     \liminf_{n\to\infty}P(\widehat{T}_{k,n}< \widehat t_{\alpha,n}(\widehat{I}_{0,n})) \ge P( W_{I_0} \le \widehat t_{\alpha,n}(I_0) + \kappa) > \alpha.
     \end{align*}
    
  \end{enumerate} 
  \end{theorem}

In the following theorem, we provide lower bounds for both the sample size and $\vert \rho_k \vert = - \rho_k $ under $H_1^{(k)}$ so that the $k$-monotonicity test has at least a given power.  Note that the result is non-asymptotic in the sense that the power will be larger or equal to the given threshold whenever $n$ and $\vert \rho_k \vert$ exceed the lower bounds given in the theorem.      

\begin{theorem}\label{thm:seperation}
Let $\xi =  \min_{j \in \mathcal S_k} p(j) \wedge \min_{j \in \mathcal S_k} (1-p(j))$ and suppose that the true p.m.f. $p$ satisfies $H_1^{(k)}$. Let $r > 0$ be such that $\vert \rho_k \vert = - \rho_k \in (0, r]$. Then, for any given $\beta \in (0,1)$,
\begin{align*}
P(X_{1:n} \in C_{k,n}(\alpha)) \ge 1- \beta,
\end{align*}
provided that
\begin{align*}
 n  \ge    \frac{\left(4 + (4 + 5 \sqrt 2) \vert \mathcal S_k \vert \right)^2}{8 \beta^2 \xi} 
 \end{align*}
and
\begin{align*}
\sqrt n\vert \rho_k \vert \ge  \sqrt \xi  \Big(\binom{2k}{k} + r^2 (k+1)\Big)^{1/2}  z_{1-\beta/2}  + \binom{2k}{k}^{1/2} z_{1- \frac{\alpha}{\vert \mathcal S_k \vert}}.
\end{align*}
 
\end{theorem}

\begin{remark}
Theorem  \ref{thm:seperation} is stated under the assumption that $\vert \rho_k \vert \le r$ for some $r > 0$. This assumption is not a real restriction since $\vert \rho_k \vert $ is always bounded. Indeed, using the expression $\nabla^k p(j) = \sum_{l=0}^k (-1)^l \binom{k}{l} p(j + l)$, it follows that $\vert \nabla^k p(j) \vert \le  2^k$, and hence $\vert \rho_k \vert \le 2^k$. Note also that $\xi \le 1/2$, thus, the power is at least $ 1- \beta$ under the more stringent condition  
\begin{align*}
\sqrt n \vert \rho_k \vert \ge  \frac{1}{\sqrt 2} \Big(\binom{2k}{k} + 2^{2k} (k+1)\Big)^{1/2}  z_{1-\beta/2}  +  \binom{2k}{k}^{1/2} z_{1- \frac{\alpha}{\vert \mathcal S_k \vert}}.
\end{align*}    
\end{remark}

\subsection{Comparison with \texorpdfstring{\cite{GIGUELAY201896}}{GIGUELAY201896}}
Before moving to the next sections where we will put our focus on testing the important shapes of monotonicity ($k=1$) and convexity ($k=2$), we would like to note that our approach differs fundamentally from the one used by \cite{GIGUELAY201896}.  First, our test is based on the results of Theorem \ref{thm:main} which characterizes the limiting behaviour of $\min_{j \in \widehat{\mathcal{S}}_{k,n}} \nabla^k \widehat p_n(j)$ in general, and not only under $k$-monotonicity. Second, our test exhibits a type I error probability attaining the nominal level $\alpha$ for the case where $\rho_k =0$, which is, as already mentioned above, equivalent to non-strict $k$-monotonicity. In other words, the set of p.m.f.s under which the asymptotic rejection probability of our test is equal to $\alpha$ includes many p.m.f.s which are allowed to have strictly $k$-monotone parts besides \lq\lq flat\rq\rq \ ones.  For example, for $k=1$,  this set consists of $2^{\vert \mathcal {S}_1\vert}  -1$ p.m.f.s, while the same set in the case of the test by \cite{GIGUELAY201896} contains only the uniform p.m.f.

Indeed, \cite{GIGUELAY201896} do not directly derive the limit distribution of $\widehat{T}_{k,n}=\sqrt{n} \min_{j \in \widehat {\mathcal{S}}_{k,n}}  \nabla^k \widehat p_n(j)$ but rather that of 
$$
\sqrt{n}\min_{j \in \widehat {\mathcal{S}}_{k,n}} \{ \nabla^k \widehat p_n(j) - \nabla^k p(j)  \},
$$
which is stochastically smaller under $k$-monotonicity; i.e.\ $H_0^{(k)}$. In fact, when $p$ is $k$-monotone we have $\nabla^k p(j) \ge 0$ for all $j \in \mathcal{S}_k$, implying that for all $j \in \widehat {\mathcal{S}}_{k,n}$ it holds
$$
\nabla^k \widehat p_n(j)  \ge \nabla^k \widehat p_n(j)  -  \nabla^k  p(j).
$$
Thus, for any $t \in \mathbb R$,
\begin{align*}
P\big(\sqrt{n}\min_{j \in \widehat{\mathcal{S}}_{k,n}} \nabla^k \widehat p_n(j)  \le t \big)   \le  P\big(\sqrt{n}\min_{j \in \widehat{\mathcal{S}}_{k,n}} \{ \nabla^k \widehat p_n(j) - \nabla^k p(j)  \} \le t \big).
\end{align*}
Since the Central Limit Theorem and the Continuous Mapping Theorem can be applied to the sequence  $\sqrt{n}\min_{j \in \widehat{\mathcal{S}}_{k,n}} \{ \nabla^k \widehat p_n(j) - \nabla^k p(j)  \}$, the above inequality gives a way of building an asymptotically valid test for a given nominal level $\alpha$. However, the resulting test of \cite{GIGUELAY201896} is conservative since the type I error probability is, by construction, always strictly smaller than $\alpha$ with equality occurring if the true p.m.f is completely \lq\lq flat\rq\rq, that is and only if $\nabla^k p(j)=0$ for all $j\in\widehat {\mathcal{S}}_k$. Figure \ref{fig:pmf-comparison} shows a comparison of (some) p.m.f.s used for calibration of our test and the test of \cite{GIGUELAY201896}. 

\begin{figure}[!ht]
\centering
\includegraphics[width = \textwidth]{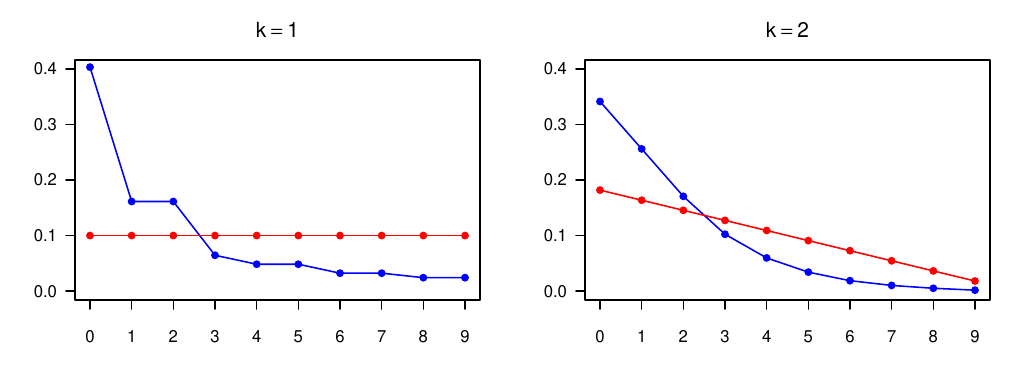}
\vspace{-1cm}
\caption{Visual comparison of $k$-monotone (for $k=1,2$)  p.m.f.s used for calibration of the proposed test (blue) and the test of \cite{GIGUELAY201896} (red). }
\label{fig:pmf-comparison}
\end{figure}

\section{Testing monotonicity and convexity}
\label{sec:mono-conv}
In this section, we first apply the results of Section \ref{sec:tests} to monotonicity and convexity; i.e., $k \in \{1, 2 \}$. We decided to focus on such cases because they are the most relevant both in practical and methodological settings. Besides the test of Section  \ref{sec:tests}, we will propose alternative tests based on the Grenander and convex least squares estimator for $k \in \{1, 2 \}$ respectively. Although the asymptotic theory of the statistics involved in these alternative tests has been established in prior work, our main goal is to show how one can use the data-driven selection procedure developed above to construct asymptotically valid and consistent tests. It must be pointed out that for such a testing framework it is much more difficult to establish theoretical guarantees for the contiguity setting and to derive quantitative separation results. Such investigation is left for future research.
In the sequel, we denote the Euclidean norm by $\Vert \cdot \Vert_2$ and refer to it as the $\ell_2$-norm.
\subsection{Monotonicity}
\label{sec:mono}
Here, we discuss procedures that are specifically tailored to testing monotonicity. First of all, we briefly specify the results obtained in Section \ref{sec:tests} for $k=1$ and subsequently apply our methodology for knot selection to an alternative approach based on the $\ell_2$-distance between the empirical estimator and its monotone projection, better known as the Grenander estimator. We show that the limit distribution of the re-scaled $\ell_2$-distance can be consistently approximated using the proposed knot selection procedure.

\subsubsection{Testing monotonicity via the empirical measure}

Based on the results shown in Section \ref{sec:tests}, our $k$-monotonicity test via the empirical measure rejects $H_0^{(1)}: \rho_1 \ge 0$ versus $H_1^{(1)}: \rho_1 < 0$ if and only if 
\begin{align*}
\widehat T_{1, n} = \sqrt n \widehat{\rho}_{1, n}  < \widehat t_{\alpha,n}(\widehat{I}_{0, n})
\end{align*}
where
\begin{align}
\label{eq:datadriven-monotone}
\widehat I_{0,n} = \Big\{j \in \widehat {S}_{1,n} :  \frac{\sqrt{n} (\widehat p_n(j) - \widehat p_n(j+1))}{ \sqrt {(\widehat{\Sigma}_{1,n})_{j-m+1,j-m+1}} }  \leq z_ {1-1/n} \Big\},
\end{align}
By Theorem \ref{thm:selection} and consistency of $\widehat {\Sigma}_{1,n}$, under the least favourable case $\rho_1 = 0$,  $\widehat t_{\alpha,n}(\widehat{I}_{0, n})$ converges in probability to the $(1-\alpha)$-quantile of $W_{I_0} = \min_{j\in I_{0}} Z_j$ (see Lemma \ref{lem:convquant}), where $I_0$ is the set of the non-knots of $p$,  and 
$(Z_{m},\dots,Z_{M-1})\sim \mathcal{N}(0,\Sigma_1)$ with $\Sigma_1$ given in Proposition \ref{prop:clt_monotonicity} in the Appendix. Moreover, by Corollary \ref{cor:rejection-region}, the test is asymptotically valid and consistent.
Note that other methods for approximating the set $I_0$ give rise to different monotonicity tests. More details will be given in Section \ref{sec:num}.

\subsubsection{Testing monotonicity via the Grenander estimator}
Here, we propose another testing procedure which is based on the monotone projection of the empirical estimator of $p$ where we apply the knot selection method of Section \ref{sec:tests}. 
First, we recall that the monotone projection of $\widehat p_n$ (also known as the Grenander estimator) denoted by $\widehat{p}^{\mathcal{M}}_n$ is given by
\begin{align*}
\widehat{p}^{\mathcal{M}}_n  = \underset{\substack{q \in \mathcal M} }{\arg \min} - \frac{1}{n} \sum_{i=1}^n \log q(X_i) = \underset{\substack{q \in \mathcal M} }{\arg \min} -  \sum_{j \in \mathcal S} \widehat p_n(j) \log q(j),
\end{align*}
where $\mathcal M$ is the set of non-increasing p.m.f.s. supported on $\mathcal S$ (here, we use the fact that the observed support $\widehat S_n$ and $\mathcal S$ are equal with probability 1 for $n$ large enough).  It is well-known that $\widehat{p}^{\mathcal{M}}_n$ is the vector of the left slopes of the least concave majorant of the empirical distribution function of $X_1, \ldots, X_n$; see \cite{jankowski2009} for more details. Note that the Grenander estimator is also equal to the $\ell_2$-projection of $\widehat p_n$ on the set $\mathcal M$, see  e.g. \cite[Example 1.10]{barlow1972}, or \cite[Theorem 1.5]{Robert88}. More specifically, 
\begin{align*}
\widehat{p}^{\mathcal{M}}_n  = \underset{\substack{q \in \mathcal M} }{\arg \min} \Vert \widehat p_n - q \Vert_2.
\end{align*}
In the following, we denote  by $\mathrm{gren}$ the operator which projects, in the sense of the $\ell_2$-norm, a given vector $v= (v_1, \ldots, v_q)$ on the set of vectors in $\mathbb R^q$ with non-increasing components; i.e., 
$$
\mathrm{gren}(v) =  \underset{\substack{w \in \mathbb R^q: w_1 \ge \ldots \ge w_q} }{\arg \min}  \Vert v - w \Vert_2.
$$
For the purpose of testing monotonicity of $p$, we can introduce the following test statistic based on $\widehat{p}^{\mathcal{M}}_n$, 
\begin{align*}
    \widehat T^\mathcal{M}_n:= \sqrt{n} \| \widehat p_n^{\mathcal{M}}- \widehat p_n\|_2.
\end{align*}
whose asymptotic behaviour is derived in \cite{jankowski2009}.   Let $p$ be a monotone non-increasing p.m.f. and denote by $J_0  = \{j_1, \ldots, j_q \} \subseteq \mathcal S_1$ the set of its knots (see Definition \ref{def:k-mon}). Then, it has been shown  in \cite[Corollary 4.1]{jankowski2009} that  
$$
\widehat T^\mathcal{M}_n \overset{d}{\to} \mathcal{T}^\mathcal{M}(J_0):= \| G^{\mathcal{M}}(J_0)-G\|_2,
$$
as $n\to\infty$, where $G\sim\mathcal{N}(0,\Gamma) $ is such that 
$$(\Gamma)_{r,s} = \mathbf{1}_{\{r=s\}}p(m -1 +r) -p(m -1 +r)p(m -1 + s)$$ for $r,s\in \{1,\dots,M-m+1\}$, and 
$G^{\mathcal{M}}(J_0)  =  (G^{\mathcal{M}}_1,  \ldots, G^{\mathcal{M}}_{M-m+1})$ is constructed as follows: 
For a constancy region 
$$\{u, \ldots, v \}  \in  \Big \{ \{m, \ldots, j_1 \}, \{j_1 +1, \ldots, j_2 \}, \ldots, \{j_{q}+1, \ldots, M \} \Big \},$$
we have that 
\begin{align*}
(G^{\mathcal{M}}_{u-m+1},  \ldots,  G^{\mathcal{M}}_{v-m+1})  = \mathrm{gren}(G_{u-m+1},  \ldots,  G_{v-m+1}).
\end{align*}
Note that when $p$ is strictly monotone, then $J_0 = \mathcal S_1$ and  $G^{\mathcal{M}}(J_0) = G $ with probability 1.  

A good approximation of the distribution of $\mathcal T^{\mathcal{M}}(J_0)$ under $H_0^{(1)}$ requires having a good guess about the constancy regions of $p$, or equivalently $J_0$. However, this is not available. This issue can be solved directly by the selection method described in Section \ref{sec:tests}. Indeed, for a given $\alpha\in(0,1)$ we will consider the rejection region
\begin{align}
\label{eq:region_mono}
     C_{n}^\mathcal{M}(\alpha)=\{ x_{1:n}\in \mathcal{S}^n: \widehat T^\mathcal{M}_n(x_{1:n}) > \widehat t_{1-\alpha,n}^\mathcal{M}(\widehat J_{0,n}) \},
\end{align}
where $\widehat t_{1-\alpha,n}^\mathcal{M}(\widehat J_{0,n})$ is the $(1-\alpha)$-quantile of the r.v.\  $\widehat{\mathcal {T}}^{\mathcal M}_n(\widehat J_{0, n})$ obtained replacing $\Gamma$ with its empirical estimator $\widehat \Gamma_n$ in the definition of $\mathcal T^{\mathcal M}(\widehat J_{0,n})$, and
$\widehat J_{0, n}  =\widehat {\mathcal S}_{1, n} \setminus \widehat I_{0,n}$  with $\widehat I_{0,n}$ as in \eqref{eq:datadriven-monotone}. In the following corollary, we provide theoretical guarantees for the testing procedure. 
\begin{corollary}
\label{cor:rejection-region-mono-proj}
Consider the test based on the rejection region $
C_{n}^{\mathcal{M}}(\alpha)$ given in (\ref{eq:region_mono}).
\begin{enumerate}
\item If $p$ satisfies $H_0^{(1)}$ and $\rho_1 =0$ then $\lim_{n \to \infty} P(X_{1:n} \in C_{n}^\mathcal{M}(\alpha)) = \alpha $. 

\item  If $p$ satisfies $H_0^{(1)}$ and $\rho_1  > 0$ then $\lim_{n \to \infty} P(X_{1:n}\in C_{n}^\mathcal{M}(\alpha)) =0 $.

\item If $p$ satisfies $H_1^{(1)}$ then  $\lim_{n \to \infty} P(X_{1:n} \in C_{n}^\mathcal{M}(\alpha)) =1 $.

\end{enumerate}

\end{corollary}

\subsection{Convexity}
\label{sec:conv}

As for monotonicity,  we describe procedures specifically tailored for testing convexity. Once again, we proceed by first specifying the results of Section \ref{sec:tests} for $k=2$ and, subsequently, applying the knot selection method to an alternative testing approach based on the $\ell_2$-distance between the empirical estimator and its convex projection proposed in \cite{DUROT2013282}.

\subsubsection{Testing convexity via the empirical measure}

Based on the results shown in Section \ref{sec:tests}, our $k$-monotonicity test via the empirical measure rejects $H_0^{(2)}: \rho_2 \ge 0$ versus $H_1^{(2)}: \rho_2 < 0$ if and only if 
\begin{align*}
\widehat T_{2, n} = \sqrt n \widehat{\rho}_{2, n}  < \widehat t_{\alpha,n}(\widehat{I}_{0, n})
\end{align*}
where
\begin{align}
\label{eq:datadriven-convex}
\widehat I_{0,n} = \Big\{j \in \widehat {\mathcal {S}}_{2,n} :  \frac{\sqrt{n} (\widehat p_n(j+2) - 2 \widehat p_n(j+1) + \widehat p_n(j))}{  { \sqrt {(\widehat{\Sigma}_{2,n})_{j-m+1,j-m+1}} }}  \leq z_ {1-1/n} \Big\}.
\end{align}
By Theorem \ref{thm:selection} and consistency of $\widehat {\Sigma}_{2,n}$, under the least favourable case $\rho_2 = 0$,  $\widehat t_{\alpha,n}(\widehat{I}_{0, n})$ converges in probability to the $(1-\alpha)$-quantile of $W_{I_0} = \min_{j\in I_{0}} Z_j$ (see Lemma \ref{lem:convquant}), where $I_0$ is the set of the non-knots of $p$,  and 
$(Z_{m},\dots,Z_{M-2})\sim \mathcal{N}(0,\Sigma_2)$ with $\Sigma_2$ given in Proposition \ref{prop:clt_convexity} in the Appendix. Moreover, by Corollary \ref{cor:rejection-region}, the test is asymptotically valid and consistent.
Note that other methods for approximating the set $I_0$ give rise to different convexity tests. More details will be given in Section \ref{sec:num}.

\subsubsection{Testing convexity via the convex LSE}
\cite{BALABDAOUI20188} propose a procedure to test convexity of a p.m.f.\ supported on $\{0, \dots, M\}$ for some integer $M \ge 1$.  In their scope,  the p.m.f.\ $p$,  although finitely supported, is assumed to be convex on $\mathbb N_0$ meaning that the definition of convexity is the same as the one used in \cite{giguelay2017}.  The proposed test is based on the convex least squares estimator (CLSE) of $p$ studied in \cite{DUROT2013282}.  More specifically, the CLSE of $p$ is given by 
\begin{align*}
\widehat p^{\mathcal{C}}_n:= \underset{q \in \mathcal C}{\arg \min} \|\widehat p_n - q\|_2 
\end{align*}
where $\mathcal C $ is the set of convex p.m.f.s supported on $\mathcal S$, see \cite{DUROT2013282} for general details and \cite{balabdaoui2017} for the characterization of $\widehat p^{\mathcal{C}}_n $ and its asymptotic behaviour.
Here, similarly to \cite{BALABDAOUI20188}, we propose a testing procedure based on a slightly modified CLSE $\widehat p^{\mathcal{C}}_n$ where the convexity constraint is imposed only on the empirical support $\widehat{\mathcal{S}}_n=\{X_{(1)}, \ldots,  X_{(n)} \}$ and not on the whole $\mathbb{N}_0$. Let us consider the test statistic 
\begin{align*}
    \widehat T _ n^{\mathcal{C}} := \sqrt{n}\| \widehat p^{\mathcal{C}}_n - \widehat p_n  \|_2.
\end{align*}
The asymptotic behaviour follows from Theorem 1.1 of \cite{BALABDAOUI20188}. In particular, when $p$ is convex and $J_0  = \{j_1, \ldots, j_q\} \subseteq \mathcal S_2$ the set of its knots  (see Definition \ref{def:k-mon}), then
 $$ \widehat T_n^{\mathcal{C}} \overset{d}{\to} \mathcal T^{\mathcal{C}}(J_0):=\|G^{\mathcal{C}}(J_0)-G\|_2,$$
 as $n\to\infty$, where $G  \sim \mathcal{N}(0,\Gamma)$ is as above,  and $G^{\mathcal{C}}(J_0)=(G^{\mathcal{C}}_1,  \ldots, G^{\mathcal{C}}_{M-m+1})$ is defined as follows: For a linear region 
 \begin{align*}
  \{u,\ldots, v \} \in \Big \{ \{m, \ldots, j_1\}, \{j_1 +1, \ldots, j_2 \}, \ldots, \{j_{q}+1, \ldots, M \}  \Big \},  
 \end{align*}
 we have that 
\begin{align*}
(G^{\mathcal{C}}_{u-m+1},  \ldots,  G^{\mathcal{C}}_{v-m+1})  =  \underset{{\substack{ \ q \ \text{convex on} \\ \{u,\ldots, v\}}}}{ \arg \min} \big(\sum_{l=u}^v  \left(G_{l-m+1}  -  q(l) \right)^2\big)^{1/2}.
\end{align*}

The distribution of the limit $\mathcal T ^{\mathcal{C}}(J_0)$  depends on the unknown location of the knots of the true p.m.f.\ $p$. Building on the results established in \cite{balabdaoui2017}, \cite{BALABDAOUI20188} use a non-negative sequence $\{v_n\}_n$ such that $v_n = o(1)$ and $v_n \gg n^{-1/2}$ to select the knots in the following way: If $\nabla^2 \widehat{p}_n(j) \le v_n$, then $j$ is declared to belong to a linear region of $p$.  In other words,  a point in the support is considered a knot of $p$ if $\nabla^2 \widehat{p}_n(j) > v_n.$  The main drawback of this approach is that there are infinitely many sequences $\{v_n\}_n$ that can be used, and it is unclear which ones are optimal for the particular distribution at hand. Thus, the arbitrariness in the choice of the sequence $\{v_n\}_n$ makes the resulting testing procedure strongly dependent on $p$, as confirmed in the simulation study presented in \cite{BALABDAOUI20188}. The authors also propose an alternative calibration method under the \lq\lq least favourable hypothesis\rq\rq. While the resulting test does not depend on a tuning parameter, it is by definition conservative since its type I error probability is always (asymptotically) strictly smaller than the nominal level, which is only approached when $p$ is a triangular p.m.f.\ and $n$ is large.

As done in the monotone problem, we use the alternative approach based on the non-knot selection given in \eqref{eq:datadriven}. For a given $\alpha\in(0,1)$ we will consider the rejection region
\begin{align}
\label{eq:region_conv}
     C_{n}^\mathcal{C}(\alpha)=\{ x_{1:n}\in \mathcal{S}^n: \widehat T^\mathcal{C}_n(x_{1:n}) > \widehat t_{1-\alpha, n}^\mathcal{C}(\widehat J_{0,n}) \},
\end{align}
where $\widehat t_{1-\alpha,n}^\mathcal{C}(\widehat J_{0,n})$ is the $(1-\alpha)$-quantile of the r.v.\  $\widehat{\mathcal {T}}^{\mathcal C}_n(\widehat J_{0, n})$ obtained replacing $\Gamma$ with its empirical estimator $\widehat \Gamma_n$ in the definition of $\mathcal T^{\mathcal C}(\widehat J_{0,n})$, and
$\widehat J_{0, n}  =\widehat {\mathcal S}_{1, n} \setminus \widehat I_{0,n}$  with $\widehat I_{0,n}$ as in \eqref{eq:datadriven-monotone}.
Similarly to Corollary \ref{cor:rejection-region-mono-proj}, in the next statement we provide theoretical guarantees that directly follow from Theorem \ref{thm:selection}. The proof is omitted since it directly follows from the results of \cite{BALABDAOUI20188} combined with the same arguments provided in the proof of Corollary \ref{cor:rejection-region-mono-proj}.

\begin{corollary}
\label{cor:rejection-region-conv-proj}
Consider the test based on the rejection region $
C_{n}^{\mathcal{C}}(\alpha)$ given in (\ref{eq:region_conv}).
\begin{enumerate}
\item If $p$ satisfies $H_0^{(2)}$ and $\rho_2 =0$ then $\lim_{n \to \infty} P(X_{1:n} \in C_{n}^\mathcal{C}(\alpha)) = \alpha $. 

\item  If $p$ satisfies $H_0^{(2)}$ and $\rho_2  > 0$ then $\lim_{n \to \infty} P(X_{1:n}\in C_{n}^\mathcal{C}(\alpha)) =0 $.

\item If $p$ satisfies $H_1^{(2)}$ then  $\lim_{n \to \infty} P(X_{1:n} \in C_{n}^\mathcal{C}(\alpha)) =1 $.

\end{enumerate}

\end{corollary}

\section{Numerical experiments}
\label{sec:num}

\subsection{Preliminaries}

In this section, we inspect the performance and applicability to real data of the proposed testing methodologies.
To provide an extensive comparison, in the following, we evaluate the performance of the following tests and resulting $\widehat k_n$ (as defined in Section \ref{sec:khat}) for a given nominal level $\alpha=0.05$:
\begin{enumerate}
    \renewcommand{\labelenumi}{(\roman{enumi})}
\item Test based on the rejection region \ref{eq:generic-region} with 
$$
\widehat I_{0,n}=\widehat{\mathcal  {S}}_{k,n} = \{X_{(1)},\ldots,X_{(n)} - k \},
$$
which is equal to $\mathcal{S}_k$ almost surely as $n\to \infty$, as a consequence of the Strong Law of Large Numbers. Although it is the simplest choice, it is not consistent for $I_0$. Such a choice leads to a test that is equivalent to the one proposed in \cite{GIGUELAY201896}, and since the resulting limiting distribution is stochastically smaller than $W_{I_0}$, the test is by definition conservative and less powerful.
\item Test based on the rejection region \ref{eq:generic-region} with
\begin{align*}
   \widehat I_{0,n}= \{j \in \widehat{\mathcal  {S}}_{k,n} : \nabla^k \widehat p_n(j) \leq v_{k,n} \}.
\end{align*}
We set $v_{k,n} = n^{-1/(k+2)} $ so that we recover the approach of \cite{BALABDAOUI20188} where the authors use $v_{k,n}=v_{n}=n^{-1/4}$ for testing convexity.

\item Test based on the rejection region \ref{eq:generic-region} with $\widehat I_{0,n}$ defined in \eqref{eq:datadriven} for which we derived full theoretical asymptotic guarantees in Theorem \ref{thm:selection}.
\item Test based on either the rejection region \eqref{eq:region_mono} (for monotonicity) or rejection region \eqref{eq:region_conv} (for convexity). The resulting test is more computationally intensive and complex to implement as it requires the computation of either $\widehat p_n^{\mathcal{M}}$ or $\widehat p_n^{\mathcal{C}}$ and of local projections for approximating the limit distribution.
\end{enumerate}

\subsection{Validation of the asymptotic theory}

To illustrate the validity of the theoretical results of Section \ref{sec:tests}, in Figure~\ref{fig:test_emp} we present the kernel density estimators of the true distribution of $\widehat T_{k,n}$, obtained for $k=1,2$ and $n=1000$. We compare the true distribution with its theoretical limiting behaviour using the true knots and our estimation of these knots, using the three proposed methods. The true $p$ is that of a truncated Poisson  with support $\{0,\ldots, 4\}$ and parameters $\lambda=1$ (for $\widehat T_{1,n}$) and $\lambda=2-\sqrt{2}$ (for $\widehat T_{2,n}$), so that $p$ is monotone and convex, respectively; see Table \ref{tab:Poisson}.
To plot the true distribution of the test statistic for $n=1000$ we generated 5000 independent Monte Carlo replicas of $(X_1,\dots, X_{1000})$ and each time we computed the statistic. To plot the theoretical limit, we drew a single sample of size 1000 from $p$, and used it to draw 5000 independent samples from the limiting random variable  $W_{I_0}$ and from its approximations obtained using the methods adopted in Tests (i), (ii) and (iii), respectively. Figure~\ref{fig:test_emp} shows that (i) provides an approximation of $W_{I_0}$ which is stochastically smaller, hence leading to a conservative test for $k=1,2$. Similarly, (ii) provides a stochastically smaller approximation for $k=2$.
On the other hand, (iii) provides a consistent estimation of the set $I_0$, leading to accurate approximations of the theoretical limiting r.v.\ $W_{I_0}$ for $k=1,2$.

\begin{figure}[!ht]

\centering
\includegraphics[width = \textwidth]{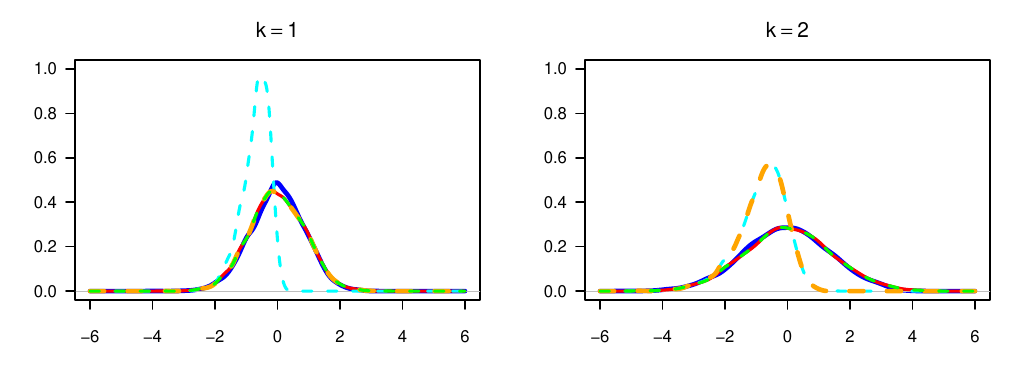}  
\vspace{-1cm}
\caption{Density of $\widehat T_{k,n}$ for $k=1,2$ and $n=1000$, its theoretical limit and its approximations: true density (blue), theoretical limit (red), approximated limit with (i) (dashed cyan), approximated limit with (ii) (dashed orange), approximated limit with (iii) (dashed green). }
\label{fig:test_emp}
\end{figure}

In Figure~\ref{fig:test_shape} we compare the kernel density estimators of the true distribution of $\widehat T^{\mathcal{M}}_n$ and $\widehat T^{\mathcal{C}}_n$, for $n=1000$ with their theoretical limiting behaviour using the true knots and our estimation of these knots. The true $p$ is chosen as done for Figures~\ref{fig:test_emp}. The true distribution is obtained as above, and for the theoretical limit, we drew a single sample of size 1000 from $p$, and used it to draw 5000 independent samples from the limit distribution and its approximation. The results show that the method provides a consistent approximation of the set of knots for both monotonicity and convexity.

\begin{figure}[!ht]

\centering
\includegraphics[width = \textwidth]{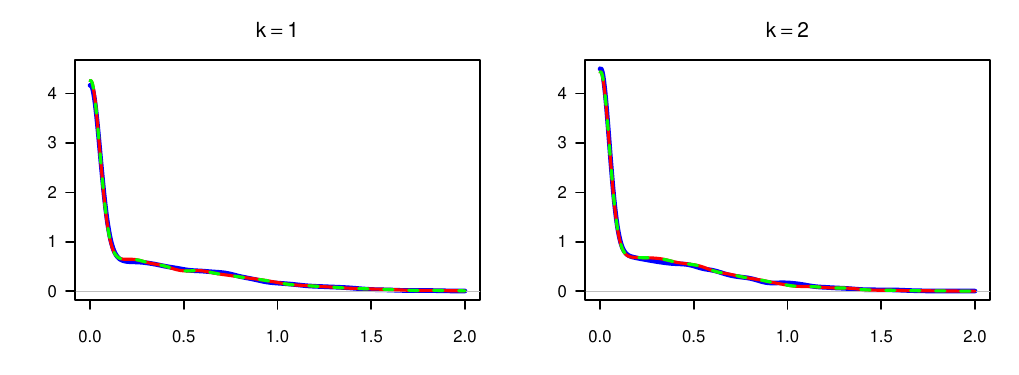} 
\vspace{-1cm}
\caption{Density of $\widehat T^\mathcal{M}_n$ (on the left) and $\widehat T^\mathcal{C}_n$ (on the right) for $n=1000$, their theoretical limits and their approximations: True density (blue), theoretical limit (red), and approximated limit (dashed green). }
\label{fig:test_shape}
\end{figure}

\subsection{Finite sample performance of the tests}
To evaluate the finite sample performance of the tests, independent Monte Carlo samples were generated under $H_0^{(k)}$, for $k=1,2$, and under some fixed alternatives.
The results for the empirical power function of the tests under the null and alternative hypotheses are reported as the percentage of rejections of the tests at the nominal level $\alpha = 0.05$. 

For our numerical investigation, we consider the following distributions: Truncated Poisson distributions denoted by $\mathcal{P}(m,M,\lambda)$ where $0<m<M$ and $\lambda$ can be selected according to the required order of $k$-monotonicity, see \cite{GIGUELAY201896} and Table 3 in \cite{balabdaoui2019multiple}. Specifically we consider $\lambda=2$ for non-monotonicity, $\lambda=1$ for monotonicity ($k=1$) and $\lambda=2-\sqrt{2} $ for convexity ($k=2$). Table \ref{tab:Poisson} summarises the considered scenarios.
\begin{table}[!ht]
 \centering
 \caption{Monotonicity and convexity of the Poisson distribution at different parameter values.}
  \medskip 
 \begin{tabular}{lcc}
        \toprule
        $\lambda$ $\backslash$ Test   & monotonicity & convexity \\
        \midrule
        $2$ (non-monotone)                       & $H_1^{(1)}$ & $H_1^{(2)}$  \\
        $1$ (1-monotone)                         & $H_0^{(1)}$ & $H_1^{(2)}$  \\
        $2-\sqrt{2}$ (2-monotone)                & $H_0^{(1)}$ & $H_0^{(2)}$  \\
        \bottomrule
    \end{tabular}
    \label{tab:Poisson}
\end{table}
Other p.m.f.s that are convex on $\{0,\dots,M\}$ are the mixtures of triangular distributions
\begin{align*}
    p= \sum_{r=1}^{M+1}\pi_r\mathcal{T}_r
\end{align*}
where $\mathcal{T}_r(i)= \frac{2(r-i)_+}{r(r+1)}$ is the Triangular p.m.f.\ supported on $\{0,\dots,r-1\}$. We denote such law as $\mathcal{MT}(\pi_1,\dots,\pi_{M+1})$. In our simulations, we considered convex distributions supported on $\{0, \ldots, M\}$ with $\pi_1=\cdots=\pi_{M+1}=1/(M+1)$. Note that mixtures of triangular distributions have to be also decreasing and, as mentioned in the introduction, this implication does not hold under our definition of $2$-monotonicity because we do not require $\nabla^2 p(j)$ to have a non-negative sign for $j > M-2$.
Under the alternative hypothesis, we consider also the truncated binomial distribution denoted by $\mathcal{B}(m,M,r,q)$ supported on $\{m,\dots,M\}$ where $r \in \{m, \ldots, M \}$ is the number of trials and $q \in (0,1)$ is the success probability.
In all the simulations, we fix $\alpha=0.05$, let $n\in\{100,1000\}$, and consider 5000 independent Monte Carlo replications drawn from each fixed model. For each Monte Carlo replication, we draw 5000 samples from the limit distribution of the test statistics used to approximate the quantile of interest.

From both Tables \ref{tab:monotonicity0} and \ref{tab:convexity0}, one can see that Test (i) exhibits a conservative behaviour since the proportion of rejections is smaller than $\alpha$. This behaviour is particularly evident in the case of $\mathcal{P}(0,4,1)$ and $\mathcal{P}(0,9,1)$ when testing for monotonicity. On the other hand, under the same models,  Tests (ii), (iii), and (iv) exhibit a proportion of rejections close to the nominal level for $n=1000$ for both monotonicity and convexity. Such improvement over the conservative test of \cite{GIGUELAY201896} (which is equivalent to Test (i)) is aligned with the insights given in Section \ref{sec:tests}.

As highlighted in \cite{GIGUELAY201896} and Proposition \ref{prop:test-properties}, a test for $k$-monotonicity of the distribution which is $s$-monotone for $s>k$ results in a type I error probability converging to 0. Therefore, this behaviour should be observed in all the considered simulation settings where the distribution satisfies $s$-monotonicity with $s$ larger than the $k$ considered in $H_0^{(k)}$.  The aforementioned behaviour can be seen for Poisson models with intensity parameter chosen such that the model is $k$-monotone with $k$ larger than the one under $H_0^{(k)}$, see Table \ref{tab:Poisson} for the considered parameter values and the corresponding $k$-monotonicity.

A similar behaviour can be noted in Table \ref{tab:convexity0} for the $\mathcal{MT}$ distributions.

Table \ref{tab:monotonicity1} shows that Tests (i) and (ii) achieve a smaller empirical power than Test (iii) under several fixed alternative hypotheses when testing for monotonicity, while Table \ref{tab:convexity1} shows that the performances become comparable when testing for convexity.
Moreover, it is worth noticing that when testing for convexity under some alternatives, Test (iv) achieves a better performance than the other tests, reaching high empirical power already for small sample sizes. For instance, in Table \ref{tab:convexity1} we see that under the Binomial alternatives it reaches 100\% power already for $n=100$. Despite this advantage in terms of power, we point out that Test (iv) is much more computationally intensive and more complex to be implemented.
As a general remark, all the proposed testing procedures are shown to be consistent, i.e.\ the proportion of rejections under fixed alternatives increases to 1 as $n$ increases.

Finally, Table \ref{tab:k_selection_sim} reports the performance of $\widehat{k}_n$, introduced in Section \ref{sec:khat} for selecting the largest monotonicity degree, in terms of Monte Carlo mean and Mean Absolute Error (MAE). We consider a truncated geometric distribution, denoted by $\mathcal{G}(m, M, p)$, for which it is not difficult to see that $k_0 = M-m$. From the reported results, it is immediate to see that $\widehat{k}_n$ tends to underestimate the true monotonicity degree $k_0$, particularly when it is large. This behaviour can be attributed to finite-sample effects: as the support increases, sparsity of the data becomes more pronounced, producing gaps and numerical artifacts that make the procedure incorrectly reject $k$-monotonicity even when it holds.
It is worth noting that choosing a smaller $\alpha$ would reduce this underestimation, but at the cost of a less conservative procedure that is more susceptible to upward fluctuations of the estimator. Indeed, for $\alpha = 0.05$, $\widehat{k}_n$ can be interpreted as a conservative estimator, effectively providing a lower bound for the true value of $k_0$. This property remains valuable in practical modelling contexts, where a lower bound for $k_0$ can meaningfully guide and constrain model selection.

\begin{table}[!ht]
    \centering
    \caption{Percentage of rejections under $H_0^{(1)}$ for monotonicity tests: Empirical evidence shows that the tests are valid while exhibiting conservativeness when the model's $s$-monotonicity degree is larger than $k=1$.}
    \medskip
    \begin{tabular}{lrrrrrrrr}
        \toprule
        & \multicolumn{4}{c}{$n=100$} & \multicolumn{4}{c}{$n=1000$} \\
        \cmidrule(lr){2-5} \cmidrule(lr){6-9}
        {Model $\backslash$ Test} & (i) & (ii) & (iii) & (iv)  & (i) & (ii) & (iii) & (iv)  \\
        \midrule
        $\mathcal{P}(0,4,1)$                             & 3.3 & 4.2 & 4.5 & 4.2 & 3.5 & 5.3 & 5.3 & 5.2 \\
        $\mathcal{P}(0,4,2-\sqrt{2})$                    & 0.0 & 0.0 & 0.0 & 0.0 & 0.0 & 0.0 & 0.0 & 0.0 \\
        $\mathcal{P}(0,9,1)$                             & 3.1 & 4.3 & 4.6 & 4.0 & 3.6 & 4.9 & 4.9 & 5.0 \\
        $\mathcal{P}(0,9,2-\sqrt{2})$                    & 0.0 & 0.0 & 0.0 & 0.0 & 0.0 & 0.0 & 0.0 & 0.0 \\
        $\mathcal{MT}(\frac{1}{5},\dots,\frac{1}{5})$    & 0.0 & 0.0 & 0.1 & 0.1 & 0.0 & 0.0 & 0.0 & 0.0 \\
        $\mathcal{MT}(\frac{1}{10},\dots,\frac{1}{10})$  & 0.0 & 0.1 & 0.1 & 0.0 & 0.0 & 0.0 & 0.0 & 0.0 \\
        \bottomrule
    \end{tabular}
    \label{tab:monotonicity0}
\end{table}

\begin{table}[!ht]
    \centering
    \caption{Percentage of rejections under $H_1^{(1)}$ for monotonicity tests: Empirical evidence shows that the tests are consistent.}
    \medskip
    \begin{tabular}{lrrrrrrrr}
        \toprule
        & \multicolumn{4}{c}{$n=100$} & \multicolumn{4}{c}{$n=1000$} \\
        \cmidrule(lr){2-5} \cmidrule(lr){6-9}
        {Model $\backslash$ Test} & (i) & (ii) & (iii) & (iv)  & (i) & (ii) & (iii) & (iv)  \\
        \midrule
        $\mathcal{P}(0,4,2)$                     & 47.3 & 48.1 & 49.0 & 57.8 & 100.0 & 100.0 & 100.0 & 100.0 \\
        $\mathcal{P}(0,9,2)$                     & 45.2 & 45.4 & 46.7 & 52.6 & 100.0 & 100.0 & 100.0 & 100.0 \\
        $\mathcal{B}(0,4,4,0.5)$                 & 91.5 & 91.9 & 92.4 & 99.8 & 100.0 & 100.0 & 100.0 & 100.0 \\
        $\mathcal{B}(0,9,4,0.5)$                 & 91.7 & 92.1 & 92.6 & 99.9 & 100.0 & 100.0 & 100.0 & 100.0 \\
        \bottomrule
    \end{tabular}
    \label{tab:monotonicity1}
\end{table}

\begin{table}[!ht]
    \centering
    \caption{Percentage of rejections under $H_0^{(2)}$ for convexity tests: Empirical evidence shows that the tests are valid while exhibiting conservativeness when the model's $s$-monotonicity degree is larger than $k=2$.}
    \medskip
    \begin{tabular}{lrrrrrrrr}
        \toprule
        & \multicolumn{4}{c}{$n=100$} & \multicolumn{4}{c}{$n=1000$} \\
        \cmidrule(lr){2-5} \cmidrule(lr){6-9}
        {Model $\backslash$ Test} & (i) & (ii) & (iii) & (iv)  & (i) & (ii) & (iii) & (iv)  \\
        \midrule
        $\mathcal{P}(0,4,2-\sqrt{2})$                    & 4.0 & 4.2 & 4.2 & 4.2 & 4.3 & 4.5 & 4.6 & 4.6 \\
        $\mathcal{P}(0,9,2-\sqrt{2})$                    & 4.1 & 4.3 & 4.3 & 4.2 & 4.2 & 4.4 & 4.4 & 4.4 \\
        $\mathcal{MT}(\frac{1}{5},\dots,\frac{1}{5})$    & 0.1 & 0.6 & 0.6 & 0.8 & 0.0 & 0.0 & 0.0 & 0.3 \\
        $\mathcal{MT}(\frac{1}{10},\dots,\frac{1}{10})$  & 0.9 & 1.1 & 1.4 & 1.1 & 0.2 & 0.2 & 0.5 & 0.2 \\
        \bottomrule
    \end{tabular}
    \label{tab:convexity0}
\end{table}

\begin{table}[!ht]
    \caption{Percentage of rejections under $H_1^{(2)}$ for convexity tests: Empirical evidence shows that the tests are consistent.}
    \medskip
    \centering
    \begin{tabular}{lrrrrrrrr}
        \toprule
        & \multicolumn{4}{c}{$n=100$} & \multicolumn{4}{c}{$n=1000$} \\
        \cmidrule(lr){2-5} \cmidrule(lr){6-9}
        {Model $\backslash$ Test} & (i) & (ii) & (iii) & (iv)  & (i) & (ii) & (iii) & (iv)  \\
        \midrule
        $\mathcal{P}(0,4,2)$                        & 29.1 & 29.1 & 29.1 & 99.8 & 100.0 & 100.0 & 100.0 & 100.0 \\
        $\mathcal{P}(0,4,1)$                        & 31.7 & 31.6 & 31.8 & 33.7 & 99.1 & 99.1 & 99.2 & 99.6 \\
        $\mathcal{P}(0,9,2)$                        & 28.1 & 28.1 & 27.9 & 94.8 & 100.0 & 99.9 & 100.0 & 100.0 \\
        $\mathcal{P}(0,9,1)$                        & 31.9 & 31.8 & 31.9 & 33.3 & 99.3 & 99.3 & 99.2 & 99.6 \\
        $\mathcal{B}(0,4,4,0.5)$                    & 51.6 & 51.7 & 51.5 & 100.0 & 100.0 & 100.0 & 100.0 & 100.0 \\
        $\mathcal{B}(0,9,4,0.5)$                    & 51.3 & 51.5 & 51.3 & 100.0 & 100.0 & 100.0 & 100.0 & 100.0 \\
        \bottomrule
    \end{tabular}
    \label{tab:convexity1}
\end{table}

\begin{table}[!ht]
\centering
\caption{Simulation results for $\widehat k_n$ based on Tests (i), (ii) and (iii). Entries show (mean, MAE) of $\widehat k_n$.}
\medskip
\begin{tabular}{llrrr}
\toprule
Model $\backslash$ Test & $n$ & (i) & (ii) & (iii) \\
\midrule
$\mathcal{G}(0,4,0.1)$   & 100    & (3.7, 0.3) & (3.7, 0.3) & (3.7, 0.3) \\
                          & 1000   & (3.8, 0.2) & (3.8, 0.2) & (3.8, 0.2) \\
$\mathcal{G}(0,9,0.1)$   & 100    & (8.2, 0.8) & (7.9, 1.1) & (8.1, 0.9) \\
                          & 1000   & (8.3, 0.7) & (8.1, 0.9) & (8.3, 0.7) \\
$\mathcal{G}(0,19,0.1)$  & 100    & (16.5, 2.5) & (15.7, 3.3) & (16.4, 2.6) \\
                          & 1000   & (17.0, 2.0) & (16.3, 2.7) & (17.0, 2.0) \\
$\mathcal{G}(0,39,0.1)$  & 100    & (29.8, 9.2) & (27.3, 11.7) & (29.8, 9.2) \\
                          & 1000   & (32.9, 6.1) & (30.4, 8.6) & (32.9, 6.1) \\
$\mathcal{G}(0,4,0.3)$   & 100    & (3.8, 0.2) & (3.8, 0.2) & (3.8, 0.2) \\
                          & 1000   & (3.9, 0.1) & (3.9, 0.1) & (3.9, 0.1) \\
$\mathcal{G}(0,9,0.3)$   & 100    & (8.0, 1.0) & (7.7, 1.3) & (8.0, 1.0) \\
                          & 1000   & (8.5, 0.5) & (8.3, 0.7) & (8.5, 0.5) \\
$\mathcal{G}(0,19,0.3)$  & 100    & (11.9, 7.1) & (11.2, 7.8) & (11.9, 7.1) \\
                          & 1000   & (15.7, 3.3) & (14.8, 4.2) & (15.7, 3.3) \\
$\mathcal{G}(0,39,0.3)$  & 100    & (12.5, 26.5) & (11.7, 27.3) & (12.5, 26.5) \\
                          & 1000   & (18.1, 20.9) & (16.8, 22.2) & (18.1, 20.9) \\
\bottomrule
\end{tabular}
\label{tab:k_selection_sim}
\end{table}

\subsection{Applications to real data}

We consider 5 real datasets to illustrate the applicability of the proposed testing procedures. Some datasets are taken from \cite{chee2016}, who proposed a method to estimate the number of species using discrete $k$-monotone models. Note that the authors also proposed a model selection approach to determine the parameter $k$. It must be pointed out that the definition of $k$-monotonicity used in \cite{chee2016} is quite different from the one adopted in this work or in \cite{GIGUELAY201896}. In fact, according to their definition, for a given value of $k$, a discrete distribution is said to be $k$-monotone if it is a mixture of discretised Beta distributions with parameters $1$ and $k$.

In the following, we report for each dataset a brief description and the cardinality of $\widehat {\mathcal{S}}_n= \{\min\{X_{1:n}\},\max\{X_{1:n}\}\}$:
\begin{itemize}
    \item[-] \lq\lq Accidents\rq\rq: This dataset contains 9461 counts ($|\widehat {\mathcal{S}}_n|=8$) of claims related to issued accident insurance policies. The dataset has been studied by \cite{chee2016}, and their analysis suggests an optimal value $k=9$.
    \item[-] \lq\lq Shakespeare\rq\rq: This dataset is about Shakespeare's vocabulary richness; see \cite{spevack68} for the original reference. It has been studied in \cite{bohning2017}, \cite{chee2016}, \cite{efron1976estimating}, \cite{GIGUELAY201896} and \cite{balabdaoui2020completely}. We use the dataset reported in Table 5 of \cite{chee2016}, which contains 30792 counts ($|\widehat {\mathcal{S}}_n|=100$) of words used by Shakespeare up to 100 occurrences, and for which the authors fitted a discrete $25$-monotone model. On the other hand, the results reported in \cite{GIGUELAY201896} suggest $k=6$. Finally, \cite{balabdaoui2020completely} successfully fitted a completely monotone distribution.
    \item[-] \lq\lq Stamboliyski\rq\rq:  This is a biodosimetry application based on the measurements of the biological response to radiation. In particular, exposure to radiation causes a certain (random) number of chromosome aberrations (generally dicentrics and/or rings) in the cells. The dataset, already used in \cite{puig2020some}, consists of 284 counts ($|\widehat {\mathcal{S}}_n|=8$) of aberrations (dicentrics and rings) from a patient exposed to radiation after the nuclear accident of Stamboliyski (Bulgaria) in 2011.
    \item[-] \lq\lq Tokai-mura\rq\rq: This is again a biodosimetry application from \cite{puig2020some}, and consists of 175 counts ($|\widehat {\mathcal{S}}_n|=10$) of dicentrics from a patient exposed to high doses of radiation caused by the nuclear accident that happened in Tokai-mura (Japan) in 1999.
    \item[-] \lq \lq Microbial\rq\rq:  This datatset contains 174 counts ($|\widehat {\mathcal{S}}_n|=65$) of observed species richness. Both \cite{bunge2008parametric} and \cite{chee2016} have studied it. In the latter paper, the authors report a selected monotonicity degree equal to $k=27$.
\end{itemize}
For all the datasets, we report the obtained test statistics and corresponding p-values in Tables~\ref{tab:data_mon}~and~\ref{tab:data_conv}.
For the Accidents and Shakespeare datasets, all the testing procedures show that there is no empirical evidence against either monotonicity or convexity. The results are in line with the results of \cite{chee2016}, who fit a 9-monotone model and a 25-monotone model, respectively. Furthermore, our results for the Shakespeare dataset are also in line with the value of $k=6$ suggested in \cite{GIGUELAY201896}.
The Stamboliyski dataset shows no evidence against either monotonicity or convexity, but convexity seems to be less likely to hold.
The Tokai-mura dataset shows strong evidence against both monotonicity and convexity.
Finally, the Microbial dataset shows no evidence at all against either monotonicity or convexity, which is also in line with \cite{chee2016} where a 27-monotone model was fitted.

Table \ref{tab:k_selection} reports the values of $\widehat k_n$ introduced in \eqref{eq:khat} for all five datasets and Tests (i), (ii) and (iii). It must be pointed out that $\widehat k_n$ reaches its possible maximum value $X_{(n)}-X_{(1)}$ for the datasets Accidents, Stamboliyski and Microbial, suggesting a high degree of $k$-monotonicity. On the other hand, the Words dataset satisfies a high degree of $k$-monotonicity, but it does not reach its maximum possible value of 99. Finally, $\widehat k _n=0 $ for the dataset Tokai-mura matching the results of Tables~\ref{tab:data_mon}~and~\ref{tab:data_conv}.

\begin{table}[!ht]
    \centering
    \caption{Monotonicity test statistics and the corresponding p-values (in parenthesis) for the 5 considered datasets.}
    \medskip
    \begin{tabular}{lrrrrr}
        \toprule
    Test $\backslash$ Dataset    &  Accidents & Shakespeare & Stamboliyski & Tokai-mura &  Microbial \\
        \midrule
        (i)                            & 0 (1.00) & -0.11 (1.00) & -0.06 (1.00) & -2.49 (0) & 0 (1) \\
        (ii)                           & 0 (1.00) & -0.11 (1.00) & -0.06 (1.00) & -2.49 (0) & 0 (1) \\
        (iii)                          & 0 (0.96) & -0.11 (0.99) & -0.06 (0.86) & -2.49 (0) & 0 (1) \\
        (iv)                           & 0 (0.96) &  0.19 (1.00) &  0.00 (0.96) &  2.11 (0) & 0 (1) \\
        \bottomrule
    \end{tabular}
    \label{tab:data_mon}
\end{table}

\begin{table}[!ht]
    \centering
    \caption{Convexity test statistics and the corresponding p-values (in parenthesis) for the 5 considered datasets.}
    \medskip
    \begin{tabular}{lrrrrr}
        \toprule
    Test $\backslash$ Dataset    &  Accidents & Shakespeare & Stamboliyski & Tokai-mura &  Microbial \\
        \midrule
        (i)                            & -0.03 (1.00) & -0.26 (0.99) & -1.19 (0.35) & -3.25 (0) & -0.15 (1) \\
        (ii)                           & -0.03 (0.98) & -0.26 (0.99) & -1.19 (0.35) & -3.25 (0) & -0.15 (1) \\
        (iii)                          & -0.03 (0.78) & -0.26 (0.97) & -1.19 (0.35) & -3.25 (0) & -0.15 (1) \\
        (iv)                           &  0.01 (0.83) &  0.31 (1.00) &  0.53 (0.53) &  2.78 (0) &  0.30 (1) \\
        \bottomrule
    \end{tabular}
    \label{tab:data_conv}
\end{table}

\begin{table}[!ht]
    \centering
    \caption{Values of $\widehat k_n$ for the 5 considered datasets.}
    \medskip
    \begin{tabular}{lrrrrr}
        \toprule
        Test $\backslash$ Dataset & Accidents & Words & Stamboliyski & Tokai-mura & Microbial \\
        \midrule
        (i)   & 7  & 90 & 7 & 0 & 64 \\
        (ii)  & 7  & 81 & 7 & 0 & 64 \\
        (iii) & 7  & 90 & 7 & 0 & 64 \\
        \bottomrule
    \end{tabular}
    \label{tab:k_selection}
\end{table}

\section{Proofs}
\label{sec:proofs}

\subsection*{Lemmas}
We first provide the following useful lemmas:
\begin{lemma}\label{lem:convquant}
Let $\mathcal S$ be a finite subset of  $\mathbb Z$.  Also, let  $\Sigma$ be a positive definite matrix $\in \mathbb R^{d \times d} $ and $\widetilde{\Sigma}_n$ a random and symmetric matrix $\in \mathbb R^{d \times d} $ based on i.i.d. random variables $W_1, \ldots, W_n$ or a triangular array $W_{n1}, \ldots, W_{nn}$  such that $\widetilde{\Sigma}_n \overset{P}{\to} \Sigma$. For any subset $I \subseteq \mathcal S$ and $\alpha \in (0,1)$, define $\widetilde t_{\alpha, n}(I)$ and  $t_\alpha(I)$ to be the $\alpha$-quantiles of the distribution of $g (\widetilde Z_n)$ and $g( Z)$ respectively,  where $g$ is continuous, $\widetilde Z_n := (\widetilde Z_{n,j})_{j \in \mathcal S} \sim \mathcal N(0, \widetilde \Sigma_{n})$  and $Z := (Z_j)_{j \in \mathcal S} \sim \mathcal N(0,\Sigma)$. Then, 
\begin{align*}
\widetilde t_{\alpha, n}(I)  \overset{P}{\to}  t_\alpha(I).
\end{align*}
\end{lemma}
\begin{proof}[Proof of Lemma \ref{lem:convquant}]  We can find a random vector $Y= (Y_j)_{j \in \mathcal S} \sim \mathcal N(0,I_{d\times d})$ which is independent of  $W_1, \ldots, W_n$ or $W_{n1}, \ldots, W_{nn}$ such that 
\begin{align*}
Z  \stackrel{d}{=}  \Sigma^{1/2}  Y, \ \ \text{and} \ \  \widetilde{Z}_n  \stackrel{d}{=}  \widetilde{\Sigma}^{1/2}_n  Y. 
\end{align*}
Using the well-known fact that the operator $M \mapsto M^{1/2}$ is continuous on the space of positive definite matrices and the assumption that $ \widetilde{\Sigma}_n \overset{P}{\to}  \Sigma$, it follows that  $\widetilde{\Sigma}^{1/2}_n  Y \overset{P}{\to} \Sigma^{1/2} Y$.
By continuity of $g$ it follows that
$$
g( \widetilde{\Sigma}^{1/2}_n  Y  )\overset{P}{\to}g( \Sigma^{1/2} Y )
$$
and therefore
$$
g (\widetilde Z_n) \overset{d}{\to} g  (Z).
$$
Let $\widetilde{F}_n$ and $F$ denote the distribution functions of $g (\widetilde Z_n)$ and $g  (Z)$ respectively. Since $F$ is continuous, convergence of $\widetilde F_n$ to $F$ must be uniform; see e.g. \cite[Chapter 1, Exercise 6]{ferguson2017}.  Since $F$ is invertible and $\widetilde{F}_n$ invertible with probability tending to 1 as $n \to \infty$, it follows that $\widetilde F_n^{-1}(\alpha) \overset{P}{\to} F^{-1}(\alpha) $, which is equivalent to the statement of the lemma.
\end{proof}

We will use the following result about $\ell_2$ projections of p.m.f.s.
\begin{lemma}\label{lemma:proj}
Let $\mathcal D$  denote either the set of all monotone non-increasing p.m.f.s or the set of all convex p.m.f.s with support $\mathcal S = \{m, \ldots, M \}$. For a given p.m.f.\ $p$ supported on $\mathcal S $, consider the minimisation problem of the functional $q \mapsto  \Vert  p - q \Vert_2$:
\begin{align}
\label{eq:min}
\min_{q\in \mathcal{D}} \Vert  p - q \Vert_2. 
\end{align}
Then:
\begin{enumerate}
\item The problem in \eqref{eq:min} admits a unique solution $ p^{\mathcal D}$.
\item Let $\widehat p^{\mathcal D}_n $ be the unique minimizer of $q \mapsto  \Vert   \widehat p_n -q \Vert_2$,  it holds that
$$\Vert \widehat p^{\mathcal D}_n -  p^{\mathcal D} \Vert_s  = O_{P}(1/\sqrt n)$$
for any $s \in [2, \infty]$.  
\end{enumerate}
\end{lemma}

\begin{proof}[Proof of Lemma \ref{lemma:proj}]
1. Let us equip the space $\mathcal D$ with the dot product.  Then, $\mathcal D$ is a convex and closed subset of  $\mathbb R^{M-m +1}$. Convexity is obvious. To show closedness, take a sequence $(q_m)_m \in \mathcal D$ which converges to $q^*$ in the $\ell_2$-norm. Then, for $j \in \mathcal S$, it holds  
$$\vert q_m(j) - q^*(j) \vert \le \Vert q_m - q^*\Vert_2,$$
and hence $\lim_{m \to \infty} q_m(j) =  q^*(j)$. This implies  that $q^*(j) \in [0,1]$ for all $j \in \mathcal S$ and $\sum_{j \in \mathcal S} q^*(j) = 1$. For $k \in \{1,2 \}$, we have that
$$
0 \le \lim_{m \to \infty} \nabla^k q_m(j)  = \nabla^k q^*(j)
$$
for all $j \in \mathcal S_k$, hence, $q^* \in \mathcal D$. Therefore, owing to the Projection Theorem on Hilbert spaces, it follows that a minimiser of $q \mapsto \Vert p - q \Vert_2 $ over $\mathcal D$ exists and is unique.  

\medskip 

\noindent 2. By the contraction property of projections and the Central Limit Theorem, it holds that
\begin{align*}
\Vert \widehat p^{\mathcal D}_n -  p^{\mathcal D}  \Vert_2 \le  \Vert \widehat p_n -  p \Vert_2 = O_P(1/\sqrt n).
\end{align*}
Recalling that  $\Vert w \Vert_s \le \Vert w \Vert_2$ for any $s \in [2, \infty]$ and sequence $w = (w(j))_j$ such that $ \Vert w \Vert_2 < \infty$, it follows that 
$$
\Vert \widehat p^{\mathcal D}_n -  p^{\mathcal D}  \Vert_s \le \Vert \widehat p^{\mathcal D}_n -  p^{\mathcal D}  \Vert_2 = O_P(1/\sqrt n),
$$
which concludes the proof.
\end{proof}

\subsection*{Proof of Theorem \ref{thm:main}}
\begin{proof}[Proof of Theorem \ref{thm:main}]
Note that the Strong Law of Large Numbers implies that, for $n$ large enough, $ \widehat{\mathcal  {S}}_{k,n} = \mathcal{S}_k$ with probability 1. Thus, without loss of generality and for simplicity of exposition, we replace $\widehat{\mathcal  {S}}_{k,n}$ by $\mathcal{S}_k$. As stated in the theorem, let $I_{\rho_k}$ be the set of indices $j$ such that 
$$
\nabla^k p(j) = \rho_k = \min_{j \in \mathcal{S}_k} \nabla^k p(j).
$$
Note that for all $j \notin I_{\rho_k}$ it holds that $\nabla^k p(j) >  \rho_k$.  By the Law of Large Numbers, we have
$$\nabla^k \widehat{p}_n(j)\overset{P}{\to} \nabla^k p(j)$$ 
for all $j \in \mathcal{S}_k$. Thus, it follows that for all $i \notin  I_{\rho_k} $  and $j \in I_{\rho_k}$ it holds 
$$P(\nabla^k \widehat{p}_n(i) >  \nabla^k \widehat{p}_n(j))\to 1 ,$$
as $n\to\infty$.
 Thus, with probability tending to 1, it holds
\begin{align*}
\sqrt n  (\min_{j \in \mathcal{S}_k}  \nabla^k \widehat p_n(j)  -  \min_{j \in \mathcal{S}_k}  \nabla^k p(j))  & =      \sqrt n  (\min_{j \in I_{\rho_k}}  \nabla^k \widehat p_n(j)  - \rho_k)  \\
& =   \min_{j \in I_{\rho_k}}    \sqrt n  (\nabla^k \widehat p_n(j)  - \rho_k)    \\
& =   \min_{j \in I_{\rho_k}}    \sqrt n  (\nabla^k \widehat p_n(j)  - \nabla^k p(j))  ,
\end{align*}
where we used that $\rho_k$ is the common value of $\nabla^k p(j)$ for all $j \in I_{\rho_k}$. Owing to the multivariate Central Limit Theorem applied to the random vector $ \sqrt n  (\nabla^k \widehat p_n(j)  - \nabla^k p(j))_{j \in \mathcal{S}_k}$ and using the Continuous Mapping Theorem, it follows that 
\begin{align*}
\min_{j \in I_{\rho_k}}  \sqrt n  (\nabla^k \widehat p_n(j)  - \nabla^k p(j))   \overset{d}{\rightarrow} \min_{j \in I_{\rho_k}}  Z_j,
 \end{align*}   
where $(Z_j)_{j \in \mathcal{S}_k} \sim \mathcal{N}({0}, \Sigma_k)$ and $\Sigma_k$ is the covariance matrix defined in the statement of the theorem. \end{proof}

\begin{proof}[Proof of Theorem \ref{thm:main} via the Generalized Delta Method]
 For some integer $s \ge 1$, let us consider the ordering function $\Gamma:\mathbb{R}^s\to\mathbb{R}^s$ given by
  \begin{align*}
  \Gamma(x_1,\ldots, x_s)  = (x_{(s)},\ldots, x_{(1)})
  \end{align*}
  where $\min_{1 \le i \le s} x_i =  x_{(1)} \le \ldots  \le  x_{(s)}  = \max_{1 \le i \le s} x_i  $.   Since we are interested in computing the quasi-differential of the function
  $$
  (x_1, \ldots, x_s)  \mapsto  x_{(1)}= \min_{1 \le i \le s} x_i
  $$
  we will consider the function $h: \mathbb{R}^s \to \mathbb{R}$ as
  $$(x_1\dots,x_s) \mapsto h(x_1,\ldots, x_s)  =  x_s.$$ 
  Let $A \subset \mathbb R^s$ and fix $\theta$ an accumulation point in $A$, then, the function $h$ is $A$-differentiable at $\theta$ (see Definition 1 of \citealp{marcheselli2000generalized}) since it is differentiable in the usual sense, and we denote by $H_\theta$ its $A$-differential at $\theta$.
  Now, let $v_1  > \cdots > v_r$ be the distinct values taken by $\theta_1,\ldots, \theta_s$, and for $i=1, \ldots, r$,  let $J_i  = \{j \in \{1, \ldots, s\}:  \theta_j =  v_i \}$ with cardinality $ | J_i |$.  For a given vector $x=(x_1, \ldots, x_s)  \in \mathbb R^s$ and for $i \in \{1, \ldots, r\}$ we denote by $x_{J_i}$ the sub-vector of $x$ whose components are indexed by $J_i$. For every $i=1,\dots,r$, let us define the function $\Gamma_i: \mathbb{R}^s \to \mathbb{R}^{|J_i|}$ as
\begin{align*}
x \mapsto \widetilde{\Gamma}_i(x)  :=  \Gamma(x_{J_i}), 
\end{align*}
 that is  $\widetilde{\Gamma}_i(x)$ is the vector of components of $x_{J_i}$ arranged in decreasing order. If $\widetilde{\Gamma}:\mathbb{R}^s \to \mathbb{R}^s$ denotes the function 
 $$
x \mapsto  \widetilde{\Gamma}(x):= (\widetilde{\Gamma}_1(x), \ldots, \widetilde{\Gamma}_r(x))  \in \mathbb R^s,$$
 then Theorem 2 of \cite{marcheselli2000generalized} implies that the (regular) quasi-differential of the function $(x_1, \ldots, x_s)  \mapsto  x_{(1)}$ at $\theta$  is given by
 \begin{align}
 \label{eq:tau}
  \begin{split}
\tau(x)  &= (H_{\Gamma(\theta)}\circ \widetilde \Gamma)(x) \\
&= \big\langle  \nabla h (\Gamma (\theta )),   \widetilde\Gamma(x) \big\rangle \\
&= \big\langle  (0,\ldots, 0, 1),   (\widetilde{\Gamma}_1(x), \ldots, \widetilde{\Gamma}_r(x))  \big\rangle  \\ 
&=  \min_{l \in J_r} x_l.
\end{split}
 \end{align}
In other words $\tau(x)$ gives the very last component of $\widetilde{\Gamma}(x)$ which is equal to the last component of  $\widetilde{\Gamma}_r(x)  =  \Gamma(x_{J_r})$, that is  $\min_{l \in J_r} x_l$.  Note that $J_r$  is the set of indices where the minimal value in $\{\theta_1, \ldots, \theta_s\}$ is attained.

Since $ \widehat{\mathcal  {S}}_{k,n} \to \mathcal{S}_k$ almost surely, without loss of generality and for simplicity of exposition, we replace $\widehat{\mathcal  {S}}_{k,n}$ by $\mathcal{S}_k$.
Let $   (Z_j)_{j  \in \mathcal{S}_k}\sim \mathcal{N}(0,\Sigma_k)$ where $\Sigma_k$ is the covariance matrix defined in the statement of the theorem, the classical multivariate Central Limit Theorem implies that 
$$\sqrt n  \big( \nabla^k \widehat p_n(j)  - \nabla^k p(j) \big)_{j\in \mathcal{S}_k}  \overset{d}{\to}   (Z_j)_{j  \in \mathcal{S}_k}. $$
Therefore, using the expression derived in \eqref{eq:tau} and the Generalized Delta Method introduced in Theorem 1 of \cite{marcheselli2000generalized} (with $a_n = \sqrt n$ and $C = \emptyset$) it follows that 
\begin{align*}
T_{k,n}=\sqrt n  (\min_{j \in \mathcal{S}_k} \nabla^k \widehat p_n(j)  - \rho_k )  \overset{d}{\to} \tau ( (Z_j)_{j  \in \mathcal{S}_k}) =  \min_{ j \in I_{\rho_k}}  Z_j,
\end{align*}
where $I_{\rho_k}$ is the set of indices in $\mathcal{S}_k$ such that $  \min_{j \in \mathcal{S}_k} \nabla^k p(j)= \rho_k $ is attained, and the proof is concluded. 
\end{proof}

\subsection*{Proof of Proposition \ref{prop:test-properties}}
\begin{proof}[Proof of Proposition \ref{prop:test-properties}]
1. First, recall that if $\{Y_n\}_n$ is a sequence of random variables which admits a weak limit, then $Y_n = O_{P}(1)$. In particular, for any deterministic sequence $(a_n)_n$ such that $\lim_{n \to \infty} a_n = \infty$ it holds that $ \lim_{n \to \infty} P(Y_n < -a_n) =0$ and $ \lim_{n \to \infty} P(Y_n < a_n) =1$.   Now, note that by Theorem \ref{thm:main}
$$
\widehat{T}_{k, n} = T_{k, n}  + \sqrt n \rho_{k}=O_{P}(1)+ \sqrt n \rho_{k},
$$
as $n\to\infty$.

If $\rho_k =0$, then $t_\alpha(I_0) = F^{-1}_{W_{I_0}}(\alpha)$ and $W_{I_0}$ is the weak limit of $T_{k, n}$. Since the minimum of 
independent continuous r.v.s is continuous, by Lemma \ref{lem:convquant} it follows that 
$$\lim_{n \to \infty}  P( \widehat{T}_{k, n} < \widehat t_{\alpha,n}(I_0))  =  P( W_{I_0} \leq t_\alpha(I_0))  = \alpha.$$ 
\medskip
\noindent
2. If $\rho_k > 0$, then, $I_0 = \emptyset$ and hence  $\widehat t_{\alpha,n}(I_0) = 0$ by definition. It follows that 
\begin{align*}
\lim_{n\to\infty}P(\widehat T_{k, n} < 0) =  \lim_{n\to\infty}P(T_{k, n}  < - \sqrt n  \rho_{k} )  = 0.
\end{align*}
\medskip
\noindent
3. If $p$ satisfies $H_1^{(k)}$, then $\rho_k < 0$.  Hence, for any integer $1 \le r \le \vert \mathcal H_k \vert$ and distinct values $h_1, \ldots, h_r \in \mathcal H_k $ we have that
\begin{align*}
\lim_{n\to\infty}P(\widehat T_{k, n} < \widehat t_{\alpha,n}( \cup_{j=1}^r I_{h_j}))  =   \lim_{n\to\infty}  P(T_{k, n}  <  \widehat t_{\alpha,n}(\cup_{j=1}^r I_{h_j}) + \sqrt n \vert \rho_{k} \vert ) = 1,
\end{align*}
which concludes the proof.
\end{proof}

\subsection*{Proof of Theorem \ref{thm:selection}}
\begin{proof}[Proof of Theorem \ref{thm:selection}]
In the proof of 1-3, we will use the Berry-Esseen Theorem, which gives a quantitative upper bound to the error in the Central Limit Theorem approximation. The statement of the theorem is as follows, see e.g.\ \cite{klenke2020probability}:

Let $Z_1, \ldots, Z_n$ be i.i.d.\ r.v.s with mean $\mu$ and variance $\sigma^2 > 0$, and absolute third centred moment $m = E[\vert Z_1 - \mu \vert^3] < \infty$, then
\begin{align}
\label{eq:berry-esseen}
 \sup_{x\in\mathbb{R}}| F_n(x)  - \Phi(x) | \leq   \frac{C m}{\sqrt n \sigma^{3}},   
\end{align}
where  $F_n(x)  =  P(\sqrt n (\Bar Z_n - \mu)/\sigma \le x)$, $\Phi$ the distribution function of the standard normal distribution and $C > 0$ is a universal constant. 
In the sequel, we will assume without generality that $m=1$ and hence $\mathcal S_k = \{1, \ldots, M- k \}$ where $M \ge k +1$.
\medskip

\noindent 1. 
Let us denote by $\widetilde{I}_{0,n}$ the set where we only replace $\widehat{\mathcal{S}}_{k,n}$   by $\mathcal{S}_k$ used in the definition of $\widehat{I}_{0,n}$. Since $P( \mathcal S_k^c \cap  \widehat{\mathcal{S}}_{k,n}\neq \emptyset)  =0 $, we have that 
\begin{align*}
P(\widehat I_{0,n} \ne \widetilde I_{0,n})  &\le    P( \widehat{\mathcal{S}}_{k,n}  \ne \mathcal S_k)  = P(\widehat{\mathcal{S}}_{k,n} \ \triangle \ \mathcal S_k\ne \emptyset)=  P( \mathcal S_k  \cap \widehat{\mathcal{S}}_{k,n}^c\neq \emptyset).
\end{align*}
Now, note that by the union bound, it follows
\begin{align*}
 P( \mathcal S_k  \cap \widehat{\mathcal{S}}_{k,n}^c\neq \emptyset) & = 
 P\big(\exists \ j \in \mathcal{S}_k: \widehat p_n(j) =0 \big)  \\
 &\le  \sum_{j \in  \mathcal{S}_k}  P(\widehat p_n(j) =0 )  \\
 &=   \sum_{j \in  \mathcal{S}_k}  P(\sqrt n (\widehat p_n(j) -  p(j)) = - \sqrt n p(j) )  \\
 & \le   \sum_{j \in  \mathcal{S}_k}  P\Big(\frac{\sqrt n (\widehat p_n(j) -  p(j))}{\sqrt{p(j) (1-p(j))}} \le - \frac{\sqrt n p(j)}{2 \sqrt{p(j) (1-p(j))}} \Big).
\end{align*}    
Let $F_{n,j}$ denote the distribution function of $\sqrt n (\widehat p_n(j) -  p(j))/\sqrt{p(j)(1-p(j))}$, then
\begin{align*}
 & \sum_{j \in  \mathcal{S}_k}  P\Big(\frac{\sqrt n (\widehat p_n(j) -  p(j))}{\sqrt{p(j) (1-p(j))}} \le - \frac{\sqrt n p(j)}{2 \sqrt{p(j) (1-p(j))}} \Big) \\
 & =   \sum_{j \in  \mathcal{S}_k} F_{n,j}\Big(- \frac{\sqrt n p(j)}{2 \sqrt{p(j) (1-p(j))}}\Big) \\
 & \le  \sum_{j \in  \mathcal{S}_k} \frac{C m_j}{\sqrt n (p(j)(1-p(j)))^{3/2}}  +   \sum_{j \in  \mathcal{S}_k} \Phi\Big(- \frac{\sqrt n p(j)}{2 \sqrt{p(j) (1-p(j))}}\Big), 
\end{align*}
where  we used \eqref{eq:berry-esseen} with $m_j =  \E[\vert \mathbf{1}_{\{X_1 = j\}} - p(j) \vert^3]$ and $C>0$. We conclude that 
 \begin{align*}
P(\widehat I_{0,n} \ne \widetilde I_{0,n})  & \le    \frac{C}{\sqrt n} \sum_{j \in  \mathcal{S}_k} \frac{m_j}{(p(j)(1-p(j)))^{3/2}}  +   \sum_{j \in  \mathcal{S}_k} \Big(1-\Phi\Big(\frac{\sqrt n p(j)}{2 \sqrt{p(j) (1-p(j))}}\Big) \Big) \\
& \le    \frac{C}{\sqrt n} \sum_{j \in  \mathcal{S}_k} \frac{m_j}{(p(j)(1-p(j)))^{3/2}} + \frac{4}{\sqrt n} \sum_{j \in  \mathcal{S}_k} \sqrt{\frac{1-p(j)}{p(j)}} \varphi\Big(\frac{\sqrt n p(j)}{2 \sqrt{p(j) (1-p(j))}}\Big) 
\end{align*}
for $n$ large enough, where we used the asymptotic behaviour of the Mill's ratio for the standard Gaussian; i.e., 
\begin{align}\label{eq:millsratio}
\frac{1-\Phi(x)}{\varphi(x)} \sim \frac{1}{x}, \ \ \text{ $x \to \infty$}
\end{align}
where $\varphi$ denotes the density of the standard Gaussian; see \cite{small2010expansions}. Since the term  
$$\varphi\Big(\frac{\sqrt n p(j)}{2 \sqrt{p(j) (1-p(j)}}\Big) = o(1)$$ 
and $\mathcal S_k$ is finite, it follows that 
$$
P(\widehat I_{0,n} \ne \widetilde I_{0,n})  = O(1/\sqrt n).
$$
Therefore, $P(\widehat{I}_{0,n}  \ne I_{0} )$ can be bounded as 
\begin{align*}
P(\widehat{I}_{0,n}  \ne I_{0} )  &=   P(  \widehat{I}_{0,n}  \ne I_0 \ \text{and} \  \widehat{I}_{0,n} = \widetilde{I}_{0,n} ) +    P(  \widehat{I}_{0,n}  \ne I_0 \ \text{and}  \  \widehat{I}_{0,n} \ne \widetilde{I}_{0,n} )  \\
& \le  P(  \widetilde{I}_{0,n}  \ne I_0)   +  P(\widehat{I}_{0,n} \ne \widetilde{I}_{0,n} )  \\
&  \le  P(  \widetilde{I}_{0,n}  \ne I_0)  +  O(1/\sqrt n).
\end{align*}
To ease notation let $\sigma^2_j:= (\Sigma_k)_{j,j}$ and $\widehat \sigma^2_{j,n}$ its consistent empirical estimator. It follows
\begin{align}
\begin{split}
\label{eq:probabilities}
P(  \widetilde{I}_{0,n}  \ne I_0)  &=   P( \exists \  j \in \mathcal{S}_k:  j \in   \widetilde{I}_{0,n} \cap I_0^c)  +  P( \exists \  j \in \mathcal{S}_k:  j \in   \widetilde{I}^c_{0,n} \cap I_0) \\
& =   P\Big( \exists \ j \in \mathcal{S}_k:   \frac{\sqrt{n}\nabla^k \widehat p_n(j)}{ \widehat \sigma_{j,n} }  \le z_{1-1/n} \  \text{and}  \  \nabla^k p(j) >  0 \Big) \\
&  \quad+   P\Big( \exists \ j  \in \mathcal{S}_k:   \frac{\sqrt{n} \nabla^k \widehat p_n(j)} { \widehat \sigma_{j,n} } > z_{1-1/n} \ \text{and}   \  \nabla^k p(j) =  0 \Big)
\end{split}
\end{align}
where in the first probability term we use the fact that $p$ is $k$-monotone, which implies that $\nabla^k p(j) \ge 0$  for all $j  \in \mathcal{S}_k$, thus, if $\nabla^k p(j) \ne 0$, then $ \nabla^k p(j) > 0$. Using the union bound and recalling that $\sigma^2_j$ is the asymptotic variance of $\sqrt{n} (\nabla^k \widehat p_n(j)  - \nabla^k p(j)) $, for $n$ large enough the first term of \eqref{eq:probabilities} can be bounded as
\begin{align}
\label{eq:first-prob-bound}
\begin{split}
 &P\Big( \exists \ j \in I^c_0:  \frac{\sqrt{n} \nabla^k \widehat p_n(j)}{\widehat \sigma_{j,n} }  \le z_{1-1/n}  \Big)\\
 &  \le \sum_{j \in I_0^c}    P\Big(   \frac{\sqrt{n} (\nabla^k \widehat p_n(j)  - \nabla^k p(j))}{ \widehat \sigma_{j,n} }  \le z_{1-1/n}  -  \frac{\sqrt n\nabla^k p(j) }{\widehat \sigma_{j,n}}  \Big)  \\
 &  =  \sum_{j \in I^c_0}    P\Big(   \frac{\sqrt{n}(\nabla^k \widehat p_n(j)  - \nabla^k p(j))}{ \sigma_j }  \le \frac{{\widehat \sigma_{j,n}}}{\sigma_j} \Big(z_{1-1/n}  -  \frac{\sqrt n\nabla^k p(j) }{\widehat \sigma_{j,n}} \Big)  \Big) \\
 & \le   \sum_{j \in I^c_0}    P\Big(   \frac{\sqrt{n}(\nabla^k \widehat p_n(j)  - \nabla^k p(j))}{ \sigma_j }  \le \frac{{\widehat \sigma_{j,n}}}{\sigma_j} \Big(z_{1-1/n}  -  \frac{\sqrt n\nabla^k p(j) }{\widehat \sigma_{j,n}} \Big) \ \ \text{and}  \ \ \widehat{\sigma}_{j, n} \le 2 \sigma_j  \Big)  \\
&  \  \  + \sum_{j \in I^c_0}   P(\widehat \sigma_{j,n}  > 2 \sigma_j   ) \\
 & \le  \sum_{j \in I^c_0}    P\Big(   \frac{\sqrt{n}(\nabla^k \widehat p_n(j)  - \nabla^k p(j))}{\sigma_j }  \le 2 z_{1-1/n}  -  \frac{\sqrt n\nabla^k p(j) }{\sigma_j}  \Big)  +  \sum_{j \in I^c_0}   P(\widehat \sigma_{j,n}  > 2 \sigma_j   ).
 \end{split}
\end{align}
Given $j \in I^c_0$, let us denote 
$$Z_{i,j}  = \nabla^k \mathbf{1}_{\{X_i = j\}}, \ \ \text{for} \ \ i=1, \ldots, n,$$
and let us consider  
$$\bar Z_j  = \frac{1}{n}\sum_{i=1}^n Z_{i,j}=  \nabla^k \widehat p_n(j) \quad \text{and} \quad \mu_j =  \nabla^k p(j) = \mathbb E[\nabla^k \mathbf{1}_{X_1 = j}].$$
Then, $\widehat{\sigma}_{j, n}^2$ can be rewritten as
\begin{align*}
\widehat{\sigma}_{j, n}^2 & =  \frac{1}{n} \sum_{i=1}^n \left(  Z_{i,j}  -\bar Z_j \right)^2  =   \frac{1}{n} \sum_{i=1}^n (  Z_{i,j}  - \mu_j )^2  - (\bar Z_j  - \mu_j)^2.
\end{align*}
As a consequence, for any $j \in I ^c_0$ and $n$ large enough, we have that
\begin{align*}
P(\widehat \sigma_{j,n}  > 2 \sigma_j    ) & =   P(\widehat \sigma^2_{j,n}  > 4 \sigma^2_j  )  \\
&=  P(\sqrt n (\widehat{\sigma}^2_{j,n} - \sigma^2_{j})  >  3  \sigma^2_{j} \sqrt n ) \\
& =  P \Big(\sqrt n \Big(\frac{1}{n} \sum_{i=1}^n ( Z_{i,j}  - \mu_j)^2 - \sigma^2_j  \Big) - \sqrt n (\bar Z_j  - \mu_j)^2  >  3  \sigma^2_{j} \sqrt n \Big)
\\
& \le P \Big(\sqrt n \Big(\frac{1}{n} \sum_{i=1}^n ( Z_{i,j}  - \mu_j)^2 - \sigma^2_j  \Big) > 3  \sigma^2_{j} \sqrt n  \Big) \\
& = |P(\sqrt n (\widehat{\sigma}^2_{j,n} - \sigma^2_{j})  > 3  \sigma^2_{j} \sqrt n/2 ) - (1-\Phi(3\sigma^2_j \sqrt{n}/2)) |+ 1-\Phi(3\sigma^2_j\sqrt{n} /2) \\
&\leq O(1/\sqrt{n})  +  \frac{ 2 \varphi(3\sigma^2_j\sqrt{n}/2)} {3\sigma^2_j\sqrt{n}/2} \\
& = O(1/\sqrt{n}),
\end{align*}
where we used \eqref{eq:berry-esseen} and \eqref{eq:millsratio}.
Thus, $\sum_{j \in I^c_0}   P(\widehat \sigma_{j,n}  > 2 \sigma_j   ) = O(1/\sqrt{n})$.
Now, we turn to the first probability term; i.e.\ we aim to show that
$$
 \sum_{j \in I^c_0}    P\Big(   \frac{\sqrt{n}(\nabla^k \widehat p_n(j)  - \nabla^k p(j))}{\sigma_j }  \le 2 z_{1-1/n}  -  \frac{\sqrt n\nabla^k p(j) }{\sigma_j}  \Big)  = O(1/\sqrt n).
$$
First, let us write $z_{1-1/n}=b_n$. Then, using \eqref{eq:millsratio} for $x = b_n$ it follows $1/n\sim\varphi(b_n)/b_n$, or equivalently $\log n \sim b_{n}^2/2$
and hence
\begin{align}
\label{eq:quantile-bound}
  z_{1-1/n} \sim \sqrt{2 \log n}.
\end{align}
For $n \to \infty$. Note that \eqref{eq:quantile-bound} implies that for $n$ large enough $z_{1-1/n} \le 2 \sqrt{2 \log n}$. Combining this with \eqref{eq:berry-esseen} and \eqref{eq:millsratio} yields for $n$ large enough
\begin{align*}
     & \sum_{j \in I^c_0}    P\Big(   \frac{\sqrt{n}(\nabla^k \widehat p_n(j)  - \nabla^k p(j))}{\sigma_j }  \le 2 z_{1-1/n}  -  \frac{\sqrt n\nabla^k p(j) }{\sigma_j}  \Big)  \\
    & \le O(1/\sqrt{n}) + \sum_{j \in  I_0^c} \Phi\Big(  2 z_{1-1/n}  -  \frac{\sqrt n\nabla^k p(j) }{\sigma_j}  \Big)\\
    &\le O(1/\sqrt{n}) + \sum_{j \in  I_0^c} \Phi\Big(  4 \sqrt{2 \log n}  -  \frac{\sqrt n\nabla^k p(j) }{\sigma_j}  \Big) \\
    &= O(1/\sqrt{n}),
\end{align*}
where in the last step we used the fact that for any $j \in I_0^c $ and for $n$ large enough
\begin{align*}
\Phi\Big(  4 \sqrt{2 \log n}  -  \frac{\sqrt n\nabla^k p(j) }{\sigma_j}  \Big) &=   1 -  \Phi\Big( \frac{\sqrt n\nabla^k p(j) }{\sigma_j} -4 \sqrt{2 \log n}   \Big) \\
& \le   \frac{2 \varphi\Big( \frac{\sqrt n\nabla^k p(j) }{\sigma_j} -4 \sqrt{2 \log n}   \Big) }{\frac{\sqrt n\nabla^k p(j) }{\sigma_j} -4 \sqrt{2 \log n}  }  =  o(1/\sqrt n).
\end{align*}
We conclude that the first probability term in \eqref{eq:probabilities} is $O(1/\sqrt n)$.  Now, we handle the second probability term of \eqref{eq:probabilities}. Let $\eta \in (0,1)$ to be chosen later, for $n$ large we have that
\begin{align*}
&P\Big( \exists \ j  \in \mathcal{S}_k:   \frac{\sqrt{n} \nabla^k \widehat p_n(j)} { \widehat \sigma_{j,n} } > z_{1-1/n} \ \text{and}   \  \nabla^k p(j) =  0 \Big)  \\
& \le \sum_{j \in I_0}  P\Big( \frac{ \sqrt{n} \nabla^k \widehat p_n(j)} { \sigma_j } > z_{1-1/n}\frac{\widehat\sigma_{j,n}}{\sigma_j}\Big)  \\
&  \le \sum_{j \in I_0}  P\Big( \frac{ \sqrt{n} \nabla^k \widehat p_n(j)} { \sigma_j } > \eta z_{1-1/n} \Big)  +  P(\widehat \sigma_{j,n} < \eta \sigma_j ) \\
& \le   \sum_{j \in I_0}  \frac{C m_j}{\sigma_j^3\sqrt n}  +   |\mathcal{S}_k | (1 -  \Phi(\eta z_{1-1/n}) ) +  P(\widehat \sigma_{j,n} < \eta \sigma_j )
\end{align*}
using \eqref{eq:berry-esseen} with  $m_j =  \E [ \vert \nabla^k \mathbf{1}_{\{X_1 =j\}}  - \nabla^k p(j) \vert^3] $. Using again \eqref{eq:quantile-bound}  we have that 
$z_{1-1/n}  \ge \eta \sqrt{2 \log n}$ for $n$ large enough.  Then, by \eqref{eq:millsratio} for $n$ large enough
\begin{align*}
1 - \Phi(\eta z_{1-1/n}) & \le    1 -  \Phi(\eta^2 \sqrt{2\log n}) \\
& \le     \frac{2\varphi( \eta^2\sqrt{2 \log n})}{\eta^2 \sqrt{2\log n}} \\
& =      \frac{2}{\eta^2 \sqrt{2\log n}} \frac{1}{\sqrt{2\pi}}   \exp(- \eta^4\log n )  \\
& =   \frac{2}{\eta^2 \sqrt{2\log n}}  \frac{1}{\sqrt{2\pi}} \frac{1}{n^{\eta^4}}  = o(1/\sqrt n)
\end{align*}
provided that $\eta \in [2^{-1/4}, 1)$.   Moreover, we have that 
\begin{align*}
& P(\widehat \sigma_{j,n} < \eta \sigma_j )  \\
&=P(\widehat \sigma_{j,n}^2 -\sigma_j^2 < \eta^2 \sigma_j^2-\sigma_j^2)\\
& = P(\sqrt n(\widehat \sigma^2_{j,n}  - \sigma_j^2) < -\sqrt n (1-\eta^2) \sigma^2_j ) \\
& \le  P\Big(\sqrt n \Big(\frac{1}{n} \sum_{i=1}^n ( Z_{i,j}  - \mu_j)^2 - \sigma^2_j  \Big) - \sqrt n (\bar Z_j  - \mu_j)^2 < -\sqrt n (1-\eta^2) \sigma^2_j \Big) \\
& \le P\Big(\sqrt n \Big(\frac{1}{n} \sum_{i=1}^n ( Z_{i,j}  - \mu_j)^2 - \sigma^2_j  \Big) < -\sqrt n (1-\eta^2) \sigma^2_j/2\Big) \\
&  \ \  + \  P\Big(- \sqrt n (\bar Z_j  - \mu_j)^2 < -\sqrt n (1-\eta^2) \sigma_j^2/2 \Big).  
\end{align*}
The first term can be shown to be $O(1/\sqrt n)$ using similar arguments as done above for bounding $P(\widehat \sigma_{j,n} > 2 \sigma_j )$. As for the second term, we can write that 
\begin{align*}
& P\Big(- \sqrt n (\bar Z_j  - \mu_j)^2 < -\sqrt n (1-\eta^2) \sigma_j^2/2 \Big) \\
& =   P\Big((\bar Z_j  - \mu_j)^2 > (1-\eta^2) \sigma_j^2/2 \Big)   \\
& \le   P\Big(\sqrt n \vert \bar Z_j  - \mu_j \vert > \sigma_j \sqrt{1-\eta^2} / \sqrt 2 \Big)   \\
& \le   2 \ | P(\sqrt n (\bar{Z}_j - \mu_j) > \sigma_j \sqrt{1-\eta^2} / \sqrt 2 )  - (1- \Phi(\sigma_j \sqrt{1-\eta^2} / \sqrt 2 ) | \\
& \ \ +   \ 2 \Big(1  -  \Phi(\sigma_j \sqrt{1-\eta^2} \sqrt n / \sqrt 2 )\Big)      \\
&\leq  O(1/\sqrt{n})  +   \frac{4 \sqrt 2 \varphi( \sigma_{j} \sqrt{1-\eta^2} \sqrt n /\sqrt 2)}{\sigma_{j} \sqrt{1-\eta^2} \sqrt n } \\
& =  O(1/\sqrt{n})  + o(1/\sqrt n)  =  O(1/\sqrt n).  
\end{align*} 
Therefore, merging all the pieces, we conclude that 
 $$
 P(\widehat{I}_{0,n}  \ne I_0)  =  O(1/\sqrt n). 
 $$
 
\medskip
\noindent 2. The aim is to show that $P(\widehat I_{0,n} \ne \emptyset ) = O(1/\sqrt n)$. Using the same argument as in point 1, we replace $\widehat S_{k,n}$ by $\mathcal S_k$ in the definition of $\widehat I_{0,n}$ and use the resulting set $\widetilde I_{0,n}$ at the cost of a further error term of order $1/\sqrt n$.  Thus, we aim to show that $$\lim_{n \to \infty}  P(\widetilde I_{0,n} \ne \emptyset) = O(1/\sqrt n).$$ 
Similarly to 1, it can be shown that $P(\widehat \sigma_{j,n} > 2 \sigma_j) = O(1/\sqrt n)$ and hence 
\begin{align*}
P(\widetilde I_{0,n} \ne \emptyset) &  \le   \sum_{j \in \mathcal S_k}  P\Big(   \frac{\sqrt{n}\nabla^k \widehat p_n(j)}{ \widehat \sigma_{j,n} }  \le z_{1-1/n}   \Big)   \\
& =  \sum_{j \in \mathcal S_k}  P\Big(   \frac{\sqrt{n}(\nabla^k \widehat p_n(j)  - \nabla^k p(j))}{ \sigma_j }  \le 2 z_{1-1/n}  - \frac{\sqrt n \nabla^k p(j)}{\sigma_j} \Big) + O(1/\sqrt n)  \\
& \le  \sum_{j \in \mathcal S_k}  P\Big(  \frac{ \sqrt{n}(\nabla^k \widehat p_n(j)  - \nabla^k p(j))}{ \sigma_j }  \le   4 \sqrt{2 \log n }   -  \frac{\sqrt n \nabla^k p(j)}{\sigma_j} \Big)  + O(1/\sqrt n)
\end{align*}
using \eqref{eq:quantile-bound}. 
By the Berry-Esseen quantitative bound in 
 \eqref{eq:berry-esseen} and using \eqref{eq:millsratio})  it follows that 
 \begin{align*}
P(\widetilde I_{0,n} \ne \emptyset)  & \le O(1/\sqrt{n})
+ \sum_{j \in \mathcal S_k}  \Big(1- \Phi\Big( \frac{\sqrt n \nabla^k p(j)}{\sigma_j} -4 \sqrt{2\log n} \Big)\Big) + O(1/\sqrt n) \\
& \le  O(1/\sqrt{n})  + 2 \sum_{j \in \mathcal S_k} \frac{\varphi\Big(\frac{\sqrt n \nabla^k p(j)}{\sigma_j} - 4 \sqrt{2\log n} \Big)}{\frac{\sqrt n \nabla^k p(j)}{\sigma_j} - 4 \sqrt{2\log n}}  + O(1/\sqrt n),   \\
& = O(1/\sqrt{n})
\end{align*}
which completes the proof of 2.

\medskip
\noindent 3. In the third case, the true p.m.f.\ $p$ is such that $\rho_k < 0$. Let us consider the set
$$
\mathcal{S}_k^-=\big \{j \in \mathcal S_k: \nabla p^k(j) \le 0 \big \}. 
$$
Note that $\mathcal{S}_k^-$ is finite and non-empty. It can be rewritten as 
$$
\mathcal{S}_k^-= I_-  \cup I_0 = \bigcup_{j\in\mathcal H_k} I_{h_j}  \cup I_0 \neq \emptyset
$$
where 
$$ \mathcal H^-_k  =   \{h \in [\rho_k, 0): \exists \ j \in \mathcal S_k \ \text{such that} \ \nabla^k p(j)=h \}.$$
Thus, for some integer $r \ge 1$, 
$$
\mathcal{H}^-_k=\{h_1, \ldots, h_r \} =   \{h\in \mathbb{R}: \nabla^k p(j)=h, \  j \in \mathcal S_k  \} \cap (-\infty, 0)  \ne \emptyset,
$$
since it contains at least $\rho_k$.
The set $I_0$ might be empty, with no consequence on the reasoning. Now, we will show that as $n\to\infty$,
\begin{align*}
 P(\widehat I_{0,n}  \ne I_- \cup I_0 )
 &= P((\widehat I_{0,n}^c \cup (I_- \cup I_0))\cup (\widehat I_{0,n}\cap (I_- \cup I_0)^c)\ne \emptyset)  \\
 &\leq P(\widehat I_{0,n}^c \cap (I_- \cup I_0) \ne \emptyset) + P(\widehat I_{0,n}\cap (I_- \cup I_0)^c\ne \emptyset) \\ 
 &= O(1/\sqrt n).
\end{align*}
As in the proof of 1 and 2, it is enough to show that the same statement holds with $\widetilde I_{0,n}$ instead of $\widehat I_{0,n}$ at the cost of an additional error of order $1/\sqrt n$. Suppose for simplicity that $(I_- \cup I_0)^c  \ne \emptyset$, and consider $j  \in \mathcal{S}_k$ such that $j \in \widetilde{I}_{0, n} \cap (I_- \cup I_0)^c$.  This means that for such $j$ it holds
\begin{align*}
  \frac{\sqrt n\nabla^k \widehat p_n(j)}{\widehat \sigma_{j,n}}  \le z_{1-1/n}, \ \ \text{and}  \ \  \nabla^k p(j) > 0.
\end{align*}
Let $\mathcal S^+_k=\{j\in \mathcal{S}_k:\nabla^k p(j) > 0\}$. By the union bound and similar arguments used for proving 2, we get
\begin{align*}
 P(\widetilde{I}_{0,n} \cap (I_- \cup I_0)^c\neq \emptyset) &\le \sum_{j\in\mathcal{S}^+_k}P\Big(  \frac{\sqrt{n} (\nabla^k \widehat p_n(j)  - \nabla^k p(j))}{ \sigma_j }  \le 2 z_{1-1/n}  - \frac{\sqrt n \nabla^k p(j)}{\sigma_j}\Big ) + O(1/\sqrt n) \\
 & =   O(1/\sqrt n).
\end{align*}
Note that the case $(I_- \cup I_0)^c = \emptyset$ is trivial because it implies $ P(\widetilde{I}_{0,n} \cap (I_- \cup I_0)^c\neq \emptyset) =0$.    Consider now  $j \in \widetilde I^c_{0, n}  \cap (I_- \cup I_0)$. Then, for all such $j$'s, it holds
\begin{align*}
  \frac{\sqrt n\nabla^k \widehat p_n(j)}{ \widehat \sigma_{j,n} }  >  z_{1-1/n}, \  \text{and}  \  \nabla^k p(j) \le 0.   
\end{align*}
Fix $\eta \in (0,1)$ and let $n$ be large enough so that 
$$
P(\widehat\sigma_{j,n}  < \eta \sigma_j ) = O(1/\sqrt n).
$$
Since $\nabla^k p(j) \le 0$, similar arguments as in 2 yield 
\begin{align*}
P(\widetilde{I_0}^c  \cap (I_- \cup I_0)\neq \emptyset)&\le \sum_{j\in\mathcal{S}^-_k}
P\Big(  \frac{\sqrt n(\nabla^k \widehat p_n(j)-\nabla^k p(j))}{\sigma_j }  >  \eta z_{1-1/n} -  \frac{\sqrt n \nabla^k p(j)}{\sigma_j}\Big)  + O(1/\sqrt n) \\
& \le O(1/\sqrt n)  + \sum_{j\in\mathcal{S}^-_k}
P\Big(  \frac{\sqrt n(\nabla^k \widehat p_n(j)-\nabla^k p(j))}{\sigma_j }  >  \eta z_{1-1/n} \Big) \\
& \le O(1/\sqrt n) + \sum_{j\in\mathcal{S}^-_k} (1  - \Phi(\eta z_{1-1/n})  ) =  O(1/\sqrt n),
\end{align*}
using the same arguments as done above and provided that $\eta \in [2^{-1/4}, 1)$.
Therefore, we conclude that
$$
P(\widehat I_{0,n}  \ne I_- \cup I_0 )  = O(1/\sqrt n),
$$
which completes the proof.
\end{proof}

\subsection*{Proof of Proposition \ref{prop:selectionk}}
\begin{proof}[Proof of Proposition \ref{prop:selectionk}]
Let $\widehat  {\mathcal S}_n$ and $\mathcal S $ be the empirical and true support, and similar arguments as in the proof of Theorem \ref{thm:selection} imply that $P(\widehat  {\mathcal S}_n \ne \mathcal S)  = O(1/\sqrt n) $. Hence, at the cost of an additive error or order $1/\sqrt n$, we can replace $X_{(n)} - X_{(1)}$ with $M-m$ in the definition of $\widehat k_n$. 

Let us start with the case $k_0 =0$.  Note that $\widehat k_n \ne 0$ implies that  $1$-monotonicity is not rejected while it does not hold.   Therefore, 
\begin{align*}
P(\widehat k_n \ne 0)  & \le   P_{H^{(1)}_1}(X_{1:n} \notin C_{1,n}(\alpha))  + O(1/\sqrt n) \\
& =  1- P_{H^{(1)}_1}(X_{1:n} \in C_{1,n}(\alpha))  + O(1/\sqrt n).
\end{align*}
From point 3 of Corollary \ref{cor:rejection-region} applied to $k=1$, it follows  that  $\lim_{n \to \infty} P(\widehat k_n \ne 0) = 0$.  

Now, we consider the case $k_0 > 0$.  We have that 
\begin{align*}
P(\widehat k_n \ne k_0 )   =   P(\widehat k_n < k_0 )  + P(\widehat k_n >  k_0 ).
\end{align*}
By definition of $\widehat k_n$ and $k_0$ the inequality $\widehat k_n < k_0 $ implies that $(\widehat k_n + 1)$-monotonicity of $p$ is rejected while it is true. Hence,
\begin{align*}
\limsup_{n \to \infty} P(\widehat k_n < k_0 ) &  \le    \limsup_{n \to \infty} \sum_{ \widehat k_n + 1 \le k \le k_0} P_{H^{(k)}_0} (X_{1:n} \in C_{k,n}(\alpha)) \\
&\le     \limsup_{n \to \infty} \sum_{ 1 \le k \le k_0} P_{H^{(k)}_0} (X_{1:n} \in C_{k,n}(\alpha)) \le \alpha
\end{align*}
by points 1 and 2 of Corollary \ref{cor:rejection-region}.  Moreover,  $\widehat k_n > k_0$  implies that we do not reject $(k_0+1)$-monotonicity of $p$ while it does not hold.   Thus,
\begin{align*}
P(\widehat k_n > k_0 )  & \le     P_{H^{(k_0+1)}_1} ( X_{1:n}  \notin C_{k_0+1, n}(\alpha) )  + O(1/\sqrt n)  \\
& =   1 - P_{H^{(k_0+1)}_1} ( X_{1:n} \in C_{k_0+1, n}(\alpha)) +   O(1/\sqrt n)
\end{align*}
and hence $\lim_{n \to \infty} P(\widehat k_n > k_0 ) = 0$, using again point 3 of Corollary \ref{cor:rejection-region} applied for $k = k_0 +1$.
\end{proof}

\subsection*{Proof of Theorem \ref{thm:contig}}
\begin{proof}[Proof of Theorem \ref{thm:contig}]
 1. Without loss of generality suppose $\mathcal{S}=\{1,\dots,M\}$. Note that 
  $\widehat {\mathcal{S}}_{k,n}=\mathcal{S}_k$ with probability 1 for $n$ large enough. Indeed, Theorem 1 from \cite{hu1989strong} implies that for a triangular array of i.i.d.\ random variables $Z_{n,1}, \ldots, Z_{n, n}$ with mean 0 and $\E[\vert Z_{1,1} \vert^{2s} ] < \infty$ for some $1 \le s < 2$, 
  $$
  \frac{1}{n^{1/s}} \sum_{j=1}^n  Z_{n,j}  \overset{\text{a.s.}}{\to} 0.
  $$
  Fix $j \in \mathcal S$ and take $Z_{n, i} =  \mathbf{1}_{\{X_{n, i} = j\}} -  p_n(j)$. Since  $\E[\vert Z_{1,1} \vert^{2} ] = p_1(j) (1-p_1(j)) \le 1 < \infty$, it follows that 
  $$
  \frac{1}{n} \sum_{j=1}^n (\mathbf{1}_{\{X_{n, i} = j\}} -  p_n(j) ) = \widehat p_n(j) -  p_n(j) \overset{\text{a.s.}}{\to} 0.
  $$
  Therefore, for all $j \in \mathcal S$, $\widehat p_n(j) \overset{\text{a.s.}}{\to} p(j) > 0$ and hence all the elements of $\mathcal S$ are observed with probability 1 for $n$ large enough.  This also implies that $\widehat S_{n, k} = \mathcal S_k$ with probability 1  for $n$ large enough.
  
Now note that for $j \in \mathcal S_k$, we have that
\begin{align*}
\nabla^k p_n(j)  = \nabla^k p(j)  +  \frac{1}{\sqrt n} \nabla^k h(j). 
\end{align*}
Let $I_+: = \{j \in \mathcal S_k: \nabla^k p(j) > 0 \}$ and note that if $j \in I_+$,  
\begin{align*}
\nabla^k p_n(j)  \ge \frac{\nabla^k p(j)}{2}  > 0
\end{align*}
for $n$ large enough. 
If $\nabla^k p(j) = 0$; i.e., $j \in I_0$,  then $\nabla^k p_n(j)  =   \nabla^k h(j)/\sqrt{n}$,  therefore, for all $j \in I_0$
$$
\lim_{n \to \infty}  \nabla^k p_n(j)  =  0.
$$
This implies that for $n$ large enough the value of $\min_{j \in \mathcal S_k} \nabla^k p_n(j)$ is reached in the set $I_0$.  Let $I^\delta$ be the same set defined above.  Then, for $n$ large enough, it holds that
\begin{align*}
\min_{j\in \mathcal{S}_k} \nabla^k  p_n(j)  =   \min_{j\in I^\delta} \nabla^k  p_n(j)  =  -\frac{\delta}{\sqrt n}.  
\end{align*}
Now, we study the weak convergence of the random vector $\sqrt n (\nabla^k \widehat p_n(j) - \nabla^k p_n(j))_{j \in \mathcal S_k}$.  
For $i \in \{1, \ldots, n \}$, let us define the vector
\begin{align}\label{Y}
Y_{n, i} =  \frac{1}{\sqrt n} (\nabla^k \mathbf{1}_{\{X_{n, i} = j\}} -  \nabla^k p_n(j))_{j \in \mathcal S_k}. \end{align}
For $\epsilon > 0$ we will show that
\begin{align}\label{eq:Cond1LF}
\lim_{n \to \infty} \sum_{i=1}^n \E[\Vert Y_{n, i} \Vert^2 \mathbf{1}_{\Vert Y_{n, i}  \Vert > \epsilon}  ] =0.
\end{align}
In fact, we have that 
\begin{align*}
\Vert Y_{n, i} \Vert^2  =  \frac{1}{n} \sum_{j \in \mathcal S_k }  (\nabla^k \mathbf{1}_{X_{n, i} = j} -  \nabla^k p_n(j))^2
\end{align*}
and hence $\Vert Y_{n, i} \Vert > \epsilon$ is equivalent to $$\sum_{j \in \mathcal S_k }  (\nabla^k \mathbf{1}_{\{X_{n, i} = j\}} -  \nabla^k p_n(j))^2 > n \epsilon^2.$$
Recall that for any real function $q$ defined in $\mathcal S$, we have that 
$$
\nabla^k q(j) = \sum_{l=0}^k (-1)^l \binom{k}{l} q(j + l).
$$
Thus, if $\Vert q \Vert_\infty \le 1$, then 
$$
\vert \nabla^k q(j) \vert \le \sum_{l=0}^k \binom{k}{l} = 2^k.
$$
Letting 
$q(j) =  \mathbf{1}_{\{X_{n, i} = j\}}  -  p_n(j),$
it holds that 
\begin{align}
\label{eq:majsup}
\sup_{j\in \mathcal{S}_k} \vert \nabla^k (\mathbf{1}_{\{X_{n, i} = j\}} - p_n(j))  \vert^2 \le 2^{2k}.
\end{align}
Then, $\Vert Y_{n, i} \Vert > \epsilon$ implies that 
\begin{align*}
\vert \mathcal S_k \vert  2^{2k} \ge \vert \mathcal S_k \vert \sup_{j \in \mathcal S_k}  | \nabla^k \mathbf{1}_{\{X_{n, i} = j\}} -  \nabla^k p_n(j) | ^2  > n \epsilon^2
\end{align*}  
which is impossible if $n \ge n_0 := \vert \mathcal S_k  \vert 4^{k} /\epsilon^2$.   Hence, for all $n \ge n_0$, we have that 
$$
\Vert Y_{n, i} \Vert^2 \mathbf{1}_{\Vert Y_{n, i}  \Vert > \epsilon} =0
$$
and the condition in \eqref{eq:Cond1LF} is satisfied.  Next, we show 
\begin{align}\label{eq:Cond2LF}
\sum_{i=1}^n \text{Cov}(Y_{n, i} ) \to \Sigma_k.
\end{align}
For $r,s \in \{1,\dots, M-k-2\}$ we have that 
\begin{align*}
\sum_{i=1}^n (\cov[Y_{n, i}])_{r, s} & =  \frac{1}{n} \sum_{i=1}^n \cov\Big[\nabla^k \mathbf{1}_{\{X_{n, i} = r\}} -  \nabla^k p_n(r), \nabla^k \mathbf{1}_{\{X_{n, i} = s\}} -  \nabla^k p_n(s) \Big]  \\
& =  \cov\Big[\nabla^k \mathbf{1}_{\{X_{n, 1} = r\}} -  \nabla^k p_n(r), \nabla^k \mathbf{1}_{\{X_{n, 1} = s\}} -  \nabla^k p_n(s) \Big] \\
& =  \E [\nabla^k \mathbf{1}_{\{X_{n, 1} = r\}} \nabla^k \mathbf{1}_{\{X_{n, 1} = s\}}]  -  \nabla^k p_n(r) \nabla^k p_n(s),
\end{align*}
where 
\begin{align*}
\nabla^k p_n(r) \nabla^k p_n(s)  & =     \nabla^k p(r) \nabla^k p(s)  +  \frac{1}{\sqrt n} \Big( \nabla^k p(r) \nabla^k h(s)  +  \nabla^k p(s) \nabla^k h(r) \Big)  +  \frac{1}{n}   \nabla^k h(r) \nabla^k h(s) \\
& \to   \nabla^k p(r) \nabla^k p(s) =  \E[\mathbf{1}_{\{X_{1} = r\}} ] \E[\mathbf{1}_{\{X_{1} = s\}}]
\end{align*}
as $n \to \infty$, with $X_1 \sim p$. Assume without loss of generality that $r \ge s$, using the fact that 
$$
\nabla^k \mathbf{1}_{\{X_{n, 1} = j\}}  = \sum_{l=0}^k (-1)^l \binom{k}{l} \mathbf{1}_{\{X_{n, 1} = j+l \}}
$$
we can write 
\begin{align*}
&\E [\nabla^k \mathbf{1}_{\{X_{n, 1} = r\}} \nabla^k \mathbf{1}_{\{X_{n, 1} = s\}}]  \\
&  =\E \Big[\Big(\sum_{l=0}^k (-1)^l \binom{k}{l} \mathbf{1}_{\{X_{n, 1} = l + r\}} \Big)  \Big(\sum_{l'=0}^k (-1)^{l'} \binom{k}{l'} \mathbf{1}_{\{X_{n, 1} = l' + s\}} \Big)    \Big]  \\
&  =  \sum_{l, l'=0}^k (-1)^{l+l'} \binom{k}{l} \binom{k}{l'} \E[\mathbf{1}_{\{X_{n, 1} = l + r\}} \mathbf{1}_{\{X_{n, 1} = l' + s\}}] \\
&  =  \sum_{l=0}^k (-1)^{2l+ r -s} \binom{k}{l} \binom{k}{l + r- s} p_n(l + r)  \\
& =  \sum_{l=0}^k (-1)^{r -s} \binom{k}{l} \binom{k}{l + r- s} p_n(l + r)
\end{align*}
where we used the fact that $\E[\mathbf{1}_{\{X_{n, 1} = l + r\}} \mathbf{1}_{\{X_{n, 1} = l' + s\}}]  = p_n(l +s)$ if $l' =  l + r -s$ and 0 otherwise.  It follows that
\begin{align*}
\E [\nabla^k \mathbf{1}_{\{X_{n, 1} = r\}} \nabla^k \mathbf{1}_{\{X_{n, 1} = s\}}]    \to  \sum_{l=0}^k (-1)^{r -s} \binom{k}{l} \binom{k}{l + r- s} p(l + r)
\end{align*}
as $n \to \infty$. Now, by the same calculations above, we can show that 
$$
\sum_{l=0}^k (-1)^{r -s} \binom{k}{l} \binom{k}{l + r- s} p(l + r) =  \E [\nabla^k \mathbf{1}_{\{X_{1} = r\}} \nabla^k \mathbf{1}_{\{X_{1} = s\}}].
$$
We conclude that
\begin{align}\label{LindCond}
\sum_{i=1}^n (\cov[Y_{n, i}])_{r, s}  \to  \E [\nabla^k \mathbf{1}_{\{X_{1} = r\}} \nabla^k \mathbf{1}_{\{X_{1} = s\}}] -  \E [\nabla^k \mathbf{1}_{\{X_{1} = r\}}] \E [\nabla^k \mathbf{1}_{\{X_{1} = s\}}]  = (\Sigma_{k})_{r, s}  
\end{align}
which shows that the condition in (\ref{eq:Cond2LF}) is fulfilled, and by the Lindeberg-Feller Central Limit Theorem (see Proposition 2.27 in \citealp{van2000asymptotic}) we have
$$
\sqrt n (\nabla^k \widehat p_n(j) - \nabla^k p_n(j))_{j \in \mathcal S_k} \overset{d}{\to} Z \sim \mathcal{N}(0,\Sigma_k).
$$
Now, we can write that
\begin{align*}
\sqrt n (\widehat \rho_{k, n} - \min_{j\in \mathcal{S}_k} \nabla^k  p_n(j) )  &=     \sqrt n \Big( \min_{j \in \mathcal S_k} \nabla^k \widehat p_n(j) + \frac{\delta}{\sqrt n}  \Big) \\
& =   \sqrt n  \min_{j \in \mathcal S_k} \Big( \nabla^k \widehat p_n(j) + \frac{\delta}{\sqrt n}  \Big) \\
& =  \sqrt n  \min_{j \in I_0} \Big (\nabla^k \widehat p_n(j) + \frac{\delta}{\sqrt n}  \Big) \wedge  \sqrt n \min_{j \in I_+} \Big(\nabla^k \widehat p_n(j) + \frac{\delta}{\sqrt n}  \Big) \\
& =   \min_{j \in  I_0}  \Big( \sqrt n(\nabla^k \widehat p_n(j) - \nabla^k p_n(j) )  +
 \nabla^k h(j) + \delta \Big) \\
&\quad \wedge  \min_{j \in  I_+}  \Big( \sqrt n(\nabla^k \widehat p_n(j) - \nabla^k p_n(j)) +  \sqrt n \nabla^k p_n(j) + \delta \Big) \\
&=  \min_{j \in  I_0}  \Big( \sqrt n(\nabla^k \widehat p_n(j) - \nabla^k p_n(j) )  +
 \nabla^k h(j) + \delta \Big)
\end{align*}
where in the last step we used that the second term is larger for $n$ large enough, as a consequence of the fact that
$$\sqrt n  \nabla^k p_n(j)  \to \infty \ \ \text{and} \ \ \sqrt n (\nabla^k \widehat p_n(j) - \nabla^k p_n (j)  = O_P(1) $$
for all $j \in I_+$.    Hence, using the Continuous Mapping Theorem, it holds with probability tending to 1 that
\begin{align*}
 \sqrt n (\widehat \rho_{k, n} - \min_{j\in \mathcal{S}_k} \nabla^k  p_n(j) ) 
 & =  \min_{j \in I_0}  \Big( \sqrt n (\nabla^k \widehat p_n(j) - \nabla^k p_n(j))   + \nabla^k h(j) + \delta \Big)   \\
& \overset{d}{\to}  \min_{j \in I_0}  (Z_j +  \nabla^k h(j)  +  \delta )  \\
&=     \min_{j \in I_0}  (Z_j +  \nabla^k h(j) ) + \delta,
 \end{align*}
as claimed.

\medskip

\noindent 2. We want to show that $P(\widehat{I}_{0, n} 
 \ne I_0)  = O(1/\sqrt n)$.  To this aim, we first bound the probability $P( \widehat S_{k, n}  \ne \mathcal S_k )$.
 Using the union bound, we have that  
\begin{align*}
  P( \widehat S_{k, n}  \ne \mathcal S_k )&= P(\exists \ j \in \mathcal S_k:  \widehat p_n(j) 
 = 0) \\
 &\le     \sum_{j \in \mathcal S_k}  P(\widehat p_n(j) = 0)  \\
 & =   \sum_{j \in \mathcal S_k}  P\Big(\frac{\sqrt n (\widehat p_n(j)  - p_n(j))}{\sqrt{p_n(j) (1- p_n(j))} }  = - \frac{\sqrt n p_n(j)}{\sqrt{p_n(j) (1- p_n(j))}} \Big)  \\
 & \le  \sum_{j \in \mathcal S_k}  P\Big(\frac{\sqrt n (\widehat p_n(j)  - p_n(j))}{\sqrt{p_n(j) (1- p_n(j))} }  \le - \frac{\sqrt n p_n(j)}{2\sqrt{p_n(j) (1- p_n(j))}} \Big) \\
 & \le  \sum_{j \in \mathcal S_k}  P\Big(\Big \vert \frac{\sqrt n (\widehat p_n(j)  - p_n(j))}{\sqrt{p_n(j) (1- p_n(j))} } \Big \vert \ge \frac{\sqrt n p_n(j)}{2\sqrt{p_n(j) (1- p_n(j))}} \Big)  \\
 & \le    \sum_{j \in \mathcal S_k} \frac{4 p_n(j) (1- p_n(j))}{n p^2_n(j)}  =  \frac{4}{n} \sum_{j \in \mathcal S_k} \frac{1- p_n(j)}{p_n(j)}
\end{align*}
where we used the Chebyshev inequality. Recalling that $p_n(j) 
 =  p(j)  +  h(j) / \sqrt n$, we have that  $p_n(j)  \ge p(j) /2 $  for all $j \in \mathcal S_k$  provided that 
 $ n > 4 \sup_{j \in \mathcal S_k}  h(j) ^2/p(j)^2$. 
 It follows that for $n$ large enough
 \begin{align}
 \label{eq:S_k}
  P( \widehat S_{k, n}  \ne\mathcal S_k )  \le  \frac{1}{n}  \frac{8 \vert \mathcal S_k \vert }{ \min_{j \in \mathcal S_k} p(j)}.
 \end{align}
Now, we turn to bounding the probability $P(\widehat I_{0, n} \ne I_0)$. Let us denote by $\widetilde{I}_{0, n}$ the set which is similarly defined as $\widehat I_{0, n}$ expect that $\widehat S_{k, n}$ is replaced by $\mathcal S_k$.  By \eqref{eq:S_k}, we have that
\begin{align*}
P( \widehat I_{0, n} \ne I_0 ) \le  P( \widetilde I_{0, n} \ne I_0 )  +  O(1/n).
\end{align*}
Now, we rewrite $P( \widetilde I_{0, n} \ne I_0 )$ as
\begin{align*}
  P( \widetilde I_{0, n} \ne I_0 ) &\le P(\exists \ j \in  \widetilde I_{0, n}  \cap I_0^c )  +  P(\exists \ j \in  \widetilde I^c_{0, n}  \cap I_0) \\
 & =:  A  +B.
\end{align*}
 Recall that $j \in \widetilde I_{0, n}$ means that 
\begin{align*}
\frac{\sqrt n \nabla^k \widehat p_n(j)}{\sqrt{(\widehat \Sigma_{k, n})_{j, j}}} \le z_{1-1/n}
\end{align*}
 where $\widehat \Sigma_{k, n}$ is the empirical estimator of the true covariance matrix of $(\nabla^k \mathbf{1}_{\{X_{n, 1} =  j \}})_{j \in \mathcal S_k}$. For simplicity, let $\widehat \sigma_{j, n}=\sqrt{(\widehat \Sigma_{k, n})_{j, j}}$,   let us denote 
$$Z_{i,j}  = \nabla^k \mathbf{1}_{\{X_{n,i} = j\}}, \ \ \text{for} \ \ i=1, \ldots, n,$$
and let $\sigma^2_{j, n}$ be the true variance of $Z_{1,j}$.
Then, we have that  
 \begin{align}
 \label{eq:sigma}
\widehat \sigma_{j, n}^2 =  \frac{1}{n} \sum_{i=1}^n (  Z_{i, j} -  \nabla^k p_n(j) )^2 - (\nabla^k \widehat p_n(j) -   \nabla^k p_n(j) )^2. 
 \end{align}
  Since $I_+ = \mathcal S_k \setminus  I_0 $, we have
\begin{align*}
A  & \le     \sum_{j \in  I_+}  P\Big( \frac{\sqrt n (\nabla^k \widehat p_n(j)  - \nabla^k p_n(j) )}{\widehat \sigma_{j, n}}   \le  z_{1-1/n}  -  \frac{\sqrt n \nabla^k p_n(j) }{\widehat \sigma_{j, n}} \Big)  \\
& =    \sum_{j \in  I_+}  P\Big( \frac{\sqrt n (\nabla^k \widehat p_n(j)  - \nabla^k p_n(j) )}{\sigma_{j}}   \le \frac{\widehat \sigma_{j, n}}{\sigma_j} z_{1-1/n}  -  \frac{\sqrt n \nabla^k p_n(j) }{\sigma_{j}} \Big)  \\
& \le   \sum_{j \in  I_+}  P\Big( \frac{\sqrt n (\nabla^k \widehat p_n(j)  - \nabla^k p_n(j) )}{\sigma_{j}}   \le 2 z_{1-1/n}  -  \frac{\sqrt n \nabla^k p_n(j) }{\sigma_{j}} \Big)  +  \sum_{j \in  I_+} P (\widehat \sigma_{j, n} > 2 \sigma_j ) \\
& =:   A_1 + A_2.
\end{align*}
For $j \in I_+$, we have that $\nabla^k p(j)  > 0$ and, for $n$ large enough, $\nabla^k p_n(j)  \ge  \nabla^k p(j)/2$ and $z_{1-1/n} \le 2 \sqrt{2 \log n}$. Let $r_j = \nabla^k p(j)/\sigma_j$,  and take  $n$ large enough such that 
$$   {4 \sqrt {2\log n}}  \leq {\frac{\sqrt{n} }{2}\min_{j \in I_+} r_j}.$$ 
Then, 
\begin{align*}
2 z_{1-1/n}  -  \frac{\sqrt n \nabla^k p_n(j) }{\sigma_{j}} &\leq 4\sqrt{2\log n} -\sqrt{n}r_j \\
&\leq \frac{\sqrt{n}r_j }{2} -\sqrt{n}r_j = -\frac{\sqrt{n}r_j }{2}
\end{align*}
for all $j \in I_+$.   Thus, for $n$ large enough
\begin{align*}
& P\Big( \frac{\sqrt n (\nabla^k \widehat p_n(j)  - \nabla^k p_n(j) )}{\sigma_{j}}   \le 2 z_{1-1/n}  -  \frac{\sqrt n \nabla^k p_n(j) }{\sigma_{j}} \Big)  \\
& \le P\Big( - \frac{\sqrt n (\nabla^k \widehat p_n(j)  - \nabla^k p_n(j) )}{\sigma_{j}}  \ge   \frac{\sqrt n  r_j }{2} \Big) \\
& \le P\Big( \Big \vert \frac{\sqrt n (\nabla^k \widehat p_n(j)  - \nabla^k p_n(j))}{\sigma_{j}} \Big \vert \ge   \frac{\sqrt n  r_j }{2} \Big)   \le  \frac{4}{r^2_j n}
\end{align*}
using the Chebyshev inequality. Hence, 
\begin{align*}
A_1 \le  \frac{4}{n} \frac{ \vert I_+ \vert }{(\min_{j \in I_+} r_j^2)}    .
\end{align*}
Now, we turn to $A_2$. By \eqref{eq:sigma} we have that 
\begin{align*}
\sqrt n (\widehat \sigma_{j, n}^2 - \sigma_{j, n}^2 )  =  \sqrt n \Big(  \frac{1}{n} \sum_{i=1}^n (  Z_{n, i} -  \nabla^k p_n(j) )^2 - \sigma_{j, n}^2  \Big)  -    \sqrt n (\nabla^k \widehat p_n(j) - \nabla^k p_n(j)) ^2.
\end{align*}
Then,
\begin{align*}
A_2 & =    \sum_{j \in  I_+} P\Big ( \sqrt n (\widehat \sigma_{j, n}^2  - \sigma^2_{j, n} )  > 3 \sqrt n  \sigma^2_j  \Big)  \\
& \le  \sum_{j \in  I_+} P\Big(  \sqrt n \Big(  \frac{1}{n} \sum_{i=1}^n \left(  Z_{n, i} -  \nabla^k p_n(j) \right)^2 - \sigma_{j, n}^2  \Big)   > 3 \sigma^2_{j, n} \sqrt n  \Big). 
\end{align*}
Let $\gamma^2_{j, n}  =  \var[ (Z_{n, 1} -  \nabla^k p_n(j))^2] $. Then, we have
\begin{align*}
\gamma^2_{j, n}  =  \E[ (Z_{1, j} -  \nabla^k p_n(j))^4  ]  - \sigma^4_j.
\end{align*}
Now, using similar arguments as above, we can show that
\begin{align*}
\gamma^2_{j, n}  \to \gamma^2_{j}  =  \E[ (\nabla^k \mathbf{1}_{\{X_1 = j \}} -  \nabla^k p(j))^4  ]  - \sigma^2_j,
\end{align*}
where $X_1 \sim p$. Thus, for $n$ large enough, $\gamma_{j,n} \ge \gamma_{j} /2$ and, 
\begin{align*}
A_{2}  & =   \sum_{j \in  I_+}P\Big( \frac{\sqrt n (  \frac{1}{n} \sum_{i=1}^n (  Z_{n, i} -  \nabla^k p_n(j) )^2 - \sigma_{j, n}^2 )}{\gamma_{j, n}} >  \sqrt n\frac{3 \sigma_j^2}{\gamma_{j, n}}  \Big)  \\
& \le   \sum_{j \in  I_+} P\Big( \frac{\sqrt n (  \frac{1}{n} \sum_{i=1}^n (  Z_{n, i} -  \nabla^k p_n(j) )^2 - \sigma_{j, n}^2 )}{\gamma_{j, n}} >  \sqrt n\frac{6 \sigma_j^2}{\gamma_{j}}  \Big)  \\
& \le   \sum_{j \in I_+} \frac{1}{n} \frac{\gamma_j^2 }{36 \sigma^4_j}  \le \frac{1}{n} \frac{\vert I_+ \vert }{36} \max_{j \in I_+} \frac{\gamma_j^2}{\sigma^4_{j}} 
\end{align*}
by the Chebyshev inequality.
We conclude that $A = O(1/n)$. Now, we turn to the term $B$.  Recall that for $j \in I_0$, it holds that
\begin{align*}
\sqrt n \nabla^k p_n(j)  =  \nabla^k h(j).
\end{align*}
Then, it follows that
\begin{align*}
B & \le    \sum_{j \in I_0} P\Big( \frac{\sqrt n \nabla^k \widehat p_n(j)}{\widehat \sigma_{j,n}} >  z_{1-1/n}   \Big) \\
& =   \sum_{j \in I_0} P\Big( \frac{\sqrt n (\nabla^k \widehat p_n(j) - \nabla^k p_n(j)) }{\sigma_{j,n}} >   \frac{\widehat \sigma_{j, n}}{\sigma_{j,n}}  z_{1-1/n} - \frac{\nabla^k h(j)}{\sigma_{j, n}} \Big) \\
& \le   \sum_{j \in I_0} P\Big( \frac{\sqrt n (\nabla^k \widehat p_n(j) - \nabla^k p_n(j)) }{\sigma_{j,n}} >   \eta z_{1-1/n} - \frac{2 \vert \nabla^k h(j) \vert }{\sigma_j} \Big) +  P( \widehat \sigma_{j,n} < \eta \sigma_{j,n} ) \\
& =:  B_{1} + B_2
\end{align*}
where $\eta \in (0,1) $ is to be chosen later. With similar arguments as above, we can show that 
$$
B_2 = O(1/n).
$$
Now we employ the general Berry-Esseen theorem in \cite{esseen1969} applied to zero-mean independent random variables that are allowed to have different distributions, possibly depending on $n$. Let us denote by $F_{j, n}$ the distribution function of 
$$
\frac{\sqrt n(\nabla^k \widehat p_n(j) - \nabla^k p_n(j))}{\sigma_{j, n}}  =   \frac{1}{\sigma_{j,n} \sqrt {n } }\Big(\sum_{i=1}^n (Z_{i, j} -  \nabla^k p_n(j)) \Big) ,
$$
and let $m_{j, n}   =  \E[\vert Z_{1, j}-  \nabla^k p_n(j) \vert^3 ]$.  Then, according to (1.1) in \cite{esseen1969}
\begin{align*}
\Vert F_{j, n}  - \Phi  \Vert_\infty \le  C \frac{m_{j, n}}{\sqrt n \sigma_{j,n}^3}
\end{align*}
for some $C > 0$. Similarly to \eqref{eq:majsup} we have that $$\Vert Z_{1, j}-  \nabla^k p_n(j) \Vert_\infty \le 2^k.$$
Thus, it follows that for all $j \in I_0$,
\begin{align*}
\frac{m_{j, n}}{\sigma_{j,n}^3} \leq \frac{2^k\E[\vert Z_{1, j}-  \nabla^k p_n(j) \vert^2 ]}{\sigma_{j,n}^3} = \frac{  2^{k} }{\sigma_j/2}  
\end{align*}
almost surely for $n$ large enough.  Therefore,  setting $q_*  = 2 \max_{j \in I_0}  \vert \nabla^k h(j) \vert/\sigma_j$, we have that
\begin{align*}
B_1  & \le \frac{C}{\sqrt n} \frac{2^{k+1} \vert I_0 \vert}{\min_{j \in I_0} \sigma_j}  +  \sum_{j \in I_0} (1-\Phi( \eta z_{1-1/n} - 2 \vert \nabla^k h(j) \vert/\sigma_j   ) )  \\
& \le   \frac{C}{\sqrt n} \frac{2^{k+1} \vert I_0 \vert}{\min_{j \in I_0} \sigma_j} +  \vert I_0 \vert \ (1 - \Phi(\eta z_{1-1/n} - q_*) ).
\end{align*}
Now we choose $n$ large enough such that $$\eta z_{1-1/n} \ge \eta^2 \sqrt{2 \log n} \ \ \text{and} \ \ q_* \le (\eta^2 - \eta^3) \sqrt{2 \log n}.$$
This implies that 
\begin{align*}
\eta z_{1-1/n} - q_* &\ge \eta^2 \sqrt{2 \log n} - (\eta^2 - \eta^3) \sqrt{2 \log n} \\
&=  \eta^3 \sqrt{2 \log n}.
\end{align*}
Then, by \eqref{eq:millsratio} 
\begin{align*}
1 - \Phi(\eta z_{1-1/n} - q_*) & \le     1 -  \Phi(\eta^3 \sqrt{2\log n}) \\
& \le      \frac{2\varphi( \eta^3 \sqrt{2 \log n})}{\eta^3  \sqrt{2 \log n} } \\
& =      \frac{2}{\sqrt{2\pi} \eta^3 \sqrt{2\log n}}  \exp(- \eta^6 \log n   )  \\
& =    \frac{2}{\sqrt{2\pi} \eta^3 \sqrt{2\log n}}  \frac{1}{n^{\eta^6}}  = o(1/\sqrt n)
\end{align*}
 provided that $\eta^6  \ge 1/2$ or equivalently $\eta \in [2^{-1/6}, 1)$.  We conclude that $B_1 = o(1/\sqrt n)$ and $B = O(1/\sqrt n)$. Putting all the pieces together yields  
 $$P(\widehat{I}_{0,n} \ne  I_0) = O(1/\sqrt n)$$
 which completes the proof.
 
 \medskip

 \noindent 3.
Note that $\widehat{T}_{k,n}$ can  be re-written as
 \begin{align}\label{eq:convContig}
     \widehat{T}_{k,n}  & =  \sqrt{n}(\widehat{\rho}_{k,n}-\min_{j\in \mathcal{S}_k}\nabla p_n (j)) + \sqrt{n} \min_{j\in \mathcal{S}_k}\nabla p_n (j) \notag \\
     & =  \sqrt{n}(\widehat{\rho}_{k,n}-\min_{j\in \mathcal{S}_k}\nabla p_n (j)) - \delta  \notag \\
     &\overset{d}{\to}  \min_{j \in I_0}  (Z_j +  \nabla^k h(j) ) 
 \end{align}
 by point 1. 
Since $\nabla^k h(j)  \le 0$ for all $j \in I_0$, we have that 
$$\min_{j \in I_0}  (Z_j +  \nabla^k h(j) ) \le \min_{j \in I_0} Z_j.
$$ 
Fix $\epsilon > 0$. Also, let us assume that $ \widehat t_{\alpha, n}(I_0) \overset{P}{\to} t_\alpha( I_{0})$ (this convergence will be proved below).  Combining this with point 2 and the convergence in \eqref{eq:convContig} yields 
\begin{align*}
  P(\widehat{T}_{k,n} < \widehat t_{\alpha, n}(\widehat{I}_{0,n}))&=  P(\widehat{T}_{k,n} < \widehat t_{\alpha, n}(I_0), \widehat{I}_{0,n} = I _0)  +  P(\widehat{T}_{k,n} < \widehat t_{\alpha, n}(\widehat{I}_{0,n}), \widehat{I}_{0,n} \ne I _0)   \\
  & \ge  P(\widehat{T}_{k,n} < \widehat t_{\alpha, n}(I_0), \widehat{I}_{0,n} = I _0) \\
  & =   P(\widehat{T}_{k,n} < \widehat t_{\alpha, n}(I_0)) - P( \widehat{T}_{k,n} < \widehat t_{\alpha, n}(I_0), \widehat{I}_{0,n} \ne I _0) \\
  & \ge    P(\widehat{T}_{k,n} < \widehat t_{\alpha, n}(I_0)) - P(\widehat{I}_{0,n} \ne I _0) \\
  & \ge  P( \widehat{T}_{k,n} < \widehat t_{\alpha, n}(I_0)) - O(1/\sqrt{n})\\
& \ge   P(  \widehat{T}_{k,n}  < \widehat t_{\alpha, n}(I_0), \widehat{T}_{k,n} <  t_\alpha(I_0) - \epsilon) - O(1/\sqrt{n})  \\
& =    P(\widehat{T}_{k,n} < t_\alpha(I_0) - \epsilon)  -  P( \widehat{T}_{k,n} <  t_\alpha(I_0) - \epsilon,  \widehat{T}_{k,n} \ge \widehat t_{\alpha, n}(I_0) ) \\
&  \ - \  O(1/\sqrt{n})  \\
& \ge    P(\widehat{T}_{k,n} < t_\alpha(I_0) - \epsilon) - P(\widehat t_{\alpha, n}(I_0) < t_\alpha(I_0) - \epsilon) - \  O(1/\sqrt{n}) \\
&=  P(\min_{j \in I_0}  (Z_j +  \nabla^k h(j) )\leq t_\alpha(I_0)  - \epsilon) + o(1)\\
&\geq    P(\min_{j \in I_0}  Z_j  \leq t_\alpha(I_0) - \epsilon) + o(1) =\alpha  + o(1),
\end{align*} 
as $\epsilon \searrow 0$. This implies that 
$$\liminf_{n \to \infty} P (\widehat{T}_{k,n}< \widehat t_{\alpha,n}(\widehat{I}_{0,n}) ) \ge \alpha.$$ 
Suppose now $\max_{j \in I_0} \nabla^k h(j) = - \kappa   $ for some $\kappa \in (0, \delta]$.  Then, $$\min_{j \in I_0} (Z_j + \nabla^k h(j)) \le  \min_{j \in I_0} Z_j - \kappa, $$ 
and, hence,
$$
P(\min_{j \in I_0}  (Z_j +  \nabla^k h(j) )\leq t_\alpha(I_0)) \ge P(\min_{j \in I_0}  Z_j  \leq t_\alpha(I_0) + \kappa).
$$
Therefore,
\begin{align*}
 P(\widehat{T}_{k,n}< \widehat t_{\alpha, n}(\widehat{I}_{0,n}))  \ge  P( \min_{j \in I_0} Z_j  \leq t_\alpha(I_0)  + \kappa )  +  o(1),
\end{align*}
and
\begin{align*}
\liminf_{n \to \infty} P(\widehat{T}_{k,n}< \widehat {t}_{\alpha, n}(\widehat{I}_{0,n})) \ge  P( \min_{j \in I_0} Z_j  \leq t_\alpha(I_0)  + \kappa ) > \alpha.
\end{align*}
Now, we show the claimed convergence
\begin{align*}
\widehat t_{\alpha, n}(I_0) \overset{P}{\to} t_\alpha( I_{0}).
\end{align*}
Recall that $\widehat t_{\alpha, n}(I_0)$ is the $\alpha$-quantile of $\min_{j \in I_0} \widehat Z_j$, where 
$$
(\widehat Z_j)_{j \in \mathcal S_k}  \sim \mathcal N(0, \widehat \Sigma_{k, n}).
$$
Using Lemma \ref{lem:convquant}, it is enough to show that $\widehat \Sigma_{k, n}  \overset{P}{\to} \Sigma_k$.  Without loss of generality, we assume that $m=1$. For $i, j \in \{m, \ldots, M- m - k\} $,  the $(i, j)$ entry of $\widehat \Sigma_{k, n}$ is given by
$$
(\widehat \Sigma_{k, n})_{ij} = \frac{1}{n} \sum_{l=1}^n \nabla^k \mathbf{1}_{\{ X_{n, l} = i \} } \nabla^k \mathbf{1}_{\{ X_{n, l} = j \} }  -  \Big(\frac{1}{n} \sum_{l=1}^n  \nabla^k \mathbf{1}_{\{ X_{n, l} = i \} } \Big)\Big(\frac{1}{n} \sum_{l=1}^n  \nabla^k \mathbf{1}_{\{ X_{n,l} = j \} } \Big). 
$$
Now, for a fixed $i \in \mathcal S_k$ define 
$$
Z_{n, l} =  \nabla^k \mathbf{1}_{\{ X_{n, l} = i \} }   -  \E[\nabla^k \mathbf{1}_{\{ X_{n, 1} = i \} }] 
$$
for $l =1, \ldots, n$.  We have $\E[Z_{n,l}] = 0$ and
\begin{align*}
 \E[\vert Z_{1,1} \vert^2]   =     \text{Var}\left(\nabla^k \mathbf{1}_{\{ X_{1, 1} = i \} }\right) 
 \le   \E[(\nabla^k \mathbf{1}_{\{ X_{n, 1} = i \} } )^2]  \le 2^{2k}.
\end{align*}
Theorem 1 from \cite{hu1989strong} implies that 
$$
\frac{1}{n} \sum_{l=1}^n Z_{n, l}  \overset{{\text{a.s.}}}{\to} 0
$$
or equivalently
$$
\frac{1}{n} \sum_{l=1}^n \nabla^k \mathbf{1}_{\{ X_{n, l} = i \} }   -  \E[\nabla^k \mathbf{1}_{\{ X_{n, 1} = i \}} ]\overset{{\text{a.s.}}}{\to}  0.
$$
Now, fix $i, j \in \mathcal S_k$ and let
$$
Z_{n, l} =  \nabla^k \mathbf{1}_{\{ X_{n, l} = i \} } \nabla^k \mathbf{1}_{\{ X_{n, l} = j \} }  -  \E[\nabla^k \mathbf{1}_{\{ X_{n, 1} = i \} } \nabla^k \mathbf{1}_{\{ X_{n, 1} = j \} }] 
$$
for $l =1, \ldots, n$. We have $\E[Z_{n,l}] = 0$ and 
\begin{align*}
 \E[\vert Z_{1,1} \vert^2]  & =    \text{Var}\left(\nabla^k \mathbf{1}_{\{ X_{1, 1} = i \} } \nabla^k \mathbf{1}_{\{ X_{1, 1} = j \} }\right)   \\
 & \le     \E\left[(\nabla^k \mathbf{1}_{\{ X_{1, 1} = i \} } )^2 (\nabla^k \mathbf{1}_{\{ X_{1, 1} = j \} })^2 \right]   \le 2^{4k}.
\end{align*}
By the same theorem, we conclude that
$$
\frac{1}{n} \sum_{l=1}^n \nabla^k \mathbf{1}_{\{ X_{n, l} = i \} } \nabla^k \mathbf{1}_{\{ X_{n, l} = j \} }  -  \E[\nabla^k \mathbf{1}_{\{ X_{n, 1} = i \} } \nabla^k \mathbf{1}_{\{ X_{n, 1} = j \} }]\overset{{\text{a.s.}}}{\to} 0.
$$
Therefore, for all $i, j \in \mathcal S_k$
$$
(\widehat \Sigma_{k, n})_{i,j} - \text{Cov}( \nabla^k \mathbf{1}_{\{ X_{n, 1} = i \} }, \nabla^k \mathbf{1}_{\{ X_{n, 1} = j \} } )  \overset{{\text{a.s.}}}{\to} 0.
$$
Noting that the expression on the right is equal to $\sum_{l=1}^n (\text{Cov}[Y_{n,l}])_{i,j}$ with $Y_{n, l}$ is the same as in (\ref{Y}) and using the convergence proved in (\ref{LindCond}) it follows that
$$
\widehat \Sigma_{k, n}  \overset{{\text{a.s.}}}{\to} \Sigma_k,
$$
which implies the cliamed convergence in probability.
 \end{proof}

\subsection*{Proof of Theorem \ref{thm:seperation}}
\begin{proof}[Proof of Theorem \ref{thm:seperation}]

Since $\widehat I_{0, n} \subseteq \widehat{S}_{k, n}$, we must have that 
$$
\widehat t_{\alpha, n}(\widehat I_{0, n}) \ge \widehat t_{\alpha, n}(\widehat S_{k, n}).
$$ 
Also, define $\widehat t_{\alpha, n}: = \widehat t_{\alpha, n}(\mathcal S_{k})$. Note $\widehat t_{\alpha, n} = \widehat t_{\alpha, n}(\widehat S_{k, n})$ when $\widehat S_{k, n} = \mathcal S_k$.  We can write
\begin{align*}
P (\widehat T_{n, k} < \widehat t_{\alpha, n}(\widehat I_{0, n}) ) & \ge   P (\widehat T_{n, k} < \widehat t_{\alpha, n}(\widehat S_{k, n}))  \\
& \ge   P (\widehat T_{n, k} < \widehat t_{\alpha, n}(\widehat S_{k, n}), \widehat{\mathcal S}_{k, n}  = \mathcal{S}_k )  \\
& =   P (\widetilde T_{n, k} < \widehat t_{\alpha, n}, \widehat{\mathcal S}_{k, n}  = \mathcal{S}_k )  \\
& =   P(\widetilde T_{n, k} < \widehat t_{\alpha, n}) -  P(\widetilde T_{n, k} < \widehat t_{\alpha, n}, \widehat{\mathcal S}_{k, n}  \ne \mathcal{S}_k  )   \\
& \ge   P(\widetilde T_{n, k} < \widehat t_{\alpha, n} ) -  P(\widehat{\mathcal S}_{k, n}  \ne \mathcal{S}_k  )  \\
& =:  A - B
\end{align*}
where $\widetilde{T}_{n, k}  = \sqrt n \min_{j \in \mathcal S_k} \nabla^k \widehat p_n(j),$  and
\begin{align*}
A & =    P(\widetilde T_{n, k} < \widehat t_{\alpha, n} )   =  P(\sqrt n (\min_{j \in \mathcal S_k} \nabla^k \widehat p_n(j) - \rho_k) < - \sqrt n \rho_k  + \widehat t_{\alpha, n}  ).
\end{align*}
Now, for any $j \in  I_{\rho_k}$, we have that 
\begin{align*}
  P(\sqrt n (\nabla^k \widehat p_n(j) - \rho_k) < - \sqrt n \rho_k  + \widehat t_{\alpha, n}  )  \le  P(\sqrt n (\min_{j \in \mathcal S_k} \nabla^k \widehat p_n(j) - \rho_k) < - \sqrt n \rho_k  + \widehat t_{\alpha, n}  )
\end{align*}
and, therefore,  
\begin{align*}
  \max_{j \in I_{\rho_k}} P(\sqrt n (\nabla^k \widehat p_n(j) - \rho_k) < - \sqrt n \rho_k  + \widehat t_{\alpha, n} )  \le  P(\sqrt n (\min_{j \in \mathcal S_k} \nabla^k \widehat p_n(j) - \rho_k) < - \sqrt n \rho_k  + \widehat t_{\alpha, n} ).
 \end{align*}
Letting
\begin{align*}
   \widetilde m_j =   \E [ \vert \nabla^k \mathbf{1}_{\{X_1=j\}}  - \nabla^k p(j)  \vert^3  ],  \quad  
 \sigma^2_j   =  \E [ (\nabla^k \mathbf{1}_{\{X_1=j\}} - \nabla^k p(j))^2 ],
\end{align*}
since for all $j \in \rho_k$ we have $\nabla^kp(j)=\rho_k$, it follows
\begin{align}
\label{eq:BoundA}
A & \ge  \max_{j \in I_{\rho_k}}  P\Big( \frac{\sqrt n (\nabla^k \widehat p_n(j) - \rho_k)}{\sigma_j }  < \frac{\sqrt n \vert \rho_k\vert  + \widehat t_{\alpha, n}}{\sigma_j} \Big)\notag  \\
& \ge   \max_{j \in I_{\rho_k}} \Big \{ -C   \frac{\widetilde m_j}{\sigma_j^3} \frac{1}{\sqrt n}  + 
  \Phi\Big( \frac{\sqrt n \vert \rho_k\vert  + \widehat t_{\alpha, n}}{\sigma_j}   \Big) \Big \}\notag \\
& \ge   \min_{j \in I_{\rho_k}} \Big( - C  \frac{\widetilde m_j}{\sigma_j^3} \Big) \frac{1}{\sqrt n}  +  \max_{j \in I_{\rho_k}}  \Phi\Big( \frac{\sqrt n \vert \rho_k\vert  + \widehat t_{\alpha, n}}{\sigma_j}   \Big)\notag \\
& \ge   -\frac{1}{2} \max_{j \in I_{\rho_k}} \frac{\widetilde m_j}{\sigma_j^3} \frac{1}{\sqrt n}  + \Phi\Big( \frac{\sqrt n \rho_k + \widehat t_{\alpha, n}}{\min_{j \in I_{\rho_k}} \sigma_j} \Big)
\end{align}
using the fact $C \approx 0.4785$ (see \citealp{tyurin2012refinement}) and assuming that $n$ is large enough such that $\sqrt n \vert \rho _k \vert + \widehat t_{\alpha, n} \ge 0$  (this condition will be later automatically satisfied).  Now, we will bound B from above.  As shown in the proof of Theorem \ref{thm:selection}, we have that 
\begin{align*}
B = P( \mathcal S_k  \cap \widehat{\mathcal{S}}_{k,n}^c\neq \emptyset)  \le  \frac{C}{\sqrt n } \sum_{j \in  \mathcal{S}_k} \frac{m_j}{(p(j)(1-p(j)))^{3/2}}  +   \sum_{j \in  \mathcal{S}_k} \Phi\Big(- \frac{\sqrt n p(j)}{2 \sqrt{p(j) (1-p(j))}}\Big), 
\end{align*}
with $m_j =  \E[\vert \mathbf{1}_{\{X_1 = j\}} - p(j) \vert^3]$ and $C \le 1/2$ is the same constant as before. Moreover, for $j \in \mathcal S_k$, it holds 
\begin{align*}
  \frac{m_j}{(p(j)(1-p(j)))^{3/2}} &  =     \frac{p(j) (1-p(j))^3 + (1-p(j)) p(j)^3}{(p(j)(1-p(j)))^{3/2}}  =   \frac{(1-p(j))^2 + p(j)^2}{\sqrt{p(j)(1-p(j))}} \\
 &\le    \frac{1}{\sqrt{p(j)(1-p(j))}} \le \frac{\sqrt 2}{\sqrt{\xi}}
\end{align*}
using $(1-p(j))^2 + p(j)^2 \le 1-p(j) + p(j) =1$, the definition of $\xi$ and the fact that if $p(j) \le 1/2$, then $1-p(j) \ge 1/2$ and $p(j) (1-p(j)) \ge \xi/2$, and if $p(j) > 1/2$, then $1-p(j) \ge \xi$ and $p(j) (1-p(j)) > \xi/2$.   Moreover,
\begin{align*}
 \Phi\Big(- \frac{\sqrt n p(j)}{2 \sqrt{p(j) (1-p(j))}}\Big)  &= 1 - \Phi\Big(\frac{\sqrt n p(j)}{2 \sqrt{p(j) (1-p(j))}}\Big)  \\
 & =   1 - \Phi\Big( \frac{\sqrt n}{2}\sqrt{\frac{p(j)}{1-p(j)}} \Big)  \\
 & \le   1- \Phi\left(\frac{\sqrt n \sqrt \xi}{2}   \right) \\
 & \le    \frac{2\varphi(\sqrt n \sqrt{\xi}/2)}{\sqrt n \sqrt{\xi}}\\
 &\le  \frac{16 e^{-1}}{\sqrt{2\pi}}  \frac{1}{\xi^{3/2} n^{3/2}}    
 \end{align*}
where we used the inequalities $1- \Phi(x) \le x \varphi(x)$ for all $x\in \mathbb{R}$ and $x \exp(-x) \le e^{-1}$ for all $x > 0$.  We conclude that for any fixed $n \ge 1$,
\begin{align}\label{eq:BoundB}
B  \le  \frac{\vert \mathcal S_k \vert \sqrt 2}{2 \sqrt{\xi} \sqrt n}  + \frac{16 e^{-1} \vert \mathcal S_k \vert }{\sqrt{2\pi} \xi^{3/2} n^{3/2}}.
\end{align}
Combining the bounds \eqref{eq:BoundA} and \eqref{eq:BoundB}, we have the power inequality
\begin{align*}
P( \widehat T_{k, n} < \widehat t_{n,\alpha}(\widehat I_{0, n}) ) &\ge A-B \\
&\geq -\frac{1}{2} \max_{j \in I_{\rho_k}} \frac{\widetilde m_j}{\sigma_j^3} \frac{1}{\sqrt n}  + \Phi\Big( \frac{\sqrt n \rho_k + \widehat t_{\alpha, n}}{\min_{j \in I_{\rho_k}} \sigma_j} \Big)- \frac{\vert \mathcal S_k \vert \sqrt 2}{2 \sqrt{\xi} \sqrt n}  - \frac{16 e^{-1} \vert \mathcal S_k \vert }{\sqrt{2\pi} \xi^{3/2} n^{3/2}}\\
&\geq 1-\beta.
\end{align*}
where the last inequality is satisfied  under the conditions
\begin{align}\label{eq:MainCond}
\frac{1}{2} \max_{j \in I_{\rho_k}} \frac{\widetilde m_j}{\sigma_j^3} \frac{1}{\sqrt n}  + \frac{\vert \mathcal S_k \vert \sqrt 2}{2 \sqrt{\xi} \sqrt n}  + \frac{16 e^{-1} \vert \mathcal S_k \vert }{\sqrt{2\pi} \xi^{3/2} n^{3/2}} \le \beta/2, \ \quad \ \sqrt n \vert \rho_k\vert  \ge z_{1-\beta/2} \min_{j \in I_{\rho_k}}\sigma_j  - \widehat t_{\alpha, n},
\end{align}
with $z_{1-\beta/2} = \Phi^{-1}(1-\beta/2)$.  To get a more explicit bound, recall that $\xi =  \min_{j \in \mathcal S_k} (p(j) \wedge (1-p(j))$ and let $q$ be a real-valued function defined on the integers. For $j \in \mathbb Z$, the following identity can be shown recursively:
\begin{align*}
\nabla^k q(j)  =  \sum_{l=0}^k (-1)^l \binom{k}{l} q(j + l).
\end{align*}
Then, since for all $j \in I_{\rho_k}$ we have $\nabla^k p(j) =\rho_k$, for such $j$'s it holds
\begin{align*}
\sigma_j^2 & = \sum_{x\in \mathcal{S}} \Big( \sum_{l=0}^k (-1)^l \binom{k}{l} \mathbf{1}_{\{x=j + l\}}-\rho_k\Big)^2  p(x) \\
&=   \sum_{l=0}^k \Big( (-1)^l \binom{k}{l} - \rho_k  \Big)^2  p(j+l)   \\
& \ge   \xi \sum_{l=0}^k \Big( (-1)^l \binom{k}{l} - \rho_k  \Big)^2\\
&=  \xi \Big(\binom{2k}{k} + (k+1) \rho_k^2 \Big)
\end{align*}
where in the last step we have used the fact that  $\sum_{l=0}^k (-1)^l \binom{k}{l} =0$ and $\sum_{l=0}^k \binom{k}{l}^2 = \binom{2k}{k}$. This implies that 
\begin{align}\label{eq:minsigma}
\min_{j \in I_{\rho_k}} \sigma_j \ge \sqrt{\xi} \Big(\binom{2k}{k} + (k+1)\rho_k^2 \Big)^{1/2}.
\end{align}
Now, we derive an upper bound for 
$$
\max_{j \in I_{\rho_k}} \frac{\widetilde m_j}{\sigma_j^3}.
$$
For $j \in I_{\rho_k}$ we have that
$$
\widetilde m_j  = \E [\vert \nabla^k \mathbf{1}_{\{X_1 = j\}} -  \rho_k \vert^3 ]  = \E[\vert \nabla^k \mathbf{1}_{\{X_1 = j\}} -  \rho_k \vert \vert \nabla^k \mathbf{1}_{\{X_1 = j\}} -  \rho_k \vert^2] 
$$
where
\begin{align*}
 \vert \nabla^k \mathbf{1}_{\{X_1 = j\}} -  \nabla^k p(j)   \vert & =   \Big \vert \sum_{l=0}^k (-1)^l \binom{k}{l} \mathbf{1}_{\{X_1 = j + l\}} + \vert \rho_k \vert \Big \vert \\
& \le    \max_{ 0\le l \le k } \Big \vert (-1)^l \binom{k}{l} + \vert \rho_k \vert \Big \vert\\
& \le    \max_{ 0\le l \le k } \binom{k}{l} + \vert \rho_k \vert \\
&=  \binom{k}{\lfloor \frac{k +1}{2} \rfloor} + \vert \rho_k \vert
\end{align*}
where $\lfloor \frac{k +1}{2} \rfloor $ is the integer part of $(k+1)/2$, i.e, it is equal to $k/2$ if $k$ is even and $(k+1)/2$ otherwise. Hence, for all $j \in I_{\rho_k}$, we have that
\begin{align*}
\frac{\widetilde{m}_j}{\sigma_j^3} & \le   \Big(\binom{k}{\lfloor \frac{k +1}{2} \rfloor} + \vert \rho_k \vert  \Big)  \frac{\E[\vert \nabla^k \mathbf{1}_{\{X_1 = j\}} -  \rho_k \vert^2]}{\sigma_j^3} \\
& =   \Big(\binom{k}{\lfloor \frac{k +1}{2} \rfloor} + \vert \rho_k \vert\Big)  \frac{1}{\sigma_j}. 
\end{align*}
Combining with the inequality in \eqref{eq:minsigma}, we conclude that
$$
\max_{j \in I_{\rho_k}}\frac{\widetilde{m}_j}{\sigma^3_j} \le  \frac{1}{\sqrt{\xi}} \frac{\binom{k}{\lfloor \frac{k +1}{2} \rfloor} + \vert \rho_k \vert}{\Big(\binom{2k}{k} + (k+1) \rho_k^2\Big)^{1/2}}.
$$
Then, the conditions in \eqref{eq:MainCond} are satisfied if $n$ and $\vert \rho_k \vert$ are such that
\begin{align}\label{eq:MainCond2}
& \frac{1}{2\sqrt{\xi}}  \frac{\binom{k}{\lfloor \frac{k +1}{2} \rfloor} + \vert \rho_k \vert}{\left(\binom{2k}{k} + (k+1) \rho_k^2\right)^{1/2}} \frac{1}{\sqrt n} + \frac{\vert \mathcal S_k \vert \sqrt 2}{2 \sqrt{\xi} \sqrt n}  + \frac{16 e^{-1} \vert \mathcal S_k \vert }{\sqrt{2\pi} \xi^{3/2} n^{3/2}} \le \beta/2,  \\
& \sqrt n \vert \rho_k \vert \ge \sqrt \xi \Big(\binom{2k}{k} + (k+1) \rho_k^2\Big)^{1/2}  z_{1-\beta/2}  - t_\alpha. \label{eq:MainCond2.2}
\end{align}
To simplify the first inequality in \eqref{eq:MainCond2}, we can use the fact that $16 e^{-1}/\sqrt{2\pi} \le 5/2$ and $\vert \mathcal S_k \vert \le  \vert \mathcal S_k \vert^3$. Also, we show and then use the fact that
\begin{align*}
\frac{\binom{k}{\lfloor \frac{k +1}{2} \rfloor} + \vert \rho_k \vert}{\Big(\binom{2k}{k} + (k+1) \rho_k^2\Big)^{1/2}} \le \sqrt 2
\end{align*}
which is equivalent to
\begin{align*}
\Big(\binom{k}{\lfloor \frac{k +1}{2} \rfloor} + \vert \rho_k \vert\Big)^2 \le 2\binom{2k}{k} + 2(k+1) \rho_k^2.
\end{align*}
Since $(a+ b)^2 \le 2 (a^2 + b^2)$, it is enough to show that
\begin{align*}
 \binom{k}{\lfloor \frac{k +1}{2} \rfloor}^2 +   \vert \rho_k \vert^2 \le  \binom{2k}{k} + (k+1) \rho_k^2.
\end{align*}
For all $k  \ge 1$, we have that $\vert \rho_k \vert^2 \le  (k +1) \vert \rho_k \vert^2 $. Thus, it is enough to show that 
\begin{align*}
 \binom{k}{\lfloor \frac{k +1}{2} \rfloor}^2 \le  \binom{2k}{k},
\end{align*}
which follows from  
 $ \binom{2k}{k} = \sum_{i=0}^k \binom{k}{i}^2  \ge  \binom{k}{\lfloor \frac{k +1}{2} \rfloor}^2$.  Therefore, \eqref{eq:MainCond2} is satisfied if
 \begin{align*}
 \frac{\sqrt 2}{2 \sqrt{\xi}} \frac{1}{\sqrt n} + \frac{\vert \mathcal S_k \vert \sqrt 2}{2 \sqrt{\xi} \sqrt n}  + \frac{5 \vert \mathcal S_k \vert^3 }{2\xi^{3/2} n^{3/2}} \le \beta/2.    
 \end{align*}
Note that we must have
$$\frac{\vert \mathcal S_k \vert}{\sqrt \xi \sqrt n} \le \frac{\beta}{\sqrt 2} < \frac{1}{\sqrt 2}
$$
and hence
$$\frac{\vert \mathcal S_k \vert^3}{ \xi^{3/2} n^{3/2}} = \Big(\frac{\vert \mathcal S_k \vert}{\sqrt \xi \sqrt n} \Big)^2 \frac{\vert \mathcal S_k \vert}{\sqrt \xi \sqrt n}  \leq  \frac{1}{2} \frac{\vert \mathcal S_k \vert}{\sqrt \xi \sqrt n}.
$$
Therefore, \eqref{eq:MainCond2} is satisfied if 
$$
\Big(\frac{1}{\sqrt 2}  +  \left(\frac{1}{\sqrt 2} +  \frac{5}{4}  \right) \vert \mathcal S_k \vert \Big)   \frac{1}{\sqrt \xi \sqrt n} \le \beta/2.
$$
which is equivalent to 
$$
(4 + (4 + 5 \sqrt 2) \vert \mathcal S_k \vert ) \frac{1}{\sqrt \xi \sqrt n} \le 2 \sqrt 2 \beta,
$$
that is, 
\begin{align*}
n \ge \frac{(4 + (4 + 5 \sqrt 2) \vert \mathcal S_k \vert )^2}{8 \beta^2 \xi}. 
\end{align*}
For the condition \eqref{eq:MainCond2.2}, we will find now a lower bound for $\widehat{t}_{\alpha, n}$.  Note that  $\widehat{t}_{\alpha, n} \ge \underline{t}_{k, \alpha}$ if and only if
\begin{align*}
P(\min_{j \in \mathcal S_k}  \widehat{Z}_j \le \underline{t}_{k, \alpha}) \le \alpha
\end{align*}
where $(\widehat{Z}_j)_{j \in \mathcal S_k} \sim \mathcal{N}(0, \widehat{\Sigma}_{k, n})$.  Let $Z \sim \mathcal N(0,1)$. In the following, we will denote the diagonal terms of $\widehat{\Sigma}_{k, n}$ by $\widehat \sigma_{j}^2, j =1, \ldots, \vert \mathcal S_k \vert$. For $t < 0$, we have that 
\begin{align*}
P(\min_{j \in \mathcal S_k}  \widehat{Z}_j \le t) & \le    \sum_{j \in \mathcal S_k} P(Z_j \le t)  \\
&=  \sum_{j \in \mathcal S_k} P(Z_j/\widehat{\sigma}_j \le t/\widehat{\sigma}_j)  \\
& =  \sum_{j \in \mathcal S_k} P(Z \le t/\widehat{\sigma}_j)  \\
& \le   \vert \mathcal S_k \vert \  \Phi\Big(\frac{t}{\max_{j \in \mathcal S_k} \widehat{\sigma}_j}  \Big)  \le \alpha 
\end{align*}
provided $t \le z_{\frac{\alpha}{\vert \mathcal S_k \vert}}  \max_{j \in \mathcal S_k} \widehat{\sigma}_j $.  This means that we can take $ \underline{t}_{k, \alpha} =  \bar{\sigma} z_{\frac{\alpha}{\vert \mathcal S_k \vert}}$
with $\bar \sigma \ge \max_{j \in \mathcal S_k} \widehat{\sigma}_j.$ Let us denote by $\E_n$ the expectation with respect to the empirical distribution. Then, for any $j \in \mathcal S_k$, we have that
\begin{align*}
\widehat{\sigma}_j^2 & =    \E_n\Big[\Big(\sum_{l=0}^k  (-1)^l \binom{k}{l} \mathbf{1}_{\{X_1 = j+l\}}\Big)^2\Big] - (\nabla^k \widehat p_n(j) )^2   \\
& \le    \E_n\Big[\Big(\sum_{l=0}^k  (-1)^l \binom{k}{l} \mathbf{1}_{\{X_1 = j+l\}}\Big)^2\Big] \\
& =    \sum_{l, s=0}^k (-1)^{l + s}\binom{k}{l}\binom{k}{s} \E_n[ \mathbf{1}_{\{X_1 = j+l\}} \mathbf{1}_{\{X_1 = j+s\}}]\\
& =  \sum_{l=0}^k \binom{k}{l}^2 \widehat p_n(j+l)  \le \binom{2k}{k}.
\end{align*}
Thus, we can take $\underline{t}_{k, \alpha} = \binom{2k}{k}^{1/2} z_{\frac{\alpha}{\vert \mathcal S_k \vert}}$.  Note that 
$ - \underline{t}_{k, \alpha} =  \binom{2k}{k}^{1/2} z_{1- \frac{\alpha}{\vert \mathcal S_k \vert}}$, using the symmetry of $\mathcal N(0,1).$ 
We conclude that the power is at least $1-\beta$ if $n$ and $\rho_k$ satisfy
\begin{align*}
&n \ge \frac{\left(4 + (4 + 5 \sqrt 2) \vert \mathcal S_k \vert \right)^2}{8 \beta^2 \xi} \\
&  \sqrt n \vert \rho_k \vert \ge \sqrt \xi \Big(\binom{2k}{k} + \rho_k^2 (k+1)\Big)^{1/2}  z_{1-\beta/2}  +  \binom{2k}{k}^{1/2} z_{1- \frac{\alpha}{\vert \mathcal S_k \vert}}.
\end{align*}
If $\vert \rho_k \vert = - \rho_k \le r \in (0, \infty)$, then to achieve a power at least $1-\beta$ it is sufficient that 
\begin{align*}
&n \ge\frac{(4 + (4 + 5 \sqrt 2) \vert \mathcal S_k \vert )^2}{8 \beta^2 \xi}, \\
&\sqrt n \vert \rho_k \vert \ge  \sqrt \xi \Big(\binom{2k}{k} + r^2 (k+1)\Big)^{1/2}  z_{1-\beta/2}  +  \binom{2k}{k}^{1/2} z_{1- \frac{\alpha}{\vert \mathcal S_k \vert}},
\end{align*}
which concludes the proof.
\end{proof}

\subsection*{Proof of Corollary \ref{cor:rejection-region-conv-proj}}

\begin{proof}[Proof of Corollary \ref{cor:rejection-region-conv-proj}]
\noindent 1. As $n \to \infty$, Theorem \ref{thm:selection} establishes that the set $\widehat J_{0, n} = \widehat{\mathcal{S}}_{1, n} \setminus \widehat I_{0, n}$ converges to $J_0$, the set of the true knots of $p$ with probability $1-O(1/\sqrt{n})$. Now, we use Lemma 6.1 from \cite{jankowski2009} which ensures that the $\operatorname{gren}$ operator is continuous. Hence, by Lemma \ref{lem:convquant} we have that
\begin{align*}
P(\widehat t^{\mathcal{M}}_{1-\alpha}(\widehat J_{0,n}) \ne t^{\mathcal{M}}_{1-\alpha}(J_0)) = o( 1)  
\end{align*}
where $t^{\mathcal{M}}_{1-\alpha}(J_0)$ is  the $(1-\alpha)$-quantile of the limit r.v.\ $T^\mathcal{M}(J_0)$. Hence, it follows that
\begin{align*}
P(X_{1:n} \in C_{n}^\mathcal{M}(\alpha))  = P(\widehat T^{\mathcal M}_n > t^{\mathcal{M}}_{1-\alpha}(J_0) )   + o(1)  \to  \alpha
\end{align*}
as $n \to \infty$. 

\medskip
\noindent 2. In this case,  Theorem \ref{thm:selection} implies that 
\begin{align*}
P( \widehat{\mathcal{S}}_{1, n} \setminus \widehat I_{0, n}  \ne  \mathcal S_1 )  = O(1/\sqrt n).
\end{align*}
Hence,  $T^{\mathcal M}(J_0)  = 0$ with probability 1 and 
$$
P(\widehat t_{1-\alpha}^\mathcal{M}(\widehat J_{0,n}) \ne 0 )  =  o(1).
$$
Thus,
\begin{align*}
P(X_{1:n} \in C_{n}^\mathcal{M}(\alpha)) & =    P( \sqrt n \Vert \widehat p^{\mathcal M}_n - \widehat p_n \Vert_2 > 0 )  + o(1) \\ 
& =    P( \Vert \widehat p^{\mathcal M}_n - \widehat p_n \Vert_2 > 0 )  + o(1)  \\
& \le  P( \exists \ j \in \mathcal S_1:  \nabla^1 \widehat p_n(j) < 0  )  + o(1) \\
& \le   \sum_{j \in \mathcal S_1} P(\nabla^1 \widehat p_n(j) < 0  )  + o(1) \\
& \le   \sum_{j \in \mathcal S_1} P(\sqrt n (\nabla^1 \widehat p_n(j) - \nabla^1 p(j)) < - \sqrt n  \nabla^1 p(j) ) + o(1) \\
& =  o(1)
\end{align*}
using the same arguments as in the proof of Theorem \ref{thm:selection}. 

\medskip

\noindent 3. Owing to point 1 of Lemma \ref{lemma:proj}, let  $ p^{\mathcal M}$ be the unique $\ell_2$-projection of $p$ on the set of non-increasing p.m.f.s $\mathcal M$.  By point 2 of Lemma \ref{lemma:proj} and the Continuous Mapping Theorem, it follows that 
$$
\Vert \widehat p^{\mathcal M}_n  - \widehat p_n \Vert_2 \overset{P}{\to}  \Vert  p^{\mathcal{M}} - p \Vert_2.
$$
In this case, it holds $\Vert  p^{\mathcal{M}} - p \Vert_2  > 0$ implying that  
\begin{align}\label{eq:div}
\widehat T^{\mathcal M}_n  \overset{P}{\to} \infty.
\end{align}
Moreover, by the same arguments given in point 1, we have that
$$
P(\widehat t_{1-\alpha}^\mathcal{M}(\widehat J_{0,n})  \ne t_{1-\alpha}^\mathcal{M}(J^*)) = o(1),
$$ where $t_{1-\alpha}^\mathcal{M}(J^*)$ is the $(1-\alpha)$-quantile of $T^{\mathcal M}(J^*)$ with 
\begin{align*}
J^* =  \mathcal{S}_1 \setminus (I_- \cup I_0),
\end{align*}
and the set $I_-$ is defined in Theorem \ref{thm:selection}.  Using \eqref{eq:div}, we conclude that
\begin{align*}
P(X_{1:n} \in C_{n}^\mathcal{M}(\alpha))  =    P( \widehat T^{\mathcal M}_n > t_{1-\alpha}^\mathcal{M}(J^*) )  + o(1) \to 1
\end{align*}
as $n \to \infty$.
\end{proof}

\section{Conclusions}
\label{sec:conc}
In this paper, we introduced a unified framework for testing $k$-monotonicity of a p.m.f.\ with unknown finite support. A well-known real-world motivation for this problem is the estimation of species richness.
We proposed an asymptotic test for $k$-monotonicity whose calibration is based on a novel selection method of the (unknown) knot points of the true p.m.f.\ which we proved to be $n^{-1/2}$-consistent under the null hypothesis.  For $k =1$ (resp. $k=2$), we showed how our selection method can be very useful to correctly calibrate the already existing tests based on the Grenander (resp. convex least squares) estimator. Although we found these tests to be quite competitive with the one proposed here, the main approach has the real merit of being much easier to implement numerically and requiring much less computational burden. Moreover, when the true monotonicity parameter is unknown, we developed an estimator for the largest $k\in \mathbb{N}_0$ for which  $j$-monotonicity is not rejected for all $j =0, \ldots, k$. We showed that such an estimator is different from the true value with probability which is asymptotically smaller than the nominal level of the test. Simulations and applications to 5 datasets showcase the performance and the practical usefulness of the proposed methods. The theoretical guarantees are supported by detailed proofs, which we believe have several mathematical aspects that are interesting in their own right.

\section*{Acknowledgements}
We are grateful to Professor Luca Pratelli (Italian Naval Academy, Leghorn) for his insightful discussions and comments, which helped improve the paper.

\appendix
\section*{Appendix}

\section{Additional experiments}

In the following tables we report the result of the experiments presented in the paper replicated for $n=50$. We adopt the same notation and simulation setting.

\begin{table}[ht!]
    \centering
    \caption{Percentage of rejections under $H_0^{(1)}$ for monotonicity tests ($n=50$).}
    \medskip
    \begin{tabular}{lrrrr}
        \toprule
        {Model $\backslash$ Test} & (i) & (ii) & (iii) & (iv) \\
        \midrule
        $\mathcal{P}(0,4,1)$                             & 2.9 & 3.7 & 4.2 & 3.7 \\
        $\mathcal{P}(0,4,2-\sqrt{2})$                    & 0.0 & 0.0 & 0.1 & 0.0 \\
        $\mathcal{P}(0,9,1)$                             & 3.1 & 3.9 & 4.3 & 3.8 \\
        $\mathcal{P}(0,9,2-\sqrt{2})$                    & 0.0 & 0.0 & 0.1 & 0.0 \\
        $\mathcal{MT}(\frac{1}{5},\dots,\frac{1}{5})$    & 0.0 & 0.1 & 0.4 & 0.3 \\
        $\mathcal{MT}(\frac{1}{10},\dots,\frac{1}{10})$  & 0.1 & 0.2 & 0.5 & 0.2 \\
        \bottomrule
    \end{tabular}
    \label{tab:monotonicity0_n50}
\end{table}

\begin{table}[ht!]
    \centering
    \caption{Percentage of rejections under $H_1^{(1)}$ for monotonicity tests ($n=50$).}
    \medskip
    \begin{tabular}{lrrrr}
        \toprule
        {Model $\backslash$ Test} & (i) & (ii) & (iii) & (iv) \\
        \midrule
        $\mathcal{P}(0,4,2)$                     & 24.2 & 24.5 & 25.3 & 30.7 \\
        $\mathcal{P}(0,9,2)$                     & 21.8 & 22.1 & 23.2 & 26.6 \\
        $\mathcal{B}(0,4,4,0.5)$                 & 49.1 & 50.0 & 52.2 & 90.8 \\
        $\mathcal{B}(0,9,4,0.5)$                 & 49.1 & 50.0 & 52.2 & 90.8 \\
        \bottomrule
    \end{tabular}
    \label{tab:monotonicity1_n50}
\end{table}

\begin{table}[ht!]
    \centering
    \caption{Percentage of rejections under $H_0^{(2)}$ for convexity tests ($n=50$).}
    \medskip
    \begin{tabular}{lrrrr}
        \toprule
        {Model $\backslash$ Test} & (i) & (ii) & (iii) & (iv) \\
        \midrule
        $\mathcal{P}(0,4,2-\sqrt{2})$                    & 4.3 & 4.2 & 4.2 & 5.1 \\
        $\mathcal{P}(0,9,2-\sqrt{2})$                    & 4.2 & 4.2 & 4.2 & 5.1 \\
        $\mathcal{MT}(\frac{1}{5},\dots,\frac{1}{5})$    & 0.5 & 1.0 & 1.1 & 2.2 \\
        $\mathcal{MT}(\frac{1}{10},\dots,\frac{1}{10})$  & 2.1 & 2.3 & 2.6 & 2.4 \\
        \bottomrule
    \end{tabular}
    \label{tab:convexity0_n50}
\end{table}

\begin{table}[ht!]
    \centering
    \caption{Percentage of rejections under $H_1^{(2)}$ for convexity tests ($n=50$).}
    \medskip
    \begin{tabular}{lrrrr}
        \toprule
        {Model $\backslash$ Test} & (i) & (ii) & (iii) & (iv) \\
        \midrule
        $\mathcal{P}(0,4,2)$                        & 19.4 & 19.2 & 19.3 & 90.4 \\
        $\mathcal{P}(0,4,1)$                        & 18.4 & 18.6 & 18.7 & 22.7 \\
        $\mathcal{P}(0,9,2)$                        & 19.0 & 19.1 & 19.3 & 71.8 \\
        $\mathcal{P}(0,9,1)$                        & 18.5 & 18.6 & 18.7 & 22.2 \\
        $\mathcal{B}(0,4,4,0.5)$                    & 30.2 & 30.1 & 30.1 & 100.0 \\
        $\mathcal{B}(0,9,4,0.5)$                    & 30.2 & 30.1 & 30.1 & 100.0 \\
        \bottomrule
    \end{tabular}
    \label{tab:convexity1_n50}
\end{table}

\section{Asymptotic covariance}
In the following propositions, we report some basic asymptotic results with details on the asymptotic covariance $\Sigma_k$ appearing in Theorem \ref{thm:main} for the cases $k=1,2$.
\begin{proposition}
\label{prop:clt_monotonicity}
As $n\to \infty$, 
$$\sqrt n(\nabla^1\widehat p _n (j)-\nabla^1p(j))_{j \in \mathcal{S}_1}\overset{d}{\to}Z\sim \mathcal{N}(0,\Sigma_1),$$  where the asymptotic covariance matrix $\Sigma_1$ is such that 
\begin{align*}
    (\Sigma_1)_{r,s}&=-\nabla^1p(m-1+r)\nabla^1p(m-1+s), \quad |r-s|\geq2,\\
  (\Sigma_1)_{r,r+1}&=-p(m+r)-\nabla^1p(m-1+r) \nabla^1p(m+r),\\
    (\Sigma_1)_{r,r}&=p(m+r)+p(m-1+r)-(\nabla^1p(m-1+r))^2
\end{align*}
for $r,s=1,\ldots,M-m$.
\end{proposition}

\begin{proof}[Proof of Proposition \ref{prop:clt_monotonicity}]
Apply the Central Limit Theorem and note that for every $r,s=1,\dots,M-m$ it holds
\begin{multline*}
         (\Sigma_1)_{r,s}=\E[(\mathbf{1}_{\{X_1=m-1+r\}}-\mathbf{1}_{\{X_1=m+r\}})(\mathbf{1}_{\{X_1=m-1+s\}}-\mathbf{1}_{\{X_1=m+s\}})]\\
      - \nabla^1p(m-1+r)\nabla^1p(m-1+s).
\end{multline*}
Therefore, we have that $(\Sigma_1)_{r,s}=-\nabla^1p(m-1+r)\nabla^1p(m-1+s)$ when $|r-s|\geq2$ since all the cross products are equal to 0,
$(\Sigma_1)_{r,r+1}=-p(m+r)-\nabla^1p(m-1+r) \nabla^1p(m+r)$ since the products of indicator functions are non-zero only for $X_1=m+r$, and similarly
$(\Sigma_1)_{r,r}=p(m+r)+p(m-1+r)-(\nabla^1p(m-1+r))^2$.
\end{proof}

\begin{proposition}
\label{prop:clt_convexity}
As $n \to \infty$,
$$\sqrt n(\nabla^2\widehat p _n (j)-\nabla^2p(j))_{j \in \mathcal{S}_2}\overset{d}{\to}Z \sim \mathcal{N}(0,\Sigma_2),$$ where the asymptotic covariance matrix $\Sigma_2$ is such that 
\begin{align*}
    (\Sigma_2)_{r,s}&=  -\nabla^2p(m-1+ r)\nabla^2p(m-1+s), \quad |r-s|>2,\\
    (\Sigma_2)_{r, r+2}&=  p(m+ r+1)-\nabla^2p(m-1+ r)\nabla^2p(m+r+1),\\
 (\Sigma_2)_{r,r+1}&=-2(p(m+r+1)+p(m + r))-\nabla^2p(m-1+r)\nabla^2p(m +r),\\
    (\Sigma_2)_{r,r}&=p(m + r+1)+4p(m+r)+p(m-1+r)-(\nabla^2p(m-1+r))^2,
\end{align*}
for $r,s=1,\ldots,M-m-1$.
\end{proposition}

\begin{proof}[Proof of Proposition \ref{prop:clt_convexity}]
Apply the Central Limit Theorem and note that for every $r,s=1,\dots,M-m-1$ it holds
     \begin{multline*}
(\Sigma_2)_{r,s}=\E[(\mathbf{1}_{\{X_1=m+r+1\}}-2\mathbf{1}_{\{X_1=m+r\}}+ \mathbf{1}_{\{X_1=m-1+r\}})\\
\times(\mathbf{1}_{\{X_1=m+s+1\}}-2\mathbf{1}_{\{X_1=m+s\}}+ \mathbf{1}_{\{X_1=m-1+s\}})] -\nabla^2p(m-1+r)\nabla^2p(m-1+s).
     \end{multline*}
Therefore, we have that $(\Sigma_2)_{r,s}=-\nabla^2p(m-1+ r)\nabla^2p(m-1+s)$ when $|r-s|>2$ since all the cross products are equal to 0,
$(\Sigma_2)_{r,r+2}= p(m+ r+1)-\nabla^2p(m-1+ r)\nabla^2p(m+r+1)$, since the products of indicator functions are non-zero only for $X_1=m+r+1$, and similarly $(\Sigma_2)_{r,r+1}=-2(p(m+r+1)+p(m + r))-\nabla^2p(m-1+r)\nabla^2p(m +r),$ $
(\Sigma_2)_{r,r}=p(m + r+1)+4p(m+r)+p(m-1+r)-(\nabla^2p(m-1+r))^2$.
\end{proof}

\section{Bootstrap}
Here, the aim is to show that Bootstrap resampling does not provide appropriate approximations to the limit distribution of the sample statistics $T_{k,n}$ introduced in Section \ref{sec:asymp} and established in Theorem \ref{thm:main}. With the notation of the paper, let $X^*_{1:n} $ be a bootstrap sample drawn from the empirical distribution of $X_{1:n}$, that is, $X^*_{1:n} $ are independent copies of a r.v.\ with p.m.f.\ $\widehat p_n$.  We denote by $\widehat p^*_n$ the empirical p.m.f.\ based on $X^*_{1:n}$.

\begin{proposition}
\label{prop:boostrap}
Suppose $I_{\rho_k} = \{ j \in \mathcal S_k:  \nabla^k p(j) = \rho_k \}$ is a singleton. Then, as $n\to\infty$, the conditional law given $X_{1:n}$ of
$$
\sqrt n (\min_{j \in \widehat{\mathcal{S}}_{k,n}} \nabla^k \widehat p^*_n(j) - \min_{j \in \widehat{\mathcal{S}}_{k,n}} \nabla^k \widehat p_n(j))
$$
converges weakly, in probability, to the law of $W_{I_{\rho_k}}$. 

Suppose $|I_{\rho_k}|\geq 2$. Then, as $n\to\infty$, the conditional law given $X_{1:n}$ of
$$
\sqrt n (\min_{j \in \widehat{\mathcal{S}}_{k,n}} \nabla^k \widehat p^*_n(j) - \min_{j \in \widehat{\mathcal{S}}_{k,n}} \nabla^k \widehat p_n(j))
$$
is stochastically larger than the law of $W_{I_{\rho_k}}$.
\end{proposition}

\begin{proof}[Proof of Proposition \ref{prop:boostrap}]
In the following, we set $\widehat {\mathcal{S}}_{k,n} = \mathcal{S}_k$ as this event occurs with probability 1 for $n$ large enough.
We have that 
\begin{align}
\label{eq:first_equality}
\begin{split}
\min_{j \in \mathcal{S}_k} \nabla^k \widehat p^*_n(j) &=   \min_{j \in \mathcal{S}_k} \Big(\nabla^k \widehat p^*_n(j) - \nabla^k \widehat p_n(j)  +  \nabla^k \widehat p_n(j) \Big) \\
& =  \min_{j \in I_{\rho_k}} \Big(\nabla^k \widehat p^*_n(j) - \nabla^k \widehat p_n(j)  +  \nabla^k \widehat p_n(j) \Big)  \\
&\quad \wedge \min_{j \in I^c_{\rho_k}} \Big(\nabla^k \widehat p^*_n(j) - \nabla^k \widehat p_n(j)  +  \nabla^k \widehat p_n(j) \Big). 
\end{split}
\end{align}
Now, conditionally on $X_{1:n}$ the Central Limit Theorem implies that
\begin{align*}
\sqrt n(\nabla^k \widehat p^*_n(j) - \nabla^k \widehat p_n(j) )_{j \in \mathcal{S}_k} \overset{d}{\to} (\widehat Z_j)_{j \in \mathcal{S}_k}  \sim \mathcal{N}(0, \widehat{\Sigma}_{k,n}) 
\end{align*}
where $\widehat{\Sigma}_{k,n}$ is the empirical estimator of the true covariance matrix $\Sigma_k$ introduced in Theorem \ref{thm:main}. Now, suppose that $I_{\rho_k}^c \ne \emptyset$. Then, for $j \in I_{\rho_k}$ and $i \in I_{\rho_k}^c$ we have that with probability tending to 1,
\begin{align*}
&\nabla^k \widehat p^*_n(i) - \nabla^k \widehat p^*_n(j) \\
 &=\nabla^k \widehat p^*_n(i) - \nabla^k \widehat p_n(i)  +  \nabla^k \widehat p_n(i)   -  (\nabla^k \widehat p^*_n(j) - \nabla^k \widehat p_n(j)  +  \nabla^k \widehat p_n(j))\\
 & = \nabla^k \widehat p_n(i)  -  \nabla^k \widehat p_n(j)  +  o_P(1) >  0
\end{align*}
as a consequence of the consistency of the empirical estimator $\widehat p_n$.  Thus, thanks to \eqref{eq:first_equality} we have that
\begin{align*}
\min_{j \in \mathcal{S}_k} \nabla^k \widehat p^*_n(j) =  \min_{j \in I_{\rho_k}} \nabla^k \widehat p^*_n(j)
\end{align*}
with probability tending to 1, and the same argument implies that, with probability tending to 1, 
$$
\min_{j \in \mathcal{S}_k} \nabla^k \widehat{p}_n(j) =  \min_{j \in I_{\rho_k}} \nabla^k \widehat{p}_n(j).
$$
Thus, we conclude that
\begin{align*}
 \min_{j \in \mathcal{S}_k}   \nabla^k \widehat p^*_n(j) - \min_{j \in \mathcal{S}_k} \nabla^k \widehat p_n(j) & =    \min_{j \in I_{\rho_k}}   \nabla^k \widehat p^*_n(j) - \min_{j \in I_{\rho_k}} \nabla^k \widehat p_n(j).\\
 & =   \min_{j \in I_{\rho_k}}  \Big(  \nabla^k \widehat p^*_n(j)  -  \nabla^k \widehat p_n(j)  +  \nabla^k \widehat p_n(j) \Big)- \min_{j \in I_{\rho_k}} \nabla^k \widehat p_n(j),
\end{align*}
with probability tending to 1.

Suppose that $I_{\rho_k} = \{i_1\}$ then, as $n\to\infty$,
\begin{align*}
 \sqrt n (\min_{j \in \mathcal{S}_k}   \nabla^k \widehat p^*_n(j) - \min_{j \in \mathcal{S}_k} \nabla^k \widehat p_n(j) )\overset{P}{=}    \sqrt n (\nabla^k \widehat p^*_n(i_1)  -  \nabla^k \widehat p_n(i_1))   \overset{d}{\to} \widehat{Z}_{i_1,n} \sim \mathcal{N}(0, \widehat{\Sigma}_{k,n})  
\end{align*}
conditionally on $X_{1:n}$, where $\widehat{\Sigma}_{k,n} \overset{P}{\to} {\Sigma}_{k} \in \mathbb{R_+} $. Thus, conditionally on $X_{1:n}$, as $n\to \infty$, it holds
$$
\widehat{Z}_{i_1,n} \overset{d}{\to} Z_{i_1} = W_{I_{\rho_k}} \sim \mathcal{N}(0,\Sigma_k),
$$
and the first claim follows.

Suppose now that $|I_{\rho_k}|\geq 2$. Using the fact that $\min _i(a_i + b_i)  \ge \min_i a_i +  \min_i b_i$, with probability tending to 1,
\begin{align*}
&\sqrt n (\min_{j \in \mathcal{S}_k}   \nabla^k \widehat p^*_n(j) - \min_{j \in \mathcal{S}_k} \nabla^k \widehat p_n(j) )\\
&=\sqrt n \Big(\min_{j \in I_{\rho_k}}  \Big(  \nabla^k \widehat p^*_n(j)  -  \nabla^k \widehat p_n(j)  +  \nabla^k \widehat p_n(j) \Big)- \min_{j \in I_{\rho_k}} \nabla^k \widehat p_n(j)  \Big)  \\
& \ge \sqrt n  \min_{j \in I_{\rho_k}}  \Big(  \nabla^k \widehat p^*_n(j)  -  \nabla^k \widehat p_n(j)\Big)  + \sqrt n \min_{j \in I_{\rho_k}} \nabla^k \widehat p_n(j)  -  \sqrt n \min_{j \in I_{\rho_k}} \nabla^k \widehat p_n(j) \\
& =   \min_{j \in I_{\rho_k}} \sqrt n  (  \nabla^k \widehat p^*_n(j)  -  \nabla^k \widehat p_n(j)).
\end{align*}
Thus, as $n\to\infty$,
\begin{align*}
\sqrt n ( \min_{j \in \mathcal{S}_k}   \nabla^k \widehat p^*_n(j) - \min_{j \in \mathcal{S}_k} \nabla^k \widehat p_n(j) ) \overset{P}{\ge}   \min_{j \in I_{\rho_k}} \sqrt n  (  \nabla^k \widehat p^*_n(j)  -  \nabla^k \widehat p_n(j))   \overset{d}{\to} \min_{j \in I_{\rho_k}} \widehat{Z}_{j,n}
\end{align*}
conditionally on $X_{1:n}$, where $(\widehat Z_{j,n})_{j \in \mathcal{S}_k}\sim\mathcal{N}(0, \widehat{\Sigma}_{k,n})$ and $\widehat{\Sigma}_{k,n}\overset{P}{\to}\Sigma_k$.  It follows that, conditionally on $X_{1:n}$, as $n\to\infty$,
$$
\min_{j \in I_{\rho_k}} \widehat{Z}_{j,n} \overset{d}{\to}  W_{\rho_k}=\min_{j \in I_{\rho_k}} Z_j,
$$
and the second claim follows.
\end{proof}

\bibliographystyle{apalike}
\bibliography{bib-kmon}

\end{document}